\documentclass[pdflatex,sn-mathphys-num]{sn-jnl}

\usepackage{graphicx}%
\usepackage{multirow}%
\usepackage{amsmath,amssymb,amsfonts}%
\usepackage{amsthm}%
\usepackage{mathrsfs}%
\usepackage[title]{appendix}%
\usepackage{xcolor}%
\usepackage{textcomp}%
\usepackage{manyfoot}%
\usepackage{booktabs}%
\usepackage{algorithm}%
\usepackage{algorithmicx}%
\usepackage{algpseudocode}%
\usepackage{listings}%
\usepackage{placeins}%
\usepackage{tikz-cd}%
\usetikzlibrary{patterns}%

\usepackage{pifont}
\newcommand{\cmark}{\ding{51}}
\newcommand{\xmark}{\ding{55}}

\theoremstyle{plain}
\newtheorem{theorem}{Theorem}[section]
\newtheorem*{theorem*}{Theorem}
\newtheorem{proposition}[theorem]{Proposition}
\newtheorem{lemma}[theorem]{Lemma}
\newtheorem{corollary}[theorem]{Corollary}
\newtheorem*{corollary*}{Corollary}

\theoremstyle{remark}

\newtheorem{conjecture}[theorem]{Conjecture}
\newtheorem*{conjecture*}{Conjecture}

\theoremstyle{definition}
\newtheorem{definition}[theorem]{Definition}

\numberwithin{equation}{section}

\renewcommand{\orcidlogo}{%
  \includegraphics[width=10pt]{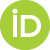}%
}

\begin{document}

\title[On the Spectral Synthesis of Lipschitz Vectorizations of Persistence Diagrams]{On the Spectral Synthesis of Lipschitz Vectorizations of Persistence Diagrams}

\author*{\fnm{Charles} \sur{Fanning} \orcid{https://orcid.org/0009-0002-8251-2802}}\email{cfannin8@students.kennesaw.edu}

\author{\fnm{Mehmet} \sur{Aktas} \orcid{https://orcid.org/0000-0002-9527-9600}}\email{maktas1@kennesaw.edu}

\affil{\orgdiv{School of Data Science and Analytics}, 
\orgname{Kennesaw State University}, 
\orgaddress{\street{1000 Chastain Rd NW}, 
\city{Kennesaw}, 
\postcode{30144}, 
\state{Georgia}, 
\country{United States}}}







\abstract{
Persistence diagrams are fundamental descriptors in topological data analysis. Many statistical and machine learning methods require persistence diagrams to be mapped into vector spaces or compared through kernels. Although many persistence diagram vectorizations and persistence kernels have been proposed, there is no general framework for determining which persistence diagram vectorizations are synthesizable. Spectral synthesis asks whether the exponential monomials contained in the translation-invariant variety generated by a function span a dense subspace of that variety. We introduce uniformly discrete spectral synthesis for persistence diagram vectorizations using Fourier analysis on groups of virtual persistence diagrams. When the space of persistence intervals is separable, we prove sufficient conditions for uniformly discrete spectral synthesis and derive approximation results for virtual persistence diagrams with bounded total multiplicity. We then classify many persistence diagram vectorizations and persistence kernels from the literature by uniformly discrete spectral synthesis.
}



\maketitle


\section{Introduction}\label{sec1}


Persistent homology tracks how topological features appear and disappear across a filtration as the filtration parameter varies \cite{892133,Zomorodian2005ComputingPH,CohenSteiner2007StabilityPD,Oudot2015PersistenceT}. Each feature determines a persistence interval whose endpoints record its birth and death. A persistence diagram is a finite multiset of persistence intervals. For a metric pair \((X,d,A)\), we write \(D(X,A)\) for the free commutative monoid of persistence diagrams and \(K(X,A)\) for its Grothendieck completion, the free abelian group of virtual persistence diagrams \cite[Sections~2.3--2.5]{bubenik2022virtual}. 

Persistence diagrams have become standard descriptors in topological data analysis and have found applications in statistics and machine learning \cite{bubenik2015statistical,adams2017persistence,Carrire2019PersLayAN}. Many statistical and machine learning methods cannot use persistence diagrams directly, so persistence diagrams must first be mapped into vector spaces or compared through kernels. This need has led to many persistence diagram vectorizations and persistence kernels \cite{ali2023survey,hensel2021survey,pun2018persistent}, including Betti curves \cite{berry2018functional}, persistence landscapes \cite{bubenik2015statistical}, persistence images \cite{adams2017persistence}, PersLay \cite{Carrire2019PersLayAN}, and the sliced Wasserstein kernel \cite{carriere2017sliced}.

The objective of this paper is to determine which persistence diagram vectorizations have uniformly discrete spectral synthesis (see Definition~\ref{def:vectorization-spectral-synthesis}) and how such vectorizations can be approximated by exponential monomials.

Our approach uses harmonic analysis on groups of virtual persistence diagrams. For a uniformly discrete subpair \(A\subseteq Z\subseteq X\), the group \(K(Z,A)\) is discrete and locally compact \cite[Corollary~3.3]{fanningaktas2026banachrkhs}, so we can use Fourier--Stieltjes transforms to analyze functions and measures on \(K(Z,A)\).


We first give sufficient conditions for uniformly discrete spectral synthesis in Theorem~\ref{thm:uniformly-discrete-fourier-stieltjes-consistency}. Here, \(i_Z\) denotes the homomorphism induced by inclusion, \(P_n^Z\) denotes the homomorphism from Lemma~\ref{lem:approximating-homomorphisms}, and \(V_Z(f)\) is the variety generated by the bounded linear scalarizations of \(f\circ i_Z\).

\begin{theorem*}
Suppose that \(X/A\) is separable, let \(E\) be a complex Banach space, and let \(f\colon K(X,A)\to E\) be continuous. Suppose that, on each nested uniformly discrete subpair \(Y_n\), every scalarization of \(f\) is the sum of an additive function and a Fourier--Stieltjes transform,
\[
\lambda(f(i_{Y_n}\eta))
=
a_{n,\lambda}(\eta)
+
\int_{\widehat{K(Y_n,A)}}
\chi(\eta)\,d\mu_{n,\lambda}(\chi),
\qquad
\eta\in K(Y_n,A),
\]
and that the representing measures pushed forward by the dual homomorphisms \(\widehat{P_n^Z}\) have uniformly bounded total variation for every uniformly discrete subpair \(A\subseteq Z\subseteq X\) and every \(\lambda\in E^*\). Then the following statements hold.
\begin{enumerate}
\item The additive functions and representing measures are compatible across the nested subpairs \(Y_n\) and under restriction to uniformly discrete subpairs.

\item On every uniformly discrete subpair \(A\subseteq Z\subseteq X\), each scalarization of \(f\circ i_Z\) can be written uniquely as the sum of an additive function and a Fourier--Stieltjes transform, and the approximating additive functions and representing measures converge to its two terms.

\item The additive functions and measures obtained in the limit are compatible under inclusions of uniformly discrete subpairs.

\item The additive functions and measures depend linearly on the scalarization, and the representing measures satisfy uniform total variation bounds.

\item Every closed support constraint satisfied by the approximating representing measures is preserved by the limiting measure.

\item The limiting additive functions and the characters in the supports of the limiting representing measures span a pointwise dense subspace of \(V_Z(f)\); consequently, \(f\) has uniformly discrete spectral synthesis.
\end{enumerate}
\end{theorem*}

That is, if each scalarization of \(f\) on the nested groups \(K(Y_n,A)\) can be written as an additive function plus the Fourier--Stieltjes transform of a measure, and the representing measures satisfy the stated total variation bounds, then \(f\) has uniformly discrete spectral synthesis.

The main idea behind the proof of Theorem~\ref{thm:uniformly-discrete-fourier-stieltjes-consistency} is that the representations on the nested subpairs \(Y_n\) converge to representations on every uniformly discrete subpair \(Z\) as the sum of an additive function and a Fourier--Stieltjes transform. The function on \(Z\) then has uniformly discrete spectral synthesis by Lemma~\ref{lem:additive-fourier-stieltjes-synthesis}.


We next show in Corollaries~\ref{cor:finite-quantization-approximation}, \ref{cor:quantitative-spectral-approximation}, and~\ref{cor:covering-number-evaluation-bound} error bounds for approximating Lipschitz vectorizations restricted to virtual persistence diagrams with bounded total multiplicity, as well as bounds on the number of evaluations required to compute the approximation.

\begin{corollary*}
Assume the hypotheses of Theorem~\ref{thm:uniformly-discrete-fourier-stieltjes-consistency}, and suppose that \(f\) is Lipschitz. Let \(C\subseteq X/A\) be totally bounded and let \(R\) be a positive integer. Then the following statements hold.
\begin{enumerate}
\item Corollary~\ref{cor:finite-quantization-approximation}: For every \(\varepsilon>0\), there are a homomorphism \(P\) induced by a finite \(\varepsilon\)-net and functions \(\psi_1,\ldots,\psi_d\) such that
\[
\sup_{\xi\in\mathcal K_R(C)}
\left\|
f(i_{Z_C}\xi)
-
\sum_{j=1}^d v_j\psi_j(P\xi)
\right\|_E
\leq
\operatorname{Lip}(f)R\varepsilon.
\]

\item If, in addition, \(\sup_{n\geq1}\|\mu_{n,\lambda}\|_{\mathrm{TV}}<\infty\) for every \(\lambda\in E^*\), the approximating additive functions vanish, and Equation~\eqref{eq:quantization-spectral-bound} holds, then the functions \(\psi_j\) may be chosen as characters, and the approximation satisfies
\[
\sup_{\xi\in\mathcal K_R(C)}
\left\|
f(i_{Z_C}\xi)
-
\sum_{j=1}^d v_j\chi_j(P\xi)
\right\|_E
\leq
\min\{
\operatorname{Lip}(f)R\varepsilon,
M\omega(R\varepsilon)
\}.
\]

\item Corollary~\ref{cor:covering-number-evaluation-bound}: After the functions \(\psi_j\) have been chosen, the number \(d_\varepsilon\) of evaluations required to compute the approximation satisfies
\[
d_\varepsilon
\leq
\sum_{r=0}^{\min\{N(C,\varepsilon),R\}}
2^r
\binom{N(C,\varepsilon)}r
\binom Rr.
\]
If \(N(C,\varepsilon)=O(\varepsilon^{-s})\) as \(\varepsilon\downarrow0\), then \(d_\varepsilon=O(\varepsilon^{-sR})\) for fixed \(R\).
\end{enumerate}
\end{corollary*}

That is, for virtual persistence diagrams with bounded total multiplicity, the values of the vectorization on the finite set of quantized virtual persistence diagrams determine an approximation of its values on all of \(\mathcal K_R(C)\), and Lipschitz continuity controls the error away from that finite set.

The main idea behind the proofs of Corollaries~\ref{cor:finite-quantization-approximation}, \ref{cor:quantitative-spectral-approximation}, and~\ref{cor:covering-number-evaluation-bound} is that the reconstruction is exact on the finite set of quantized virtual persistence diagrams. The approximation error is then bounded by how much the vectorization can change with respect to the corresponding approximation. 

Since the approximation is finite and the total multiplicity is bounded, we can count the possible quantized virtual persistence diagrams. This count bounds the number of evaluations needed to compute the approximation.


We then classify persistence diagram vectorizations and persistence kernels according to whether they admit a Lipschitz extension to \(K(X,A)\) and whether they have uniformly discrete spectral synthesis. Table~\ref{tab:persistence-vectorizations} summarizes these classifications for the persistence diagram vectorizations. Section~\ref{sec:examples} gives the corresponding classifications for persistence codebooks and ATOL in Table~\ref{tab:persistence-codebooks}, PersLay in Table~\ref{tab:perslay-vectorizations}, and persistence kernels in Table~\ref{tab:persistence-kernels}. In every example we classify, if a Lipschitz extension exists, then there is a Lipschitz extension with uniformly discrete spectral synthesis. This pattern motivates the following conjecture.

\begin{conjecture*}
Let \((X,d,A)\) be a metric pair with \(X/A\) separable, let \(E\) be a complex Banach space, and let \(f\colon K(X,A)\to E\) be Lipschitz. Then \(f\) has uniformly discrete spectral synthesis in the sense of Definition~\ref{def:vectorization-spectral-synthesis}.
\end{conjecture*}

\begin{table}[htbp]
\centering
\small
\setlength{\tabcolsep}{4pt}
\begin{tabular}{llll}
\hline
Vectorization
&
Codomain
&
\begin{tabular}{c}
Lipschitz\\
Extension
\end{tabular}
&
\begin{tabular}{c}
Spectral\\
Synthesis
\end{tabular}
\\
\hline

Lipschitz-free embedding\textsuperscript{a}
&
\(\mathcal F(K(X,A))\)
&
\cmark
&
\cmark
\\

Additive functions\textsuperscript{b}
&
\(\mathbb C\)
&
\cmark
&
\cmark
\\

Characters\textsuperscript{b}
&
\(\mathbb C\)
&
\cmark
&
\cmark
\\

Persistence statistics\textsuperscript{c}
&
\(\mathbb R^{38}\)
&
\xmark
&
\xmark
\\

Persistent entropy\textsuperscript{d}
&
\(\mathbb R\)
&
\xmark
&
\xmark
\\

Entropy summary function, \(p=1\)\textsuperscript{e}
&
\(L^1(\mathbb R)\)
&
\cmark
&
\cmark
\\

\begin{tabular}[c]{@{}l@{}}
Entropy summary function,\\
\(1<p\leq\infty\)\textsuperscript{e}
\end{tabular}
&
\(L^p(\mathbb R)\)
&
\xmark
&
\xmark
\\

Betti curve\textsuperscript{f}
&
\(L^1(\mathbb R)\)
&
\cmark
&
\cmark
\\

\begin{tabular}[c]{@{}l@{}}
Lifespan curve,\\
\(1\leq p\leq\infty\)\textsuperscript{g}
\end{tabular}
&
\(L^p(\mathbb R)\)
&
\xmark
&
\xmark
\\

\begin{tabular}[c]{@{}l@{}}
Persistence landscape,\\
\(1\leq p<\infty\)\textsuperscript{h}
\end{tabular}
&
\(L^p(\mathbb N\times\mathbb R)\)
&
\xmark
&
\xmark
\\

Persistence landscape, \(p=\infty\)\textsuperscript{h}
&
\(L^\infty(\mathbb N\times\mathbb R)\)
&
\cmark
&
\cmark
\\

\begin{tabular}[c]{@{}l@{}}
Persistence silhouette,\\
\(1\leq p<\infty\)\textsuperscript{i}
\end{tabular}
&
\(L^p(\mathbb R)\)
&
\xmark
&
\xmark
\\

Persistence silhouette, \(p=\infty\)\textsuperscript{i}
&
\(L^\infty(\mathbb R)\)
&
\cmark
&
\cmark
\\

Persistence image\textsuperscript{j}
&
\(\mathbb R^N\)
&
\cmark
&
\cmark
\\

Template functions\textsuperscript{k}
&
\(\mathbb R^n\)
&
\cmark
&
\cmark
\\

Tropical coordinates\textsuperscript{l}
&
\(\mathbb R^7\)
&
\xmark
&
\xmark
\\

Schauder basis\textsuperscript{m}
&
\(\ell^1\)
&
\cmark
&
\cmark
\\

\begin{tabular}[c]{@{}l@{}}
Truncated signature of an\\
integrated persistence landscape, \(N,M\geq1\)\textsuperscript{n}
\end{tabular}
&
\(\displaystyle\bigoplus_{m=0}^{M}(\mathbb C^N)^{\otimes m}\)
&
\xmark
&
\xmark
\\

\begin{tabular}[c]{@{}l@{}}
Linear functionals of truncated signatures\\
of integrated persistence landscapes
\textsuperscript{n,\(\dagger\)}
\end{tabular}
&
\begin{tabular}[c]{@{}l@{}}
finite-dimensional complex\\
normed space
\end{tabular}
&
\cmark
&
\cmark
\\

\hline
\end{tabular}

\begin{minipage}{\linewidth}
\footnotesize
\textsuperscript{a}\cite{bubenik2022virtual}.
\textsuperscript{b}\cite{szekelyhidi2015spectral}.
\textsuperscript{c}\cite[Definition~3.1]{ali2023survey}, \cite{asaad2022persistent,chung2022persistence,pun2018persistent}.
\textsuperscript{d}\cite{chintakunta2015entropy,atienza2020stability}.
\textsuperscript{e}\cite{atienza2020stability}.
\textsuperscript{f}\cite{berry2018functional}.
\textsuperscript{g}\cite{asaad2022persistent}.
\textsuperscript{h}\cite{bubenik2015statistical}.
\textsuperscript{i}\cite{chazal2014stochastic}.
\textsuperscript{j}\cite{adams2017persistence}.
\textsuperscript{k}\cite{perea2023approximating}.
\textsuperscript{l}\cite[Definition~3.4 and the subsequent seven coordinates]{ali2023survey}, \cite[Section~8]{kalivsnik2019tropical}.
\textsuperscript{m}\cite[Definition~5.1 and Theorem~5.2]{bubenik2025schauder}.
\textsuperscript{n}\cite{chevyrev2018persistence,giusti2023signatures}.

\textsuperscript{\(\dagger\)}
Among the functions obtained by applying linear functionals to truncated signatures of integrated persistence landscapes, only the constant functions are Lipschitz with respect to \(W_1\).
\end{minipage}

\caption{Persistence diagram vectorizations classified by Lipschitz extension and uniformly discrete spectral synthesis. Table~\ref{tab:persistence-vectorizations} gives additional hypotheses.}
\end{table}

The classifications of the persistence diagram vectorizations in Table~\ref{tab:persistence-vectorizations} are proved in Appendix~\ref{sec:vectorization-synthesis-derivations}. The classifications of the learnable vectorizations in Section~\ref{sec:learnable-vectorizations}, given by Tables~\ref{tab:persistence-codebooks} and~\ref{tab:perslay-vectorizations}, are proved in Appendix~\ref{sec:learnable-vectorization-synthesis-derivations}. The classifications of the persistence kernels in Table~\ref{tab:persistence-kernels} are proved in Appendix~\ref{sec:kernel-synthesis-derivations}.

We make the following assumptions throughout this paper. We assume that \(A\neq\varnothing\). We also assume that the \(1\)-strengthened distance \(d_1\) from Equation~\eqref{eq:one-strengthening} is finite and separated on \(X/A\). We also assume that \(X/A\) is separable in Theorem~\ref{thm:uniformly-discrete-fourier-stieltjes-consistency} and in the approximation results.

We leave for future work the open question of finding necessary and sufficient conditions for Lipschitz persistence diagram vectorizations to have uniformly discrete spectral synthesis.

The paper is organized as follows. Section~\ref{sec:background} introduces the background and notation used throughout the paper. Section~\ref{sec:uniformly-discrete-synthesis} proves the sufficient conditions for uniformly discrete spectral synthesis and the approximation theorems for persistence diagram vectorizations. Section~\ref{sec:examples} gives the classification results summarized in Tables~\ref{tab:persistence-vectorizations}, \ref{tab:persistence-codebooks}, \ref{tab:perslay-vectorizations}, and~\ref{tab:persistence-kernels}, with the learnable vectorizations treated in Section~\ref{sec:learnable-vectorizations}. Appendix~\ref{sec:vectorization-synthesis-derivations} proves the classifications in Table~\ref{tab:persistence-vectorizations}, Appendix~\ref{sec:learnable-vectorization-synthesis-derivations} proves the classifications in Section~\ref{sec:learnable-vectorizations}, and Appendix~\ref{sec:kernel-synthesis-derivations} proves the classifications in Table~\ref{tab:persistence-kernels}.

\section{Background and Notation}\label{sec:background}

Given a filtration \(K^0\subseteq K^1\subseteq\cdots\subseteq K^m\) of a simplicial complex, persistent homology tracks the birth and death of homology classes in \(H_k(K^i)\) along the filtration \cite{892133, Zomorodian2005ComputingPH, Oudot2015PersistenceT}. The resulting persistence intervals determine the persistence diagram associated with the filtration \cite{Oudot2015PersistenceT}.

A matching pairs points between two persistence diagrams and may also pair points with the diagonal \cite{CohenSteiner2007StabilityPD,Oudot2015PersistenceT}. The bottleneck distance is the infimum over all such matchings of the largest \(\ell^\infty\) distance between matched points. The \(p\)-Wasserstein distance is the infimum over all such matchings of the corresponding \(p\)-norm of the distances between matched points \cite{CohenSteiner2007StabilityPD,Oudot2015PersistenceT}.

If \(X\) is triangulable and \(f,g\colon X\to\mathbb R\) are continuous tame functions, then the bottleneck distance between the persistence diagrams obtained from computing persistent homology on the sublevel-set filtrations of \(f\) and \(g\) is at most \(\|f-g\|_\infty\) \cite{CohenSteiner2007StabilityPD,Oudot2015PersistenceT}.

\begin{figure}[t]
    \centering
    \includegraphics[width=\textwidth]{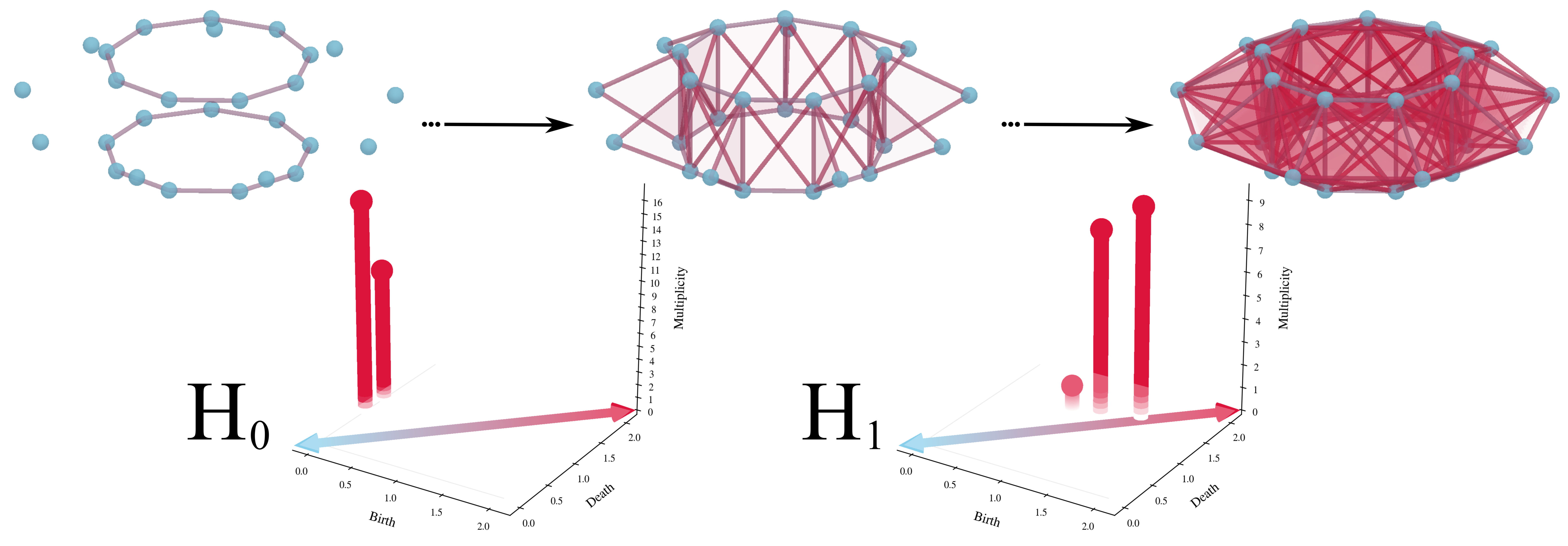}

    \caption{
    The upper panels show a filtration of a simplicial complex. The lower panels show the corresponding persistence diagrams for connected components (left) and cycles (right).
    }

    \label{fig:persistent-homology}
\end{figure}

\subsection{Vectorizations in Topological Data Analysis}

Vectorization is standard in the literature on persistent homology and topological data analysis \cite{ali2023survey, hensel2021survey, pun2018persistent}. However, there is no consensus on a formal definition. We therefore introduce a definition for vectorization.

Vectorizations make barcodes and persistence diagrams accessible to statistical and machine-learning methods. Standard examples include persistence landscapes \cite{bubenik2015statistical, bubenik2020persistence} (Figure~\ref{fig:curve-vectorizations}), persistence images \cite{adams2017persistence} (Figure~\ref{fig:surface-vectorizations}), persistence silhouettes \cite{chazal2014stochastic}, tropical coordinates \cite{kalivsnik2019tropical}, persistence paths and signature features \cite{chevyrev2018persistence, giusti2023signatures}, PersLay \cite{Carrire2019PersLayAN}, and persistence kernels \cite{reininghaus2015stable, kusano2016persistence, carriere2017sliced, le2018persistence}.

The codomains of standard persistence diagram vectorizations include several types of normed spaces. Persistence images \cite{adams2017persistence} and tropical coordinates \cite{kalivsnik2019tropical} have finite-dimensional normed vector spaces as codomains. Persistence landscapes \cite{bubenik2015statistical, bubenik2020persistence} have Banach spaces as codomains, and the vectorizations associated with persistence kernels \cite{reininghaus2015stable, kusano2016persistence, carriere2017sliced, le2018persistence} have Hilbert spaces as codomains.

Finite-dimensional normed vector spaces and Hilbert spaces are both Banach spaces, so we take the codomain of a vectorization to be a Banach space.

\begin{definition}\label{def:tda-vectorization}
Let \(\mathcal B\) be a collection of persistence diagrams (or the corresponding persistence barcodes), and let \(E\) be a Banach space. A \emph{vectorization} is a map \(\Phi\colon\mathcal B\to E\).
\end{definition}

\begin{figure}[t]
    \centering
    \includegraphics[width=\textwidth]{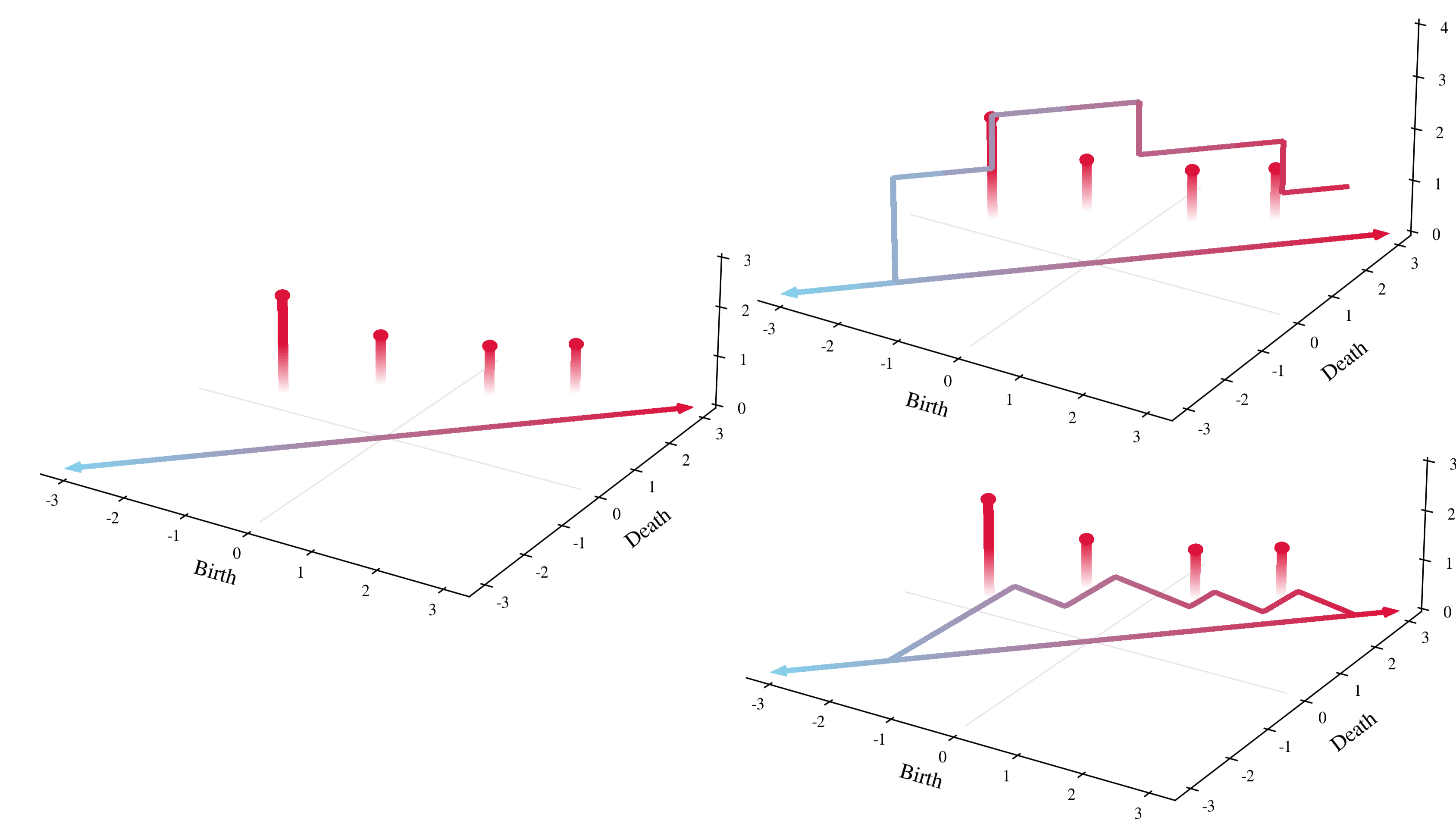}
    \caption{
        A persistence diagram and two standard vectorizations evaluated on it. The left panel shows the persistence diagram.  The upper-right panel shows the Betti curve, and the lower-right panel shows the first persistence landscape.
    }
    \label{fig:curve-vectorizations}
\end{figure}

\begin{figure}[t]
    \centering
    \includegraphics[width=\textwidth]{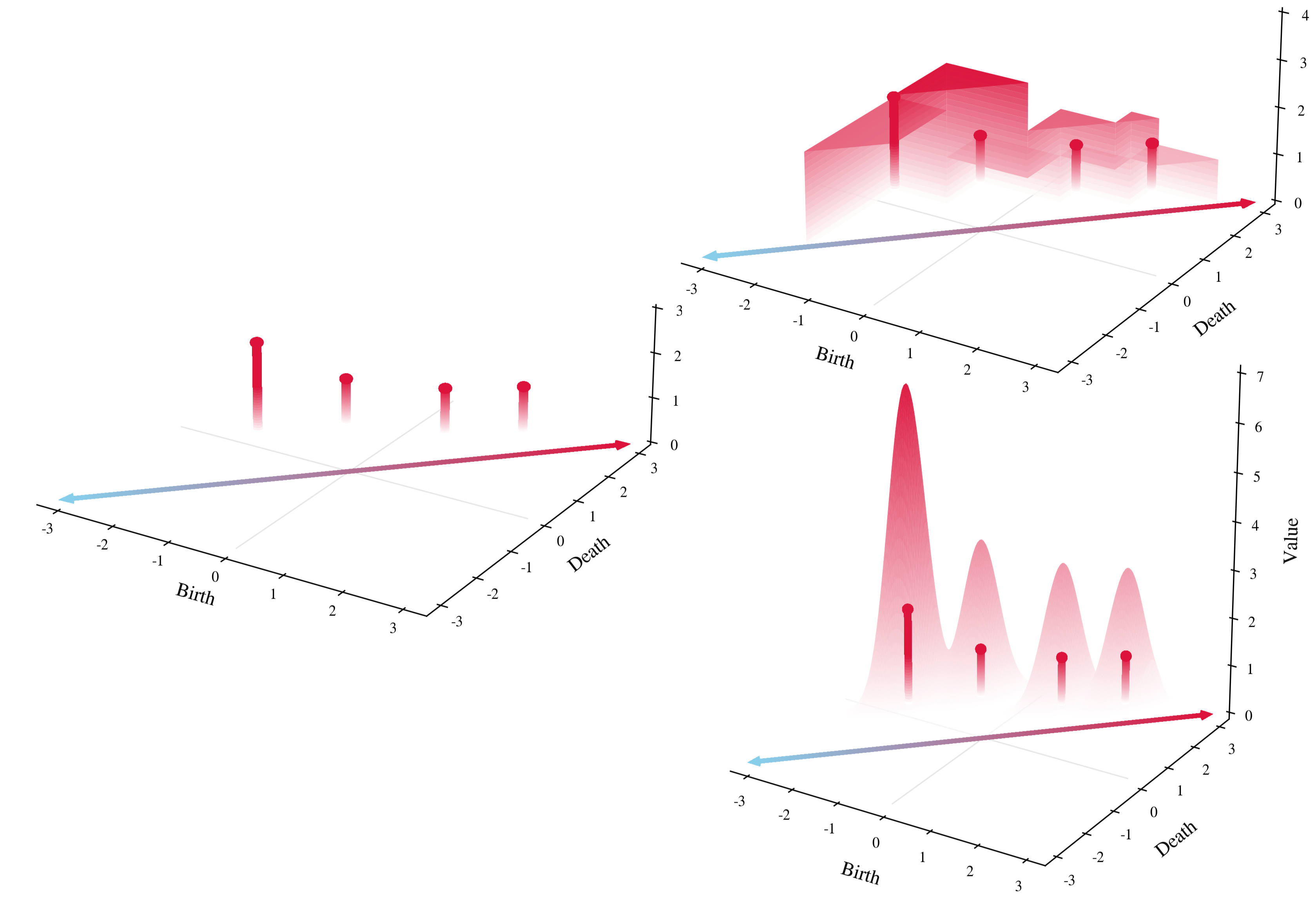}
    \caption{
    A persistence diagram and two standard vectorizations evaluated on it. The left panel shows the persistence diagram. The upper-right panel shows the rank function, and the lower-right panel shows the persistence image.
    }
    \label{fig:surface-vectorizations}
\end{figure}

\subsection{Virtual Persistence Diagrams}

A \emph{metric pair} is a triple \((X,d,A)\), where \(d\colon X\times X\to[0,\infty]\) is an extended pseudometric and \(A\subseteq X\) \cite[Section~2.1]{bubenik2022virtual}. In this paper, \(A\) represents the diagonal in \(X\).

For a metric pair \((X,d,A)\), let \(X/A\) be the quotient space obtained by replacing \(A\) with a single point \([A]\). The canonical quotient map \(q\colon X\to X/A\) sends every point of \(A\) to \([A]\) and fixes every point of \(X\setminus A\).

\begin{definition}\label{def:one-strengthening}
Let \((X,d,A)\) be a metric pair. For \(x\in X\), define \(d(x,A)=\inf_{a\in A}d(x,a)\). The \emph{\(1\)-strengthened pseudometric} on \(X\) is
\begin{equation}\label{eq:one-strengthening}
d_1(x,y)=\min\{d(x,y),d(x,A)+d(y,A)\}.
\end{equation}
\end{definition}

The quotient metric \(\overline d_1\) on \(X/A\) is given by \(\overline d_1(q(x),q(y))=d_1(x,y)\) \cite[Section~2.1]{bubenik2022virtual}.

\begin{definition}\label{def:persistence-diagram-monoid}
Let \(D(X)\) be the free commutative monoid on \(X\), whose elements are finite formal sums \(\alpha=\sum_{i=1}^{n}m_i x_i\), where \(m_i\in\mathbb N\) and \(x_i\in X\). The free commutative monoid of persistence diagrams over the metric pair \((X,d,A)\) is \(D(X,A)=D(X)/D(A)\) \cite[Sections~2.2--2.3]{bubenik2022virtual}.
\end{definition}

Two persistence diagrams are equivalent modulo the diagonal, \(A\), if they have the same multiplicities on the same persistence intervals off of the diagonal, regardless of their multiplicities on the diagonal. That is, \(D(X,A)\cong D(X\setminus A)\).

With the notation of \cite[Section~2.4]{bubenik2022virtual}, let \(\pi_1,\pi_2\colon X\times X\to X\) be the coordinate projections defined by \(\pi_1(x,y)=x\) and \(\pi_2(x,y)=y\). A matching between \(\alpha,\beta\in D(X,A)\) is an element \(\sigma\in D(X\times X)\) satisfying
\[
(\pi_1)_*\sigma=\alpha\ (\mathrm{mod}\ D(A)), \qquad
(\pi_2)_*\sigma=\beta\ (\mathrm{mod}\ D(A)).
\]

\begin{definition}\label{def:wasserstein-diagrams}
For \(\alpha,\beta\in D(X,A)\), the \emph{\(1\)-Wasserstein distance} is
\[
W_1(\alpha,\beta)
=
\inf_{\sigma=\sum_{i=1}^{n}(x_i,y_i)}
\sum_{i=1}^{n} d_1(x_i,y_i).
\]
Here, \(\sigma=\sum_{i=1}^{n}(x_i,y_i)\) is a matching between \(\alpha\) and \(\beta\) \cite[Section~2.4 and Proposition~4.3]{bubenik2022virtual}.
\end{definition}

\begin{theorem}\label{thm:wasserstein-translation-invariant}
The \(1\)-Wasserstein metric on the commutative monoid of persistence diagrams is translation-invariant \cite[Corollary~4.6]{bubenik2022virtual}. That is,
\begin{equation}\label{eq:wasserstein-translation-invariant}
W_1(\alpha+\eta,\beta+\eta)=W_1(\alpha,\beta)
\end{equation}
for all \(\alpha,\beta,\eta\in D(X,A)\).
\end{theorem}

\begin{definition}\label{def:virtual-persistence-diagrams}
The \emph{Grothendieck completion} \(K(X,A)\) of \(D(X,A)\) is the set of equivalence classes of pairs \((\alpha,\beta)\in D(X,A)\times D(X,A)\), where \((\alpha,\beta)\sim(\gamma,\delta)\) if and only if \(\alpha+\delta=\gamma+\beta\). Addition is defined by \([(\alpha,\beta)]+[(\gamma,\delta)]=[(\alpha+\gamma,\beta+\delta)]\).

We write \(\alpha-\beta\) for the equivalence class of \((\alpha,\beta)\). The elements of \(K(X,A)\) are called \emph{virtual persistence diagrams}, which allow for both positive and negative-valued multiplicities on persistence intervals, as illustrated in Figure~\ref{fig:virtual-persistence-diagrams} \cite[Section~2.5 and Definition~4.10]{bubenik2022virtual}.
\end{definition}

\begin{definition}\label{def:vpd-metric}
Let \(\rho\) denote the metric on \(K(X,A)\) from \cite[Definition~4.10]{bubenik2022virtual}, defined by
\begin{equation}\label{eq:groth-metric}
\rho(\alpha-\beta,\gamma-\delta)
=
W_1(\alpha+\delta,\gamma+\beta),
\qquad
\alpha,\beta,\gamma,\delta\in D(X,A).
\end{equation}
\end{definition}

\begin{figure}[t]
    \centering
    \includegraphics[width=\textwidth]{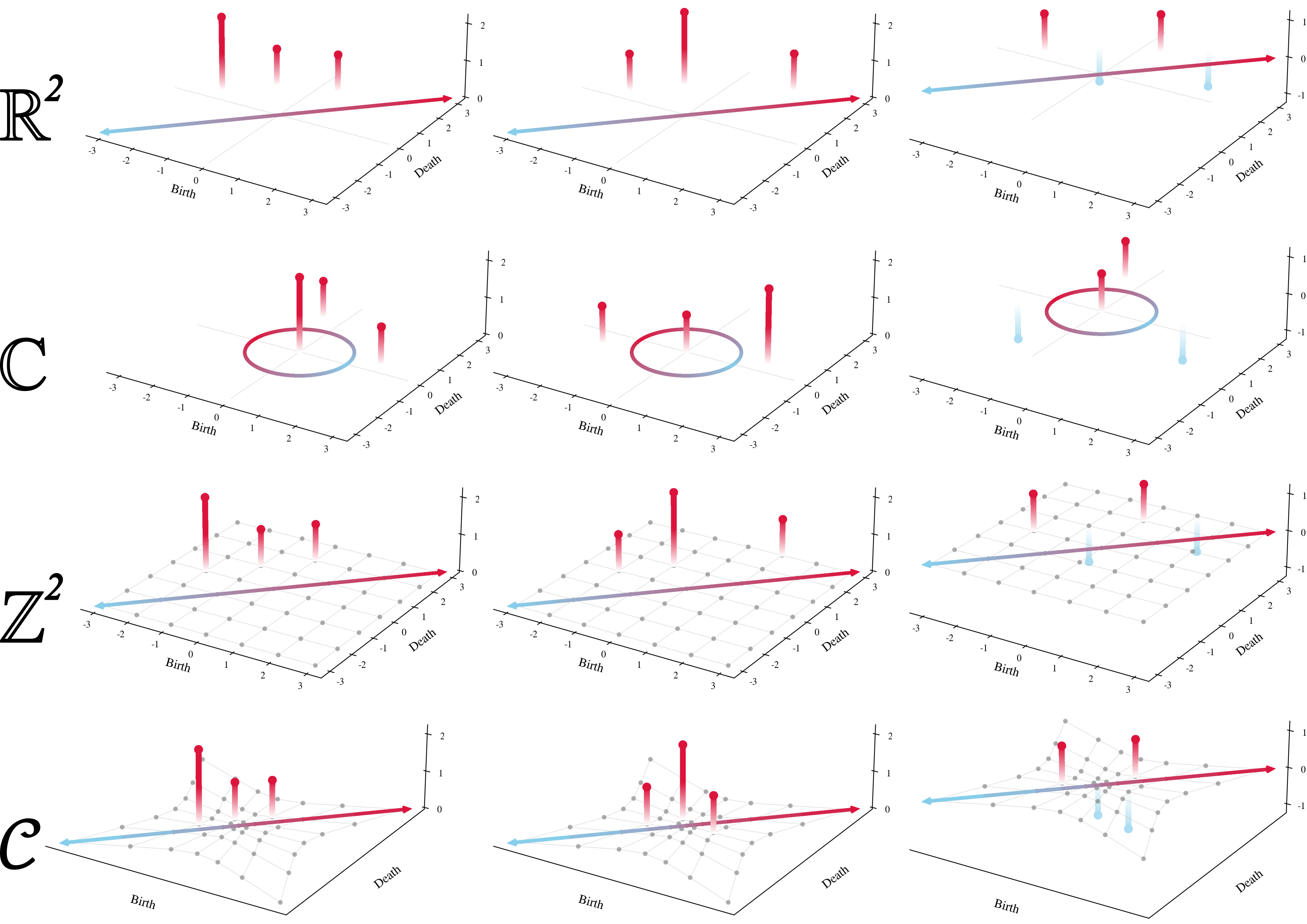}
    \caption{Examples of persistence and virtual persistence diagrams. From top to bottom, the metric pairs are $(\mathbb{R}^2,\Delta)$, $(\mathbb{C},S^1)$, $(\mathbb{Z}^2,\Delta_{\mathbb{Z}})$, and $(\mathcal{C},\Delta_{\mathcal{C}})$, where $\mathcal{C}$ is a coarse metric space. In each row, the first two columns are persistence diagrams and the third is the virtual persistence diagram obtained by subtracting their multiplicities pointwise. Positive multiplicities are shown above the plane in red and negative multiplicities below it in blue.}
    \label{fig:virtual-persistence-diagrams}
\end{figure}

If \(X/A\) is finite, then \(K(X,A)\cong\mathbb Z^{|X\setminus A|}\) \cite{fanningaktas2025rkhsvpd}.

\begin{theorem}\label{thm:vpd-local-compactness}
The following are equivalent \cite[Corollary~3.3]{fanningaktas2026banachrkhs}:
\begin{enumerate}
\item The pointed metric space \((X/A,\overline d_1,[A])\) is uniformly discrete.
\item The metric group \((K(X,A),\rho)\) is discrete.
\item The metric group \((K(X,A),\rho)\) is locally compact.
\end{enumerate}
\end{theorem}

\subsection{Spectral Synthesis}

We use the following notation and conventions for spectral synthesis throughout the paper.

\begin{figure}[t]
    \centering
    \includegraphics[width=\textwidth]{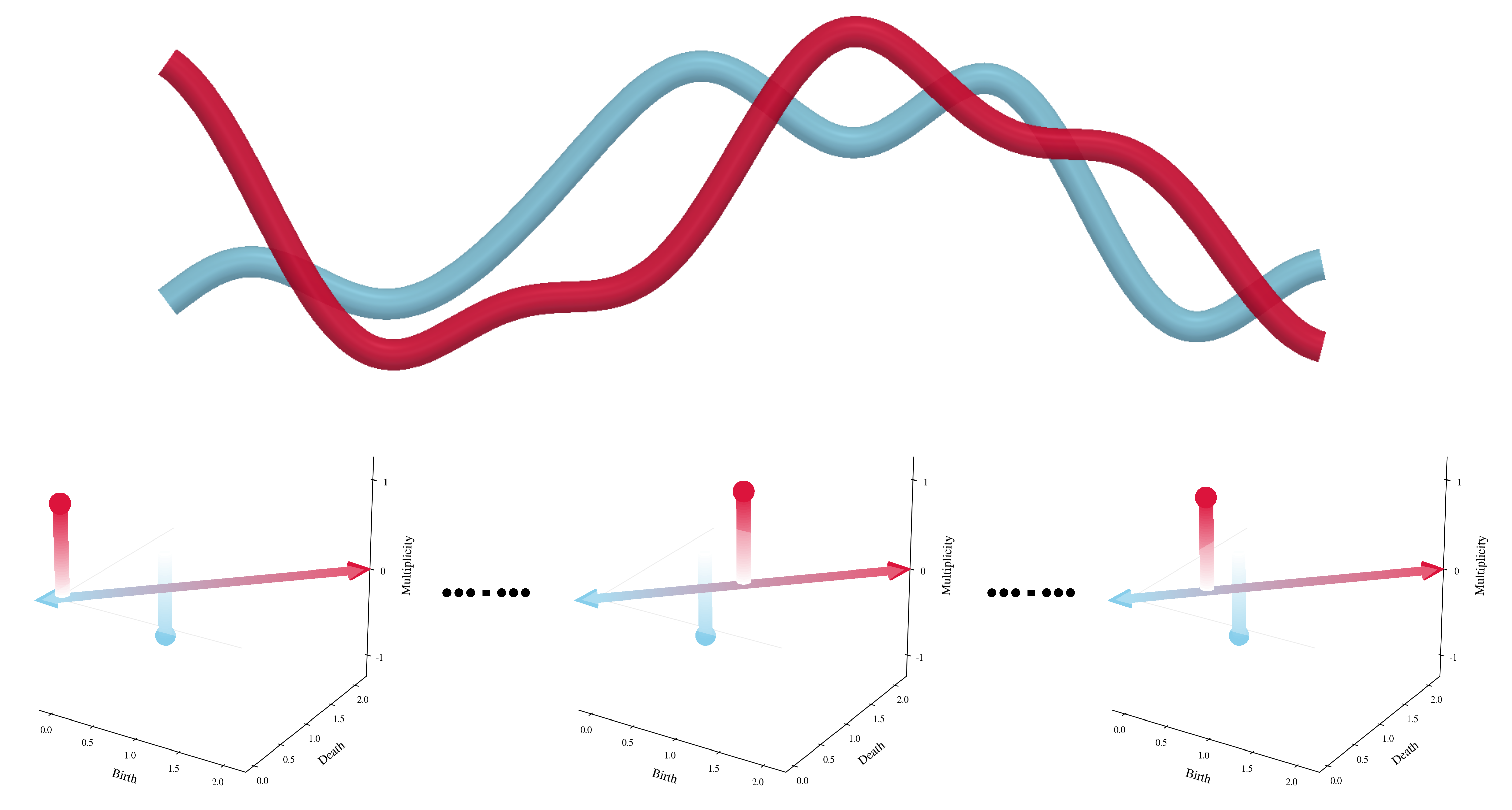}
    \caption{The three virtual persistence diagrams show a persistence interval moving parallel to and near the diagonal while another remains fixed. The red and blue curves show the real and imaginary parts, respectively, of a continuous complex-valued function on the space of virtual persistence diagrams.}
    \label{fig:continuous-vpd-function}
\end{figure}

\begin{definition}\label{def:variety}
Let \(G\) be a topological abelian group.
\begin{enumerate}
\item Let \(C(G)\) denote the space of continuous complex-valued functions on \(G\), equipped with the topology of uniform convergence on compact subsets \cite[p.~29]{szekelyhidi2015spectral}, as illustrated in Figure~\ref{fig:continuous-vpd-function}.
\item For each \(y\in G\), define the translation operator \(\tau_y\) by \(\tau_yu(x)=u(x+y)\).
\item A \emph{variety} on \(G\) is a closed complex linear subspace \(V\) of \(C(G)\) such that \(\tau_y V=V\) for every \(y\in G\) \cite[p.~30]{szekelyhidi2015spectral} \cite[p.~207]{szekelyhidi2014harmonic}, as illustrated in Figure~\ref{fig:variety}.
\begin{enumerate}
    \item A \emph{subvariety} of a variety \(V\) is a variety on \(G\) that is a subset of \(V\).
    \item For \(\mathcal F\subseteq C(G)\), the \emph{variety generated by \(\mathcal F\)}, denoted \(\operatorname{Var}_G(\mathcal F)\), is the smallest variety on \(G\) containing \(\mathcal F\).
\end{enumerate}
\item An \emph{exponential} is a continuous homomorphism \(m\colon G\to\mathbb C^\times\), so \(m(x+y)=m(x)m(y)\) for all \(x,y\in G\).
\begin{enumerate}
    \item An \emph{additive function} is a continuous homomorphism \(a\colon G\to\mathbb C\), so \(a(x+y)=a(x)+a(y)\) for all \(x,y\in G\).
    \item A \emph{polynomial} on \(G\) is a function \(p\colon G\to\mathbb C\) defined by \(p(x)=P(a_1(x),\ldots,a_N(x))\), where \(P\) is an ordinary complex polynomial in \(N\) variables and \(a_1,\ldots,a_N\) are additive functions \cite[p.~31]{szekelyhidi2015spectral}.
    \item An \emph{exponential monomial} on \(G\) is a function \(pm\), where \(p\) is a polynomial and \(m\) is an exponential \cite[pp.~31--32]{szekelyhidi2015spectral}.
    \item An \emph{exponential polynomial} on \(G\) is a finite complex linear combination of exponential monomials.
\end{enumerate}
\end{enumerate}
\end{definition}

\begin{figure}[t]
    \centering
    \includegraphics[width=\textwidth]{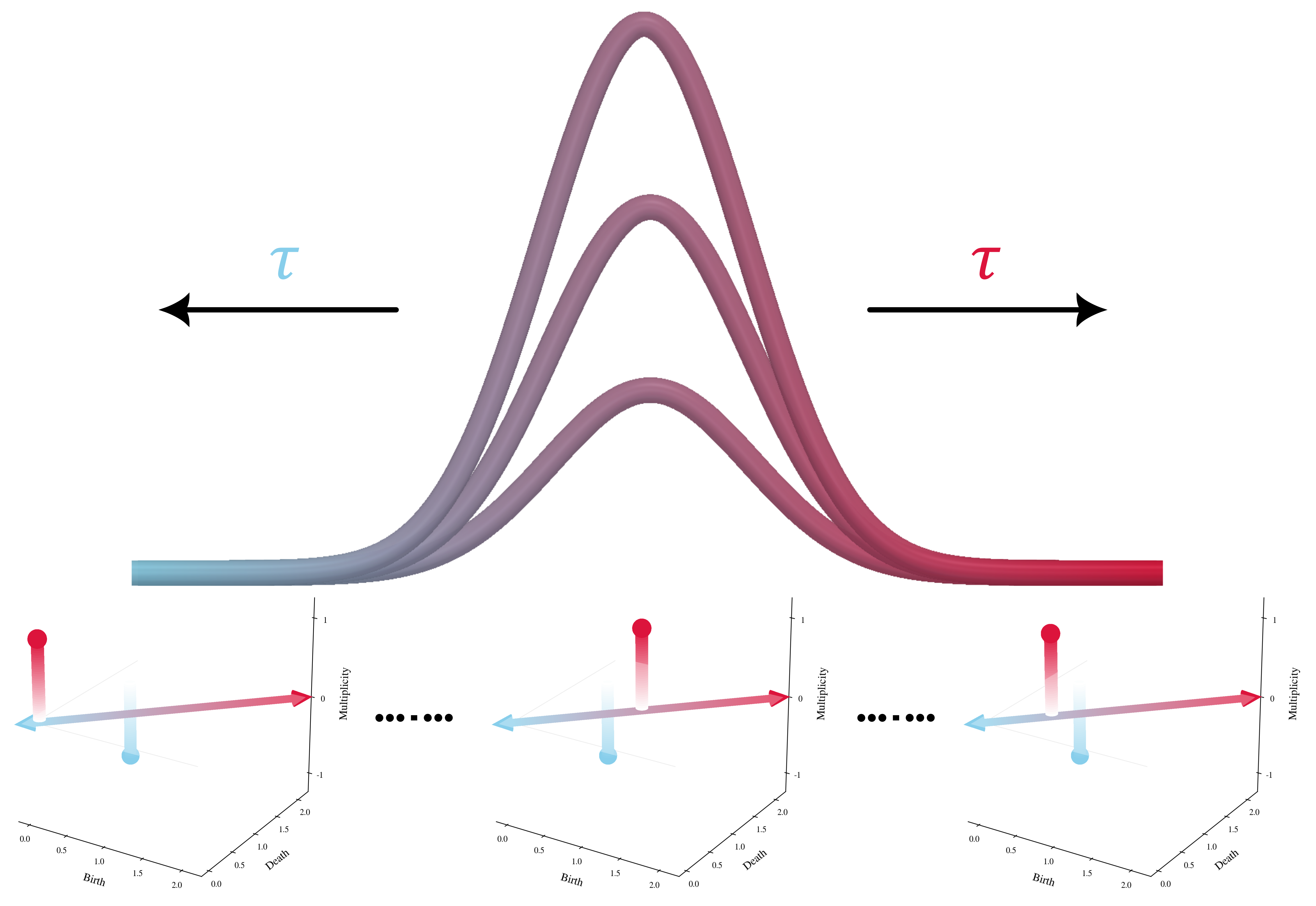}
    \caption{The stacked curves show scalar multiples of a function in a variety. The blue and red arrows show translations of those curves in two different directions.}
    \label{fig:variety}
\end{figure}

Exponential polynomials give a tractable way to analyze translation-invariant spaces of functions. Every translate of an exponential polynomial is a linear combination of the same finite collection of functions, so analyzing how translations transform these functions reduces to a problem in linear algebra.

\begin{definition}\label{def:synthesizable-variety}
A variety \(V\) is \emph{synthesizable} if its finite-dimensional subvarieties span a dense subvariety of \(V\) \cite[p.~207]{szekelyhidi2014harmonic}.
\end{definition}

\begin{theorem}\label{thm:synthesis-equivalent-characterizations}
Let \(V\) be a variety on \(G\). The following conditions are equivalent \cite[Theorem~15.6]{szekelyhidi2014harmonic}:
\begin{enumerate}
\item The variety \(V\) is synthesizable.
\item The exponential monomials in \(V\) span a dense subspace of \(V\).
\item The exponential polynomials in \(V\) span a dense subspace of \(V\).
\end{enumerate}
\end{theorem}

Spectral synthesis holds on \(G\) if every variety on \(G\) is synthesizable \cite[p.~207]{szekelyhidi2014harmonic} \cite[p.~32]{szekelyhidi2015spectral}. Every function in a finite-dimensional variety is an exponential polynomial, so the variety is spanned exactly by exponential polynomials and is therefore synthesizable \cite[p.~207]{szekelyhidi2014harmonic} \cite[p.~32]{szekelyhidi2015spectral}.

Spectral synthesis on discrete abelian groups is completely classified by the following theorem.

\begin{theorem}\label{thm:discrete-group-spectral-synthesis}
Spectral synthesis holds on a discrete Abelian group \(G\) if and only if its torsion free rank is finite \cite[Theorem~1]{laczkovich2007spectral}.
\end{theorem}

\subsection{Pontryagin Duality and the Fourier Transform}

Pontryagin duality extends classical Fourier analysis to locally compact abelian groups through their characters. Figure~\ref{fig:fourier-transform} illustrates the Fourier transform on \(\mathbb Z^2\).

\begin{definition}\label{def:pontryagin-dual}
Let \(G\) be a locally compact abelian group. The \emph{Pontryagin dual} \(\widehat G\) is the group of continuous homomorphisms \(\chi\colon G\to\mathbb T\), equipped with the compact-open topology. The elements of \(\widehat G\) are called \emph{characters} \cite[Definition~4.1]{folland2016course}.
\end{definition}

We write \(\mathbf1\) for the trivial character. Every character satisfies \(\tau_h\chi=\chi(h)\chi\) for each \(h\in G\), so characters are simultaneous eigenfunctions of the translation operators.

\begin{definition}\label{def:fourier-stieltjes-transform}
Let \(M(\widehat G)\) denote the Banach space of finite complex regular Borel measures on \(\widehat G\), equipped with the total variation norm. For \(\mu\in M(\widehat G)\), the \emph{Fourier--Stieltjes transform} of \(\mu\) is the function \(\widehat\mu\colon G\to\mathbb C\) defined by
\begin{equation}\label{eq:fourier-stieltjes-transform}
\widehat\mu(x)=\int_{\widehat G}\chi(x)\,d\mu(\chi).
\end{equation}
We write \(B(G)\) for the space of Fourier--Stieltjes transforms on \(G\) \cite[Equation~(4.17) and p.~105]{folland2016course}.
\end{definition}

For a complex measure, its support is the support of its total variation measure. A measure with finite support is a finite linear combination of point masses, so its Fourier--Stieltjes transform is a finite linear combination of characters.

Each function in \(B(G)\) has a unique representing measure in \(M(\widehat G)\) \cite[Corollary~4.34]{folland2016course}. 

\begin{theorem}[Bochner's Theorem]\label{thm:bochner}
Every continuous positive-definite function on a locally compact abelian group \(G\) is the Fourier--Stieltjes transform of a unique finite positive measure on \(\widehat G\) \cite[Theorem~4.19]{folland2016course}.
\end{theorem}

\begin{figure}[t]
    \centering
    \includegraphics[width=\textwidth]{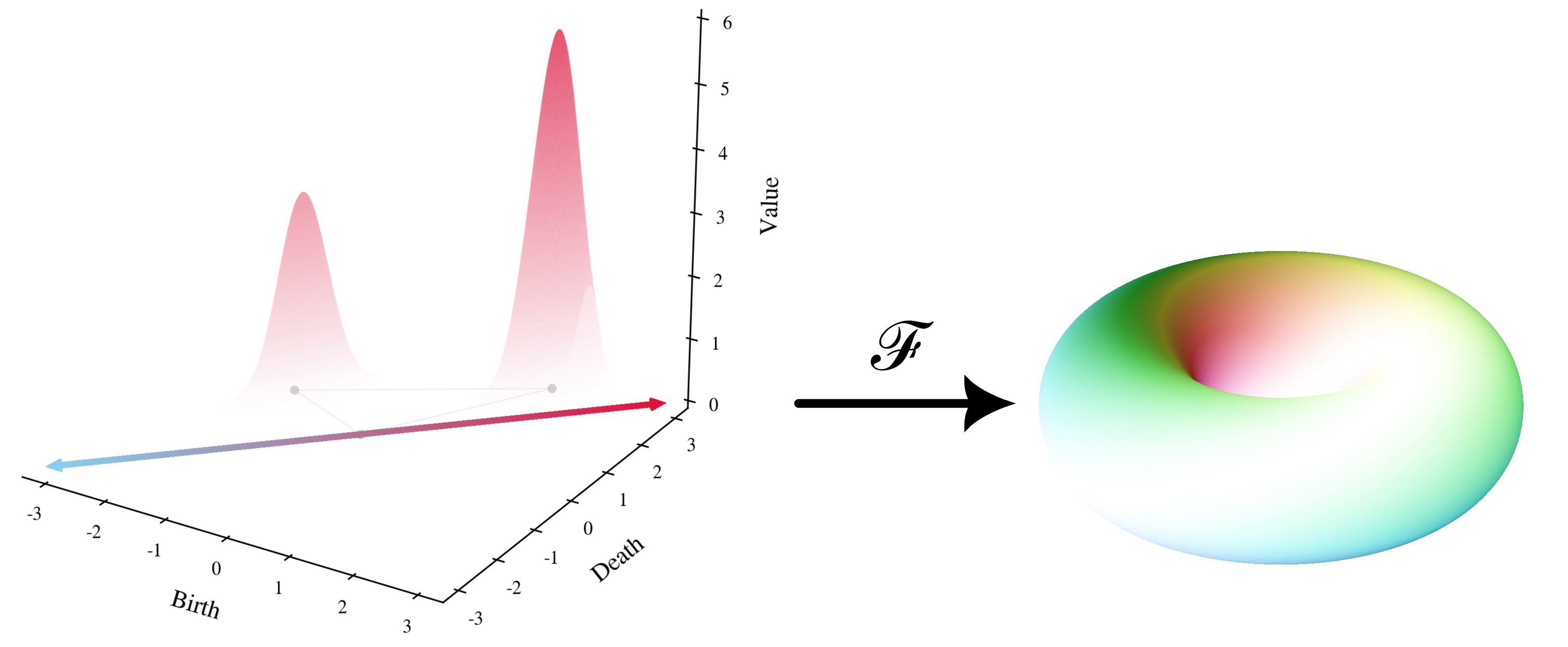}
    \caption{
        Fourier transform on the subgroup generated by two off-diagonal points. The left panel shows a function on the space of persistence intervals and two selected points whose values are $a$ and $b$. These define $f=a\delta_{(1,0)}+b\delta_{(0,1)}$ on $\mathbb{Z}^2$. The right panel shows $\widehat{f}(\theta_1,\theta_2)=ae^{-i\theta_1}+be^{-i\theta_2}$ on $\widehat{\mathbb{Z}^2}\cong\mathbb{T}^2$, with hue showing complex argument and brightness showing magnitude.
    }
    \label{fig:fourier-transform}
\end{figure}

\subsection{Reproducing Kernel Hilbert Spaces}

Reproducing kernel Hilbert spaces allow us to study positive definite kernels through the inner products of Hilbert spaces.

\begin{definition}\label{def:pd-kernel}
Let \(X\) be a set. A function \(k\colon X\times X\to\mathbb C\) is a \emph{positive definite kernel} if \(\sum_{i,j=1}^n c_i\overline{c_j}k(x_i,x_j)\geq0\) for every positive integer \(n\), every \(x_1,\ldots,x_n\in X\), and every \(c_1,\ldots,c_n\in\mathbb C\) \cite[Definition~2 and the following observation, pp.~10--11]{Berlinet2004RKHS}.
\end{definition}

\begin{definition}\label{def:rkhs}
Let \(\mathcal H\) be a Hilbert space of complex-valued functions on \(X\). A function \(k\colon X\times X\to\mathbb C\) is a \emph{reproducing kernel} of \(\mathcal H\) if \(k(\cdot,x)\in\mathcal H\) and
\begin{equation}\label{eq:reproducing-property}
f(x)
=
\langle f,k(\cdot,x)\rangle_{\mathcal H}
\end{equation}
for every \(f\in\mathcal H\) and \(x\in X\) \cite[Definition~1, pp.~6--7]{Berlinet2004RKHS}.
\end{definition}

\begin{theorem}[Moore--Aronszajn]\label{thm:moore-aronszajn}
Every positive definite kernel \(k\) on \(X\) determines a unique reproducing kernel Hilbert space \(\mathcal H_k\) with reproducing kernel \(k\). The linear span of \(\{k(\cdot,x):x\in X\}\) is dense in \(\mathcal H_k\) \cite[Theorem~3, pp.~19--21]{Berlinet2004RKHS}.
\end{theorem}

\section{Spectral Synthesis for Uniformly Discrete Subpairs}
\label{sec:uniformly-discrete-synthesis}

We study spectral synthesis on \(K(X,A)\) through its subgroups \(K(Z,A)\), obtained by restricting the corresponding space of persistence intervals \(X\) to subsets \(A\subseteq Z\subseteq X\) for which \(Z/A\) is uniformly discrete.

\subsection{Additive Functions and Fourier--Stieltjes Transforms}

The scalar-valued functions in Theorem~\ref{thm:uniformly-discrete-fourier-stieltjes-consistency} have the form \(a+\widehat\mu\), where \(a\) is additive and \(\widehat\mu\) is a Fourier--Stieltjes transform. We prove that these functions generate synthesizable varieties.

For \(h\in G\), we write \(\Delta_hu=\tau_hu-u\). For \(F\subseteq G\), we write \(T_Fu=u|_F\).

\begin{lemma}\label{lem:additive-fourier-stieltjes-synthesis}
Let \(G\) be a discrete abelian group, let \(\mathcal F\subseteq C(G)\), and suppose that every \(u\in\mathcal F\) can be written as \(u=a_u+\widehat\mu_u\), where \(a_u\colon G\to\mathbb C\) is additive and \(\mu_u\in M(\widehat G)\). Set \(V=\operatorname{Var}_G(\mathcal F)\). Then:
\begin{enumerate}
\item The variety \(V\) is the pointwise closed linear span of the additive functions \(a_u\), the characters in \(\operatorname{supp}(\mu_u)\), and \(\mathbf1\) whenever some \(a_u\) is nonzero. In particular, \(V\) is synthesizable.
\item For every finite nonempty \(F\subseteq G\), the restrictions of these functions to \(F\) span \(T_F(V)\). If every \(a_u\) vanishes, then
\begin{equation}\label{eq:fourier-stieltjes-finite-support-span}
T_F(V)
=
\operatorname{span}_{\mathbb C}
\left\{
T_F\chi:
u\in\mathcal F,\
\chi\in\operatorname{supp}(\mu_u)
\right\}.
\end{equation}
\end{enumerate}
\end{lemma}

\begin{proof}
By the definition of spectral synthesis and Proposition~4.79 of \cite{folland2016course}, the weak-* closed linear span of the translates of \(\widehat\nu\) equals the weak-* closed linear span of the characters in \(\operatorname{supp}(\nu)\). Since \(G\) is discrete, evaluation at each \(x\in G\) is weak-* continuous on \(\ell^\infty(G)\), as \(\delta_x\in\ell^1(G)\). Hence every character in \(\operatorname{supp}(\nu)\) belongs to \(\operatorname{Var}_G\{\widehat\nu\}\), and \(\widehat\nu\) belongs to their pointwise closed linear span.

For \(u=a+\widehat\mu\in\mathcal F\) and \(h\in G\), the increment \(\Delta_hu\) belongs to \(\operatorname{Var}_G\{u\}\) and satisfies
\[
\Delta_hu(x)=a(h)+\int_{\widehat G}\chi(x)(\chi(h)-1)\,d\mu(\chi).
\]
Thus \(\Delta_hu\) is the Fourier--Stieltjes transform of \(a(h)\delta_{\mathbf1}+(\chi(h)-1)\mu\), where \(\delta_{\mathbf1}\) is the point mass at \(\mathbf1\) and the second term denotes multiplication of \(\mu\) by \(\chi\mapsto\chi(h)-1\). For each \(\chi_0\in\operatorname{supp}(\mu)\setminus\{\mathbf1\}\), choose \(h\in G\) with \(\chi_0(h)\neq1\). Since \(\chi\mapsto\chi(h)-1\) is continuous and nonzero at \(\chi_0\), the measure of \(\Delta_hu\) has \(\chi_0\) in its support. The preceding result gives \(\chi_0\in\operatorname{Var}_G\{u\}\). If \(a\neq0\), choose \(h\) with \(a(h)\neq0\). The measure of \(\Delta_hu\) then has a nonzero atom at \(\mathbf1\), so \(\mathbf1\in\operatorname{Var}_G\{u\}\). If \(a=0\) and \(\mathbf1\in\operatorname{supp}(\mu)\), this conclusion follows directly from \(u=\widehat\mu\). Therefore every character in \(\operatorname{supp}(\mu)\) belongs to \(\operatorname{Var}_G\{u\}\). Since \(\widehat\mu\) belongs to their pointwise closed linear span, the additive function \(a=u-\widehat\mu\) also belongs to \(\operatorname{Var}_G\{u\}\).

Let \(W\) be the pointwise closed linear span of the functions in the lemma. The preceding argument gives \(W\subseteq V\). Conversely, the identities \(\tau_ha_u=a_u+a_u(h)\mathbf1\) and \(\tau_h\chi=\chi(h)\chi\) show that \(W\) is translation-invariant. Since \(W\) contains each \(\widehat\mu_u\), it also contains \(\mathcal F\), so \(V\subseteq W\). Hence \(V=W\). Every generator of \(W\) is an exponential monomial, so Theorem~\ref{thm:synthesis-equivalent-characterizations} implies that \(V\) is synthesizable.

For finite nonempty \(F\subseteq G\), the restriction map \(T_F\) is continuous for pointwise convergence, and every linear subspace of \(\mathbb C^F\) is closed. Therefore the restrictions of the generators of \(W\) span \(T_F(V)\). If every \(a_u\) vanishes, only the characters in the supports of the measures remain, giving Equation~\eqref{eq:fourier-stieltjes-finite-support-span}.
\end{proof}

\begin{lemma}\label{lem:additive-fourier-stieltjes-uniqueness}
Let \(G\) be a discrete abelian group, let \(a,b\colon G\to\mathbb C\) be additive, and let \(\mu,\nu\in M(\widehat G)\). If \(a+\widehat\mu=b+\widehat\nu\), then \(a=b\) and \(\mu=\nu\).
\end{lemma}

\begin{proof}
The equality \(a-b=\widehat{\nu-\mu}\) shows that \(a-b\) is bounded and additive. Since \((a-b)(ng)=n(a-b)(g)\) for every \(g\in G\) and every positive integer \(n\), boundedness implies \(a=b\). The injectivity of the Fourier--Stieltjes transform then gives \(\mu=\nu\) \cite[Corollary~4.34]{folland2016course}.
\end{proof}

\subsection{Restrictions to Uniformly Discrete Subpairs}

Theorem~\ref{thm:uniformly-discrete-fourier-stieltjes-consistency} gives conditions under which the scalarizations of a persistence diagram vectorization \(f\colon K(X,A)\to E\) have the form \(a+\widehat\mu\) on every uniformly discrete subpair \( (Z,A) \) with \(A\subseteq Z\subseteq X\).
\begin{enumerate}
\item We write \(i_Z\colon K(Z,A)\to K(X,A)\) for the homomorphism induced by inclusion.
\item We define \(V_Z(f)=\operatorname{Var}_{K(Z,A)}\{\lambda\circ f\circ i_Z:\lambda\in E^*\}\).
\end{enumerate}

\begin{definition}\label{def:vectorization-spectral-synthesis}
The map \(f\) has \emph{uniformly discrete spectral synthesis} if \(V_Z(f)\) is synthesizable for every uniformly discrete subpair \(A\subseteq Z\subseteq X\).
\end{definition}

When \(X/A\) is separable, we can study uniformly discrete spectral synthesis using a filtration of uniformly discrete subpairs whose images have dense union in \(X/A\).

\begin{lemma}\label{lem:separable-uniformly-discrete-approximants}
Suppose that \(X/A\) is separable. There are nested subpairs \(A\subseteq Y_1\subseteq Y_2\subseteq\cdots\subseteq X\) and positive numbers \(\varepsilon_n\downarrow0\) such that \(Y_n/A\) is uniformly discrete and the image of \(Y_n/A\) in \(X/A\) is an \(\varepsilon_n\)-net for every \(n\).
\end{lemma}

\begin{proof}
Let \(\{p_j:j\geq1\}\) be a countable dense subset of \(X/A\) containing \([A]\), and choose positive numbers \(\varepsilon_n\downarrow0\).

Construct \(Q_1\) by starting with \([A]\) and adding each \(p_j\), in order, if its distance from the points already selected is at least \(\varepsilon_1/2\). For \(n\geq1\), construct \(Q_{n+1}\) by starting with \(Q_n\) and applying the same procedure with threshold \(\varepsilon_{n+1}/2\). Since \(\varepsilon_{n+1}<\varepsilon_n\), the points of \(Q_n\) already satisfy the new separation condition. Hence \(Q_n\subseteq Q_{n+1}\), and each \(Q_n\) is \(\varepsilon_n/2\)-separated.

To show that \(Q_n\) is an \(\varepsilon_n\)-net, let \(p\in X/A\). By density, some \(p_j\) satisfies \(\overline d_1(p,p_j)<\varepsilon_n/2\). The construction ensures that \(p_j\) either belongs to \(Q_n\) or lies within distance \(\varepsilon_n/2\) of a point \(q\in Q_n\). The triangle inequality then gives \(\overline d_1(p,q)<\varepsilon_n\).

Choose one representative in \(X\) for each non-basepoint class in \(\bigcup_n Q_n\), and let \(Y_n\) contain \(A\) and the representatives of \(Q_n\setminus\{[A]\}\). Since the representatives are fixed across all levels, \(Y_n\subseteq Y_{n+1}\). Moreover, \(Y_n/A\) is isometric to \(Q_n\), so it is uniformly discrete and its image in \(X/A\) is an \(\varepsilon_n\)-net.
\end{proof}

\begin{figure}[t]
    \centering
    \includegraphics[width=\textwidth]{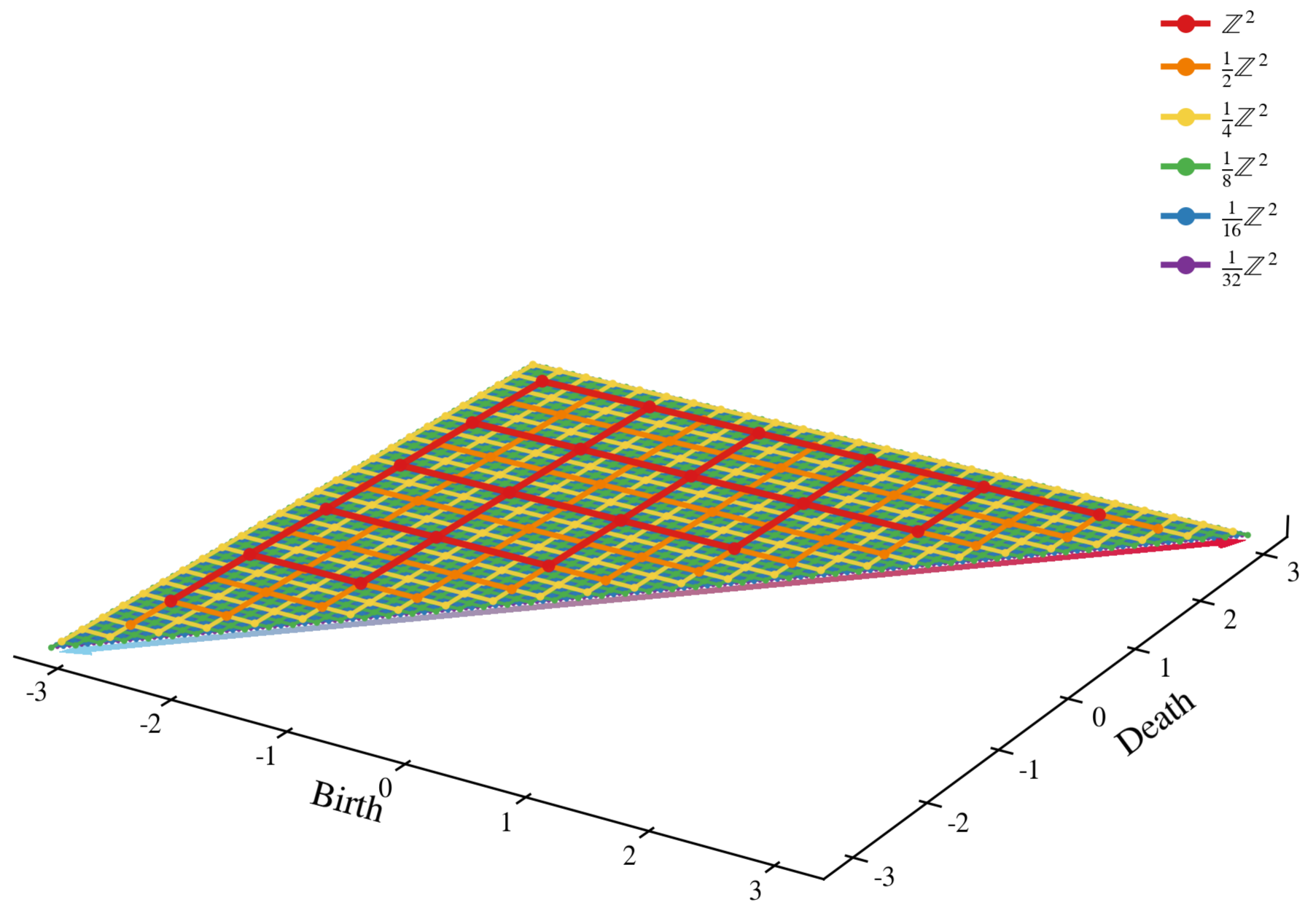}
    \caption{A filtration of uniformly discrete subpairs of \((\mathbb R^2,\Delta)\). The lattices $\mathbb Z^2$, $\frac12\mathbb Z^2$, $\frac14\mathbb Z^2$, $\frac18\mathbb Z^2$, $\frac1{16}\mathbb Z^2$, and $\frac1{32}\mathbb Z^2$ are shown from coarse to fine. Each lattice contains the preceding one, and the corresponding subpairs form $\varepsilon_n$-nets in $\mathbb R^2/\Delta$ with $\varepsilon_n\to0$, as in Lemma~\ref{lem:separable-uniformly-discrete-approximants}.}
    \label{fig:uniformly-discrete-approximants}
\end{figure}

The next lemma constructs homomorphisms from \(K(Z,A)\) to the uniformly discrete approximations \(K(Y_n,A)\) from Lemma~\ref{lem:separable-uniformly-discrete-approximants}.

\begin{lemma}\label{lem:approximating-homomorphisms}
For each \(n\), choose \(r_n\colon X/A\to Y_n/A\) such that \(r_n([A])=[A]\), \(r_n(p)=p\) for \(p\in Y_n/A\), and \(d_1(r_n(p),p)\leq\varepsilon_n\). For every subpair \(A\subseteq Z\subseteq X\), restriction of \(r_n\) induces a homomorphism \(P_n^Z\colon K(Z,A)\to K(Y_n,A)\). These homomorphisms satisfy the following properties.
\begin{enumerate}
\item If \(\xi=\alpha-\beta\in K(Z,A)\), then
\begin{equation}\label{eq:approximating-homomorphism-error}
\rho(i_{Y_n}P_n^Z\xi,i_Z\xi)
\leq
(N(\alpha)+N(\beta))\varepsilon_n.
\end{equation}

\item If \(A\subseteq Z_0\subseteq Z_1\subseteq X\) and \(j_{Z_1,Z_0}\colon K(Z_0,A)\to K(Z_1,A)\) denotes inclusion, then
\begin{equation}\label{eq:approximating-homomorphism-compatibility}
P_n^{Z_0}
=
P_n^{Z_1}\circ j_{Z_1,Z_0}.
\end{equation}
\end{enumerate}
\end{lemma}

\begin{proof}
Because \(r_n([A])=[A]\), the restriction of \(r_n\) to \(Z/A\) induces a monoid homomorphism \(D(Z,A)\to D(Y_n,A)\). By the universal property of the Grothendieck group, this homomorphism extends uniquely to \(P_n^Z\).

Let \(\xi=\alpha-\beta\). We match every occurrence in \(\alpha\) and \(\beta\) with the class selected by \(r_n\). Then \(W_1(i_{Y_n}P_n^Z\alpha,i_Z\alpha)\leq N(\alpha)\varepsilon_n\), and the analogous inequality holds for \(\beta\). By Equation~\eqref{eq:groth-metric} and the triangle inequality, we obtain Equation~\eqref{eq:approximating-homomorphism-error}.

By construction, all maps \(P_n^Z\) arise from the same map \(r_n\). Hence \(P_n^{Z_0}\) and \(P_n^{Z_1}\circ j_{Z_1,Z_0}\) agree on every generator of \(K(Z_0,A)\). This proves Equation~\eqref{eq:approximating-homomorphism-compatibility}.
\end{proof}

\begin{theorem}\label{thm:uniformly-discrete-fourier-stieltjes-consistency}
Suppose that \(X/A\) is separable. Let \(E\) be a complex Banach space, let \(f\colon K(X,A)\to E\) be continuous, and let \(Y_n\) and \(P_n^Z\) be as in Lemmas~\ref{lem:separable-uniformly-discrete-approximants} and \ref{lem:approximating-homomorphisms}.

For every uniformly discrete subpair \(A\subseteq Z\subseteq X\), let \(\widehat{P_n^Z}\colon\widehat{K(Y_n,A)}\to\widehat{K(Z,A)}\) denote the dual homomorphism. Suppose that:
\begin{itemize}
\item For every \(n\geq1\) and every \(\lambda\in E^*\), there are an additive function \(a_{n,\lambda}\colon K(Y_n,A)\to\mathbb C\) and a measure \(\mu_{n,\lambda}\in M(\widehat{K(Y_n,A)})\) such that
\begin{equation}\label{eq:approximant-additive-fourier-stieltjes}
\lambda(f(i_{Y_n}\eta))
=
a_{n,\lambda}(\eta)
+
\int_{\widehat{K(Y_n,A)}}
\chi(\eta)\,d\mu_{n,\lambda}(\chi)
\end{equation}
for every \(\eta\in K(Y_n,A)\).

\item For every uniformly discrete subpair \(A\subseteq Z\subseteq X\) and every \(\lambda\in E^*\), the measures
\(\nu_{n,\lambda}^Z=(\widehat{P_n^Z})_\#\mu_{n,\lambda}\) satisfy
\[
\sup_{n\geq1}\|\nu_{n,\lambda}^Z\|_{\mathrm{TV}}<\infty.
\]
\end{itemize}

Write \(a_{n,\lambda}^Z=a_{n,\lambda}\circ P_n^Z\). For uniformly discrete subpairs \(A\subseteq Z_0\subseteq Z_1\subseteq X\), let \(j_{Z_1,Z_0}\colon K(Z_0,A)\to K(Z_1,A)\) denote the inclusion homomorphism. Then we have the following:
\begin{enumerate}
\item For every \(\lambda\in E^*\) and integers \(m\geq n\geq1\),
\begin{equation}\label{eq:approximant-level-measure-consistency}
a_{n,\lambda}
=
a_{m,\lambda}\circ j_{Y_m,Y_n},
\qquad
(\widehat{j_{Y_m,Y_n}})_\#\mu_{m,\lambda}
=
\mu_{n,\lambda}.
\end{equation}
For every \(\lambda\in E^*\), every \(n\geq1\), and uniformly discrete subpairs \(A\subseteq Z_0\subseteq Z_1\subseteq X\),
\begin{equation}\label{eq:finite-level-measure-consistency}
a_{n,\lambda}^{Z_0}
=
a_{n,\lambda}^{Z_1}\circ j_{Z_1,Z_0},
\qquad
(\widehat{j_{Z_1,Z_0}})_\#
\nu_{n,\lambda}^{Z_1}
=
\nu_{n,\lambda}^{Z_0}.
\end{equation}

\item For every uniformly discrete subpair \(A\subseteq Z\subseteq X\) and every \(\lambda\in E^*\), there are unique \(a_\lambda^Z\in\operatorname{Hom}(K(Z,A),\mathbb C)\) and \(\nu_\lambda^Z\in M(\widehat{K(Z,A)})\) such that
\begin{equation}\label{eq:compatible-fourier-stieltjes-representation}
\lambda(f(i_Z\xi))
=
a_\lambda^Z(\xi)
+
\int_{\widehat{K(Z,A)}}
\chi(\xi)\,d\nu_\lambda^Z(\chi)
\end{equation}
for every \(\xi\in K(Z,A)\). Moreover,
\begin{equation}\label{eq:compatible-fourier-stieltjes-convergence}
\nu_{n,\lambda}^Z
\stackrel{*}{\rightharpoonup}
\nu_\lambda^Z,
\end{equation}
and \(a_{n,\lambda}^Z(\xi)\to a_\lambda^Z(\xi)\) for every \(\xi\in K(Z,A)\).

\item For every \(\lambda\in E^*\) and uniformly discrete subpairs \(A\subseteq Z_0\subseteq Z_1\subseteq X\), the inclusion homomorphism \(j_{Z_1,Z_0}\colon K(Z_0,A)\to K(Z_1,A)\) satisfies
\begin{equation}\label{eq:compatible-fourier-stieltjes-restriction}
a_\lambda^{Z_0}
=
a_\lambda^{Z_1}\circ j_{Z_1,Z_0},
\qquad
(\widehat{j_{Z_1,Z_0}})_\#
\nu_\lambda^{Z_1}
=
\nu_\lambda^{Z_0}.
\end{equation}

\item For every uniformly discrete subpair \(A\subseteq Z\subseteq X\) and every \(n\geq1\), the maps \(\lambda\mapsto a_{n,\lambda}^Z\) and \(\lambda\mapsto\nu_{n,\lambda}^Z\) are linear. The limiting maps \(\lambda\mapsto a_\lambda^Z\) and \(\lambda\mapsto\nu_\lambda^Z\) are also linear. The measure-valued maps are bounded, and
\begin{equation}\label{eq:uniform-measure-operator-bound}
\sup_{n\geq1}
\sup_{\|\lambda\|\leq1}
\|\nu_{n,\lambda}^Z\|_{\mathrm{TV}}
<
\infty.
\end{equation}
Moreover,
\begin{equation}\label{eq:limiting-measure-operator-bound}
\sup_{\|\lambda\|\leq1}
\|\nu_\lambda^Z\|_{\mathrm{TV}}
\leq
\sup_{n\geq1}
\sup_{\|\lambda\|\leq1}
\|\nu_{n,\lambda}^Z\|_{\mathrm{TV}}.
\end{equation}

\item For every uniformly discrete subpair \(A\subseteq Z\subseteq X\), every \(\lambda\in E^*\), and every closed set \(S\subseteq\widehat{K(Z,A)}\), if \(\operatorname{supp}(\nu_{n,\lambda}^Z)\subseteq S\) for every \(n\), then \(\operatorname{supp}(\nu_\lambda^Z)\subseteq S\).

\item For every uniformly discrete subpair \(A\subseteq Z\subseteq X\), the variety \(V_Z(f)\) equals the pointwise closed linear span of the additive functions \(a_\lambda^Z\), the characters in \(\operatorname{supp}(\nu_\lambda^Z)\), and \(\mathbf1\) whenever some \(a_\lambda^Z\) is nonzero, as \(\lambda\) ranges over \(E^*\). Consequently, \(f\) has uniformly discrete spectral synthesis.
\end{enumerate}
\end{theorem}

\begin{proof}
\textbf{Proof of (1).}
For \(\lambda\in E^*\) and integers \(m\geq n\geq1\), restrict Equation~\eqref{eq:approximant-additive-fourier-stieltjes} from \(K(Y_m,A)\) to \(K(Y_n,A)\). The additive part becomes \(a_{m,\lambda}\circ j_{Y_m,Y_n}\), and the measure becomes \((\widehat{j_{Y_m,Y_n}})_\#\mu_{m,\lambda}\). Both give the same function \(\lambda\circ f\circ i_{Y_n}\) as the additive part \(a_{n,\lambda}\) and measure \(\mu_{n,\lambda}\) at level \(n\). Lemma~\ref{lem:additive-fourier-stieltjes-uniqueness} therefore gives Equation~\eqref{eq:approximant-level-measure-consistency}.

For uniformly discrete subpairs \(A\subseteq Z_0\subseteq Z_1\subseteq X\), Equation~\eqref{eq:approximating-homomorphism-compatibility} gives \(\widehat{P_n^{Z_0}}=\widehat{j_{Z_1,Z_0}}\circ\widehat{P_n^{Z_1}}\). Composition of the additive functions and the corresponding pushforwards of the measures give Equation~\eqref{eq:finite-level-measure-consistency}.

\textbf{Proof of (2).}
Let \(A\subseteq Z\subseteq X\) be a uniformly discrete subpair and let \(\lambda\in E^*\). By Theorem~\ref{thm:vpd-local-compactness}, \(K(Z,A)\) is discrete and \(\widehat{K(Z,A)}\) is compact. The assumed total-variation bound and the Banach--Alaoglu theorem give a weak-* convergent subnet of \((\nu_{n,\lambda}^Z)\) in \(M(\widehat{K(Z,A)})=C(\widehat{K(Z,A)})^*\). Let \(\nu\) be its limit.

For every \(\xi\in K(Z,A)\), Equation~\eqref{eq:approximant-additive-fourier-stieltjes} and the definition of pushforward give
\[
a_{n,\lambda}^Z(\xi)
=
\lambda(f(i_{Y_n}P_n^Z\xi))
-
\int_{\widehat{K(Z,A)}}\chi(\xi)\,d\nu_{n,\lambda}^Z(\chi).
\]
Equation~\eqref{eq:approximating-homomorphism-error} and continuity of \(\lambda\circ f\) imply that the first term converges to \(\lambda(f(i_Z\xi))\). Since \(\chi\mapsto\chi(\xi)\) is continuous on \(\widehat{K(Z,A)}\), weak-* convergence gives convergence of the integral along the subnet. Consequently,
\[
a_{n,\lambda}^Z(\xi)
\longrightarrow
a(\xi)
:=
\lambda(f(i_Z\xi))
-
\int_{\widehat{K(Z,A)}}\chi(\xi)\,d\nu(\chi).
\]
The function \(a\) is additive because it is a pointwise limit of additive functions. Thus \(\lambda\circ f\circ i_Z=a+\widehat\nu\). Lemma~\ref{lem:additive-fourier-stieltjes-uniqueness} gives the unique pair \(a_\lambda^Z,\nu_\lambda^Z\) in Equation~\eqref{eq:compatible-fourier-stieltjes-representation}.

The same argument applies to every weak-* convergent subnet of \((\nu_{n,\lambda}^Z)\). By uniqueness, every weak-* cluster point equals \(\nu_\lambda^Z\). Since the sequence lies in a weak-* compact Hausdorff space, it converges to this unique cluster point, proving Equation~\eqref{eq:compatible-fourier-stieltjes-convergence}. The preceding identity then gives \(a_{n,\lambda}^Z(\xi)\to a_\lambda^Z(\xi)\) for every \(\xi\in K(Z,A)\).

\textbf{Proof of (3).}
For uniformly discrete subpairs \(A\subseteq Z_0\subseteq Z_1\subseteq X\), the pushforward induced by \(\widehat{j_{Z_1,Z_0}}\) is weak-* continuous because \(\varphi\circ\widehat{j_{Z_1,Z_0}}\) is continuous on \(\widehat{K(Z_1,A)}\) for every \(\varphi\in C(\widehat{K(Z_0,A)})\). Equation~\eqref{eq:compatible-fourier-stieltjes-convergence} therefore allows us to take weak-* limits in the measure identity of Equation~\eqref{eq:finite-level-measure-consistency}. The pointwise convergence of the additive functions gives the corresponding additive identity. These limits establish Equation~\eqref{eq:compatible-fourier-stieltjes-restriction}.

\textbf{Proof of (4).}
Linearity of the scalarizations and Lemma~\ref{lem:additive-fourier-stieltjes-uniqueness} imply linearity of the assignments \(\lambda\mapsto a_{n,\lambda}^Z\), \(\lambda\mapsto\nu_{n,\lambda}^Z\), \(\lambda\mapsto a_\lambda^Z\), and \(\lambda\mapsto\nu_\lambda^Z\).

For a uniformly discrete subpair \(A\subseteq Z\subseteq X\) and \(n\geq1\), consider the linear map \(\lambda\mapsto\nu_{n,\lambda}^Z\) from \(E^*\) to \(M(\widehat{K(Z,A)})\). Suppose that \(\lambda_k\to\lambda\) in \(E^*\) and \(\nu_{n,\lambda_k}^Z\to\nu\) in total variation. For every \(\xi\in K(Z,A)\), the identity
\[
a_{n,\lambda_k}^Z(\xi)
=
\lambda_k(f(i_{Y_n}P_n^Z\xi))
-
\int_{\widehat{K(Z,A)}}\chi(\xi)\,d\nu_{n,\lambda_k}^Z(\chi)
\]
shows that \(a_{n,\lambda_k}^Z\) converges pointwise to \(b=\lambda\circ f\circ i_{Y_n}\circ P_n^Z-\widehat\nu\). Since each \(a_{n,\lambda_k}^Z\) is additive, \(b\) is additive. Lemma~\ref{lem:additive-fourier-stieltjes-uniqueness} then gives \(\nu=\nu_{n,\lambda}^Z\). Thus the map has a closed graph and is bounded by the closed graph theorem.

The assumed total-variation bounds and the uniform boundedness principle give Equation~\eqref{eq:uniform-measure-operator-bound}. Moreover, Equation~\eqref{eq:compatible-fourier-stieltjes-convergence} and weak-* lower semicontinuity of the total-variation norm give \(\|\nu_\lambda^Z\|_{\mathrm{TV}}\leq\liminf_{n\to\infty}\|\nu_{n,\lambda}^Z\|_{\mathrm{TV}}\). Taking the supremum over \(\|\lambda\|\leq1\) proves Equation~\eqref{eq:limiting-measure-operator-bound} and boundedness of the limiting measure-valued map.

\textbf{Proof of (5).}
Let \(S\subseteq\widehat{K(Z,A)}\) be closed, with \(\operatorname{supp}(\nu_{n,\lambda}^Z)\subseteq S\) for every \(n\). Every continuous function on \(\widehat{K(Z,A)}\) that vanishes on \(S\) has zero integral against each \(\nu_{n,\lambda}^Z\). By Equation~\eqref{eq:compatible-fourier-stieltjes-convergence}, its integral against \(\nu_\lambda^Z\) also vanishes. Since \(\nu_\lambda^Z\) is regular, \(\operatorname{supp}(\nu_\lambda^Z)\subseteq S\).

\textbf{Proof of (6).}
For every uniformly discrete subpair \(A\subseteq Z\subseteq X\), Equation~\eqref{eq:compatible-fourier-stieltjes-representation} expresses each scalarization \(\lambda\circ f\circ i_Z\) as \(a_\lambda^Z+\widehat{\nu_\lambda^Z}\). Lemma~\ref{lem:additive-fourier-stieltjes-synthesis} gives the asserted pointwise closed linear span of \(V_Z(f)\) and proves that this variety is synthesizable. Hence \(f\) has uniformly discrete spectral synthesis.
\end{proof}

\begin{corollary}\label{cor:pure-fourier-stieltjes-finite-span}
Assume the hypotheses of Theorem~\ref{thm:uniformly-discrete-fourier-stieltjes-consistency} and suppose that \(a_{n,\lambda}=0\) for every \(n\) and every \(\lambda\in E^*\). For every uniformly discrete subpair \(A\subseteq Z\subseteq X\) and every finite nonempty \(F\subseteq K(Z,A)\),
\begin{equation}\label{eq:uniformly-discrete-finite-character-span}
T_F(V_Z(f))
=
\operatorname{span}_{\mathbb C}
\left\{
T_F\chi:
\lambda\in E^*,\
\chi\in\operatorname{supp}(\nu_\lambda^Z)
\right\}.
\end{equation}
\end{corollary}

\begin{proof}
For every \(n\), every \(\lambda\in E^*\), and every uniformly discrete subpair \(A\subseteq Z\subseteq X\), we have \(a_{n,\lambda}^Z=a_{n,\lambda}\circ P_n^Z=0\). By item~2 of Theorem~\ref{thm:uniformly-discrete-fourier-stieltjes-consistency}, \(a_{n,\lambda}^Z(\xi)\to a_\lambda^Z(\xi)\) for every \(\xi\in K(Z,A)\), so \(a_\lambda^Z=0\). Applying Lemma~\ref{lem:additive-fourier-stieltjes-synthesis} to \(\{\lambda\circ f\circ i_Z:\lambda\in E^*\}\) and using \(a_\lambda^Z=0\) for every \(\lambda\in E^*\), we obtain Equation~\eqref{eq:uniformly-discrete-finite-character-span}.
\end{proof}

\subsection{Spectral Synthesis on \texorpdfstring{\(K(X,A)\)}{K(X,A)} with the Discrete Topology}

Let \(K(X,A)_{\mathrm d}\) denote \(K(X,A)\) with the discrete topology, and define \(V_X(f)=\operatorname{Var}_{K(X,A)_{\mathrm d}}\{\lambda\circ f:\lambda\in E^*\}\). Since \(K(X,A)\) is free abelian on \(X/A\setminus\{[A]\}\), we have \(\widehat{K(X,A)_{\mathrm d}}\cong\mathbb T^{\,X/A\setminus\{[A]\}}\). For every finite subpair \(A\subseteq Z\subseteq X\), the dual homomorphism \(\widehat{i_Z}\) is the corresponding coordinate projection and is therefore surjective.

Every element of \(K(X,A)\) belongs to \(K(Z,A)\) for some finite subpair \(A\subseteq Z\subseteq X\). Under the assumptions of Corollary~\ref{cor:full-group-spectral-synthesis}, Equation~\eqref{eq:compatible-fourier-stieltjes-restriction} then allows us to extend certain additive functions and measures to \(K(X,A)_{\mathrm d}\).

\begin{corollary}\label{cor:full-group-spectral-synthesis}
Assume the hypotheses of Theorem~\ref{thm:uniformly-discrete-fourier-stieltjes-consistency}. Suppose that, for every \(\lambda\in E^*\),
\begin{equation}\label{eq:finite-restriction-uniform-measure-bound}
\sup_{\substack{A\subseteq Z\subseteq X\\Z/A\text{ finite}}}
\|\nu_\lambda^Z\|_{\mathrm{TV}}
<
\infty.
\end{equation}
Then we have the following:
\begin{enumerate}
\item The measures satisfy
\begin{equation}\label{eq:finite-restriction-uniform-operator-bound}
C
:=
\sup_{\substack{A\subseteq Z\subseteq X\\Z/A\text{ finite}}}
\sup_{\|\lambda\|\leq1}
\|\nu_\lambda^Z\|_{\mathrm{TV}}
<
\infty.
\end{equation}

\item For every \(\lambda\in E^*\), there are unique \(a_\lambda^X\in\operatorname{Hom}(K(X,A),\mathbb C)\) and \(\nu_\lambda^X\in M(\widehat{K(X,A)_{\mathrm d}})\) such that
\begin{equation}\label{eq:full-group-fourier-stieltjes-representation}
\lambda(f(\xi))
=
a_\lambda^X(\xi)
+
\int_{\widehat{K(X,A)_{\mathrm d}}}
\chi(\xi)\,d\nu_\lambda^X(\chi)
\end{equation}
for every \(\xi\in K(X,A)\). For every finite subpair \(A\subseteq Z\subseteq X\),
\begin{equation}\label{eq:full-group-measure-pushforward}
a_\lambda^Z
=
a_\lambda^X\circ i_Z,
\qquad
(\widehat{i_Z})_\#\nu_\lambda^X
=
\nu_\lambda^Z.
\end{equation}

\item The maps \(\lambda\mapsto a_\lambda^X\) and \(\lambda\mapsto\nu_\lambda^X\) are linear, the measure-valued map is bounded, and
\begin{equation}\label{eq:full-group-measure-operator-bound}
\|\nu_\lambda^X\|_{\mathrm{TV}}
\leq
C\|\lambda\|.
\end{equation}

\item The variety \(V_X(f)\) equals the pointwise closure of the span of the additive functions \(a_\lambda^X\) for \(\lambda\in E^*\), the characters in \(\operatorname{supp}(\nu_\lambda^X)\) for \(\lambda\in E^*\), and \(\mathbf1\) whenever some \(a_\lambda^X\) is nonzero. Consequently, \(V_X(f)\) is synthesizable.
\end{enumerate}
\end{corollary}

\begin{proof}
\textbf{Proof of (1).}
For every finite subpair \(A\subseteq Z\subseteq X\), item~4 of Theorem~\ref{thm:uniformly-discrete-fourier-stieltjes-consistency} gives a bounded linear map \(\lambda\mapsto\nu_\lambda^Z\) from \(E^*\) to \(M(\widehat{K(Z,A)})\). For each such \(Z\) and each continuous complex-valued function \(g\) on \(\widehat{K(Z,A)}\) with \(\|g\|_\infty\leq1\), consider the bounded linear functional
\[
\lambda
\longmapsto
\int_{\widehat{K(Z,A)}}g\,d\nu_\lambda^Z.
\]
For every \(\lambda\in E^*\), Equation~\eqref{eq:finite-restriction-uniform-measure-bound} gives
\[
\sup_{\substack{A\subseteq Z\subseteq X,\ Z/A\text{ finite}\\ \|g\|_\infty\leq1}}
\left|
\int_{\widehat{K(Z,A)}}g\,d\nu_\lambda^Z
\right|
=
\sup_{\substack{A\subseteq Z\subseteq X\\ Z/A\text{ finite}}}
\|\nu_\lambda^Z\|_{\mathrm{TV}}
<
\infty.
\]
By the uniform boundedness principle, the operator norms of these scalar functionals are uniformly bounded. Taking the supremum over \(g\) proves Equation~\eqref{eq:finite-restriction-uniform-operator-bound}.

\textbf{Proof of (2).}
Fix \(\xi\in K(X,A)\). Since \(K(X,A)\) is free abelian on the non-basepoint classes of \(X/A\), there are a finite subpair \(A\subseteq Z\subseteq X\) and \(\eta\in K(Z,A)\) such that \(i_Z\eta=\xi\). We define \(a_\lambda^X(\xi)=a_\lambda^Z(\eta)\).

To show that \(a_\lambda^X\) is well defined, suppose that \(i_{Z_0}\eta_0=\xi=i_{Z_1}\eta_1\) for finite subpairs \(A\subseteq Z_0\subseteq X\) and \(A\subseteq Z_1\subseteq X\). Set \(Z=Z_0\cup Z_1\). Since \(i_Z\circ j_{Z,Z_k}=i_{Z_k}\) for \(k=0,1\), both \(j_{Z,Z_0}\eta_0\) and \(j_{Z,Z_1}\eta_1\) map to \(\xi\) under \(i_Z\). The map \(i_Z\) is injective because these Grothendieck groups are free abelian on nested sets of generators. Hence \(j_{Z,Z_0}\eta_0=j_{Z,Z_1}\eta_1\). By Equation~\eqref{eq:compatible-fourier-stieltjes-restriction},
\[
a_\lambda^{Z_0}(\eta_0)
=
a_\lambda^Z(j_{Z,Z_0}\eta_0)
=
a_\lambda^Z(j_{Z,Z_1}\eta_1)
=
a_\lambda^{Z_1}(\eta_1).
\]
Hence \(a_\lambda^X\) is well defined. Given \(\xi,\zeta\in K(X,A)\), we can choose one finite subpair \(A\subseteq Z\subseteq X\) and \(\eta,\theta\in K(Z,A)\) such that \(i_Z\eta=\xi\) and \(i_Z\theta=\zeta\). Then
\[
a_\lambda^X(\xi+\zeta)
=
a_\lambda^Z(\eta+\theta)
=
a_\lambda^Z(\eta)+a_\lambda^Z(\theta)
=
a_\lambda^X(\xi)+a_\lambda^X(\zeta),
\]
so \(a_\lambda^X\) is additive. Its definition and Equation~\eqref{eq:compatible-fourier-stieltjes-restriction} also give the additive identity in Equation~\eqref{eq:full-group-measure-pushforward}.

We call a function on \(\widehat{K(X,A)_{\mathrm d}}\) a continuous cylinder function if it has the form \(g\circ\widehat{i_Z}\), where \(A\subseteq Z\subseteq X\) is a finite subpair and \(g\) is continuous on \(\widehat{K(Z,A)}\). For such a function, define
\[
L_\lambda(g\circ\widehat{i_Z})
=
\int_{\widehat{K(Z,A)}}g\,d\nu_\lambda^Z.
\]
Suppose that \(g_0\circ\widehat{i_{Z_0}}=g_1\circ\widehat{i_{Z_1}}\). Set \(Z=Z_0\cup Z_1\), and let \(\widehat{j_{Z,Z_k}}\colon\widehat{K(Z,A)}\to\widehat{K(Z_k,A)}\) denote restriction for \(k=0,1\). Since \(\widehat{i_{Z_k}}=\widehat{j_{Z,Z_k}}\circ\widehat{i_Z}\) and \(\widehat{i_Z}\) is surjective, we have \(g_0\circ\widehat{j_{Z,Z_0}}=g_1\circ\widehat{j_{Z,Z_1}}\) on \(\widehat{K(Z,A)}\). By Equation~\eqref{eq:compatible-fourier-stieltjes-restriction},
\[
\begin{aligned}
\int_{\widehat{K(Z_0,A)}}g_0\,d\nu_\lambda^{Z_0}
&=
\int_{\widehat{K(Z,A)}}
g_0\circ\widehat{j_{Z,Z_0}}
\,d\nu_\lambda^Z\\
&=
\int_{\widehat{K(Z,A)}}
g_1\circ\widehat{j_{Z,Z_1}}
\,d\nu_\lambda^Z\\
&=
\int_{\widehat{K(Z_1,A)}}g_1\,d\nu_\lambda^{Z_1}.
\end{aligned}
\]
Hence \(L_\lambda\) is well defined. If two cylinder functions factor through different finite subpairs, both factor through their union, so linearity of integration on the common finite factor shows that \(L_\lambda\) is linear.

Surjectivity of \(\widehat{i_Z}\) gives \(\|g\circ\widehat{i_Z}\|_\infty=\|g\|_\infty\). By Equation~\eqref{eq:finite-restriction-uniform-operator-bound}, we have
\[
|L_\lambda(g\circ\widehat{i_Z})|
\leq
C\|\lambda\|\,
\|g\circ\widehat{i_Z}\|_\infty.
\]
The continuous cylinder functions form a unital self-adjoint algebra that separates points of \(\widehat{K(X,A)_{\mathrm d}}\). By the Stone--Weierstrass theorem, they are uniformly dense in the Banach space of continuous complex-valued functions on \(\widehat{K(X,A)_{\mathrm d}}\), equipped with the supremum norm. Hence \(L_\lambda\) extends uniquely to a bounded linear functional on this Banach space with norm at most \(C\|\lambda\|\). By the Riesz representation theorem, there is a unique \(\nu_\lambda^X\in M(\widehat{K(X,A)_{\mathrm d}})\) representing this functional. Equation~\eqref{eq:full-group-measure-operator-bound} follows.

By construction, for every finite subpair \(A\subseteq Z\subseteq X\) and every continuous complex-valued function \(g\) on \(\widehat{K(Z,A)}\),
\[
\int_{\widehat{K(X,A)_{\mathrm d}}}
g\circ\widehat{i_Z}\,d\nu_\lambda^X
=
\int_{\widehat{K(Z,A)}}g\,d\nu_\lambda^Z.
\]
Hence \((\widehat{i_Z})_\#\nu_\lambda^X=\nu_\lambda^Z\), which proves the measure identity in Equation~\eqref{eq:full-group-measure-pushforward}.

Fix \(\xi\in K(X,A)\) and choose a finite subpair \(A\subseteq Z\subseteq X\) and \(\eta\in K(Z,A)\) such that \(i_Z\eta=\xi\). By Equation~\eqref{eq:compatible-fourier-stieltjes-representation},
\[
\lambda(f(\xi))
=
a_\lambda^Z(\eta)
+
\int_{\widehat{K(Z,A)}}\chi(\eta)\,d\nu_\lambda^Z(\chi).
\]
By Equation~\eqref{eq:full-group-measure-pushforward}, \(a_\lambda^Z(\eta)=a_\lambda^X(\xi)\), and the measure identity in that equation gives
\[
\int_{\widehat{K(Z,A)}}\chi(\eta)\,d\nu_\lambda^Z(\chi)
=
\int_{\widehat{K(X,A)_{\mathrm d}}}\chi(\xi)\,d\nu_\lambda^X(\chi).
\]
Thus Equation~\eqref{eq:full-group-fourier-stieltjes-representation} holds. By Lemma~\ref{lem:additive-fourier-stieltjes-uniqueness}, the additive function and measure in this equation are unique on \(K(X,A)_{\mathrm d}\).

\textbf{Proof of (3).}
Linearity of the scalarizations and Lemma~\ref{lem:additive-fourier-stieltjes-uniqueness} imply that the maps \(\lambda\mapsto a_\lambda^X\) and \(\lambda\mapsto\nu_\lambda^X\) are linear. The norm estimate established in (2) gives Equation~\eqref{eq:full-group-measure-operator-bound} and boundedness of the measure-valued map.

\textbf{Proof of (4).}
Lemma~\ref{lem:additive-fourier-stieltjes-synthesis}, applied to \(\{\lambda\circ f:\lambda\in E^*\}\) on \(K(X,A)_{\mathrm d}\), gives the asserted pointwise closed linear span of \(V_X(f)\) and proves that \(V_X(f)\) is synthesizable.
\end{proof}

Equation~\eqref{eq:approximating-measure-bound} implies the hypotheses of Theorem~\ref{thm:uniformly-discrete-fourier-stieltjes-consistency} and Corollary~\ref{cor:full-group-spectral-synthesis}, because pushforward is contractive in total variation and the total variation norm is weak-* lower semicontinuous.

\begin{corollary}\label{cor:uniform-approximating-measure-bound}
Assume the hypotheses and notation of Theorem~\ref{thm:uniformly-discrete-fourier-stieltjes-consistency}, with its hypothesis on the total variation of \(\nu_{n,\lambda}^Z\) replaced by
\begin{equation}\label{eq:approximating-measure-bound}
\sup_{n\geq1}
\|\mu_{n,\lambda}\|_{\mathrm{TV}}
<
\infty
\end{equation}
for every \(\lambda\in E^*\). Then we have the following:
\begin{enumerate}
\item Every conclusion of Theorem~\ref{thm:uniformly-discrete-fourier-stieltjes-consistency} and Corollary~\ref{cor:full-group-spectral-synthesis} holds.

\item The measures \(\mu_{n,\lambda}\) satisfy
\begin{equation}\label{eq:uniform-approximating-measure-bound}
M
:=
\sup_{n\geq1}
\sup_{\|\lambda\|\leq1}
\|\mu_{n,\lambda}\|_{\mathrm{TV}}
<
\infty.
\end{equation}

\item For every uniformly discrete subpair \(A\subseteq Z\subseteq X\) and every \(\lambda\in E^*\), the limiting measures satisfy
\begin{equation}\label{eq:uniform-local-measure-bound}
\|\nu_\lambda^Z\|_{\mathrm{TV}}
\leq M\|\lambda\|,
\qquad
\|\nu_\lambda^X\|_{\mathrm{TV}}
\leq M\|\lambda\|.
\end{equation}
\end{enumerate}
\end{corollary}

\begin{proof}
\textbf{Proof of (1).}
Pushforward is contractive for the total variation norm, so \(\|\nu_{n,\lambda}^Z\|_{\mathrm{TV}}\leq\|\mu_{n,\lambda}\|_{\mathrm{TV}}\) for every \(n\), every \(\lambda\in E^*\), and every uniformly discrete subpair \(A\subseteq Z\subseteq X\). Equation~\eqref{eq:approximating-measure-bound} therefore implies the hypothesis on \(\nu_{n,\lambda}^Z\) in Theorem~\ref{thm:uniformly-discrete-fourier-stieltjes-consistency}. We can therefore apply Theorem~\ref{thm:uniformly-discrete-fourier-stieltjes-consistency}. Since the total variation norm is weak-* lower semicontinuous,
\[
\|\nu_\lambda^Z\|_{\mathrm{TV}}
\leq
\liminf_{n\to\infty}
\|\nu_{n,\lambda}^Z\|_{\mathrm{TV}}
\leq
\sup_{n\geq1}
\|\mu_{n,\lambda}\|_{\mathrm{TV}}.
\]
In particular, Equation~\eqref{eq:finite-restriction-uniform-measure-bound} holds, so Corollary~\ref{cor:full-group-spectral-synthesis} applies.

\textbf{Proof of (2).}
Fix \(n\geq1\). By Lemma~\ref{lem:additive-fourier-stieltjes-uniqueness}, the map \(\lambda\mapsto\mu_{n,\lambda}\) is linear. Suppose that \(\lambda_k\to\lambda\) in \(E^*\) and \(\mu_{n,\lambda_k}\to\mu\) in total variation. For every \(\eta\in K(Y_n,A)\), the additive functions \(a_{n,\lambda_k}\) converge pointwise to
\[
b(\eta)
=
\lambda(f(i_{Y_n}\eta))
-
\int_{\widehat{K(Y_n,A)}}
\chi(\eta)\,d\mu(\chi).
\]
Hence \(b\) is additive. Thus \(\lambda\circ f\circ i_{Y_n}=b+\widehat\mu\). Equation~\eqref{eq:approximant-additive-fourier-stieltjes} also gives \(\lambda\circ f\circ i_{Y_n}=a_{n,\lambda}+\widehat{\mu_{n,\lambda}}\). By Lemma~\ref{lem:additive-fourier-stieltjes-uniqueness}, \(\mu=\mu_{n,\lambda}\). Thus the graph of \(\lambda\mapsto\mu_{n,\lambda}\) is closed. By the closed graph theorem, this map is bounded.

For every \(n\geq1\) and every continuous complex-valued function \(g\) on \(\widehat{K(Y_n,A)}\) with \(\|g\|_\infty\leq1\), consider the bounded linear functional
\[
\lambda
\longmapsto
\int_{\widehat{K(Y_n,A)}}g\,d\mu_{n,\lambda}.
\]
For every \(\lambda\in E^*\), Equation~\eqref{eq:approximating-measure-bound} gives
\[
\sup_{\substack{n\geq1\\\|g\|_\infty\leq1}}
\left|
\int_{\widehat{K(Y_n,A)}}g\,d\mu_{n,\lambda}
\right|
=
\sup_{n\geq1}
\|\mu_{n,\lambda}\|_{\mathrm{TV}}
<
\infty.
\]
By the uniform boundedness principle, these linear functionals have uniformly bounded norms. Taking the supremum over \(g\) proves Equation~\eqref{eq:uniform-approximating-measure-bound}.

\textbf{Proof of (3).}
Since pushforward is contractive for the total variation norm and this norm is weak-* lower semicontinuous,
\[
\|\nu_\lambda^Z\|_{\mathrm{TV}}
\leq
\liminf_{n\to\infty}
\|\nu_{n,\lambda}^Z\|_{\mathrm{TV}}
\leq
M\|\lambda\|.
\]
Consequently, the constant \(C\) in Equation~\eqref{eq:finite-restriction-uniform-operator-bound} satisfies \(C\leq M\). Equation~\eqref{eq:full-group-measure-operator-bound} therefore gives \(\|\nu_\lambda^X\|_{\mathrm{TV}}\leq M\|\lambda\|\), proving Equation~\eqref{eq:uniform-local-measure-bound}.
\end{proof}

\subsection{Recovery from Increments}

Let \(G\) be a discrete abelian group. We show that the increments \(\Delta_hu\), for every \(h\in G\), uniquely determine \(a\) and \(\mu\) in \(u=c+a+\widehat\mu\), provided \(\mu(\{\mathbf1\})=0\). The value \(u(0)\) then determines \(c\).

For a translation-invariant positive definite kernel, we recover the measure in Bochner's theorem from the inner products of increments of its RKHS feature map and the value \(k(0,0)\). 

Retaining only the magnitudes of the increments of the RKHS feature map would lead naturally to a phase retrieval problem. Phase retrieval is a branch of applied harmonic analysis concerned with recovering a signal \(u\) from \(|Tu|\), where \(T\) is a linear operator \cite[Section~3]{allain2025phasebook}.

\begin{corollary}\label{cor:increment-measure-uniqueness}
Let \(G\) be a discrete abelian group. Suppose that \(u=c+a+\widehat\mu\) and \(v=c'+a'+\widehat\nu\), where \(c,c'\in\mathbb C\), \(a,a'\in\operatorname{Hom}(G,\mathbb C)\), \(\mu,\nu\in M(\widehat G)\), and \(\mu(\{\mathbf1\})=\nu(\{\mathbf1\})=0\). If \(\Delta_hu=\Delta_hv\) for every \(h\in G\), then \(a=a'\), \(\mu=\nu\), and \(u-v\) is constant. Further, if \(u(0)=v(0)\), then \(u=v\).
\end{corollary}

\begin{proof}
Set \(b=a-a'\) and \(\sigma=\mu-\nu\). Equality of the increments gives
\begin{equation}\label{eq:increment-measure-uniqueness}
b(h)
+
\int_{\widehat G}
\chi(x)(\chi(h)-1)\,d\sigma(\chi)
=
0
\end{equation}
for every \(x,h\in G\).

We fix \(h\in G\). By Equation~\eqref{eq:increment-measure-uniqueness}, the Fourier--Stieltjes transform of \((\chi(h)-1)\sigma\) is the constant function \(-b(h)\). Since the Fourier--Stieltjes transform of \(-b(h)\delta_{\mathbf1}\) is the same constant function, injectivity of the Fourier--Stieltjes transform \cite[Corollary~4.34]{folland2016course} gives
\[
(\chi(h)-1)\sigma
=
-b(h)\delta_{\mathbf1}.
\]
Since \((\chi(h)-1)\sigma\) has zero mass at \(\mathbf1\), we obtain \(b(h)=0\). Since \(h\) was arbitrary, \(b=0\), so \(a=a'\).

Equation~\eqref{eq:increment-measure-uniqueness} now gives \((\chi(h)-1)\sigma=0\) for every \(h\in G\). We fix \(\chi_0\neq\mathbf1\) and choose \(h\in G\) such that \(\chi_0(h)\neq1\). By continuity, there are a neighborhood \(U\) of \(\chi_0\) and \(\delta>0\) such that \(|\chi(h)-1|\geq\delta\) for every \(\chi\in U\). Since the multiplier \(\chi\mapsto\chi(h)-1\) is bounded away from zero on \(U\), the identity \((\chi(h)-1)\sigma=0\) implies \(\sigma|_U=0\). We apply this argument to every \(\chi_0\neq\mathbf1\), so \(\operatorname{supp}(\sigma)\subseteq\{\mathbf1\}\). The normalization gives \(\sigma(\{\mathbf1\})=0\), hence \(\sigma=0\). Therefore \(\mu=\nu\), and \(u-v=c-c'\) is constant. If \(u(0)=v(0)\), then \(c=c'\), so \(u=v\).
\end{proof}

We next consider the same question for translation-invariant positive definite kernels. Let \(G\) be a discrete abelian group and \(\varphi\colon G\to\mathbb C\) a positive-definite function. Let \(\mu_\varphi\) be the finite positive measure associated with \(\varphi\) by Bochner's theorem. Define \(k(x,y)=\varphi(x-y)\) and \(\Phi_\varphi(x)=k(\cdot,x)\). For \(h\in G\), define the kernel
\[
k_h(x,y)
=
\left\langle
\Phi_\varphi(y+h)-\Phi_\varphi(y),
\Phi_\varphi(x+h)-\Phi_\varphi(x)
\right\rangle_{\mathcal H_k}.
\]

\begin{corollary}\label{cor:increment-kernel-spectral-recovery}
For every \(x,y,h\in G\),
\begin{equation}\label{eq:rkhs-increment-kernel-spectrum}
k_h(x,y)
=
\int_{\widehat G}
\chi(x-y)|\chi(h)-1|^2\,d\mu_\varphi(\chi).
\end{equation}
The family \(\{k_h:h\in G\}\) determines \(\mu_\varphi|_{\widehat G\setminus\{\mathbf1\}}\). Together with \(k(0,0)\), it determines \(\mu_\varphi\).
\end{corollary}

\begin{proof}
By the reproducing property and Bochner's theorem,
\[
\begin{aligned}
k_h(x,y)
&=
k(x+h,y+h)-k(x+h,y)-k(x,y+h)+k(x,y)\\
&=
\int_{\widehat G}
\chi(x-y)
\bigl(2-\chi(h)-\chi(-h)\bigr)
\,d\mu_\varphi(\chi).
\end{aligned}
\]
Since \(\chi(-h)=\overline{\chi(h)}\), we obtain Equation~\eqref{eq:rkhs-increment-kernel-spectrum}.

We fix \(h\in G\). The function \(x\mapsto k_h(x,0)\) determines the measure \(|\chi(h)-1|^2\mu_\varphi\) by Bochner's theorem. We fix \(\chi_0\neq\mathbf1\) and choose \(h\in G\) such that \(\chi_0(h)\neq1\). By continuity, there are a neighborhood \(U\) of \(\chi_0\) and \(\delta>0\) such that \(|\chi(h)-1|^2\geq\delta\) for every \(\chi\in U\). Since \(|\chi(h)-1|^2\) is bounded away from zero on \(U\), the weighted measure uniquely determines \(\mu_\varphi|_U\). We apply this argument to every \(\chi_0\neq\mathbf1\), so the family \(\{k_h:h\in G\}\) determines \(\mu_\varphi|_{\widehat G\setminus\{\mathbf1\}}\).

Finally, \(k(0,0)=\varphi(0)=\mu_\varphi(\widehat G)\). We therefore obtain \(\mu_\varphi(\{\mathbf1\})\) by subtracting \(\mu_\varphi(\widehat G\setminus\{\mathbf1\})\) from \(k(0,0)\). Hence \(k(0,0)\) together with the family \(\{k_h:h\in G\}\) determines \(\mu_\varphi\).
\end{proof}

\begin{figure}[t]
    \centering
    \includegraphics[width=0.72\textwidth]{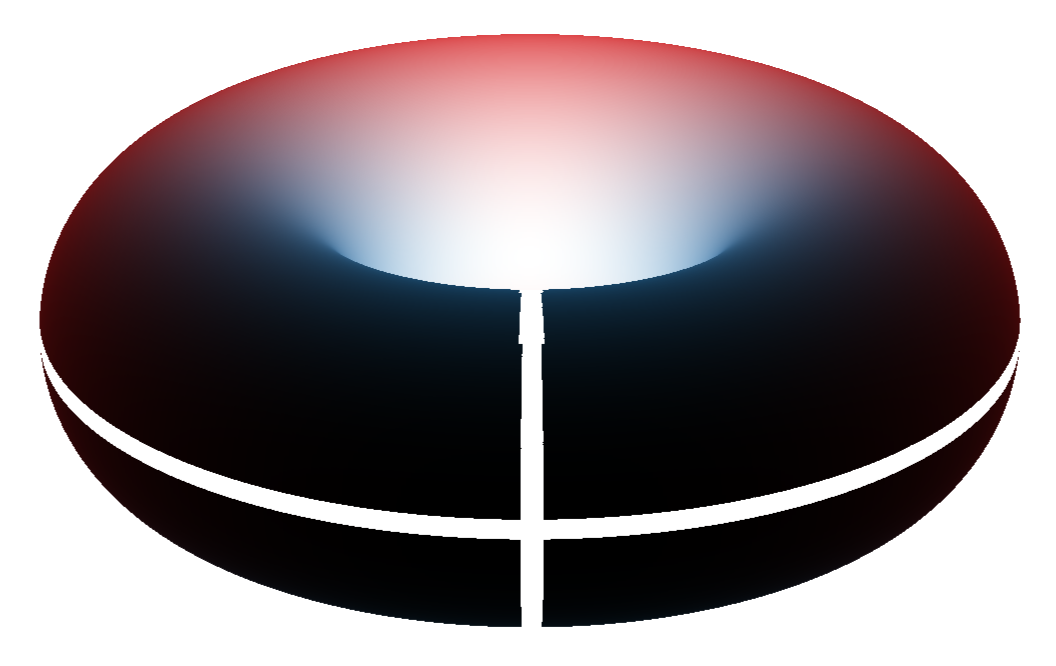}
    \caption{Increment multipliers on $\widehat{\mathbb Z^2}\cong\mathbb T^2$. The red and blue components show $|\chi(h_1)-1|^2$ and $|\chi(h_2)-1|^2$ for $h_1=(1,0)$ and $h_2=(0,1)$, respectively. The white curves mark their zero sets. These curves intersect only at the trivial character, illustrating why the two increments detect every nontrivial character.}
    \label{fig:increment-kernel-multipliers}
\end{figure}

\subsection{Finite Reconstruction and Approximation}

We now return to the compatible representations from Theorem~\ref{thm:uniformly-discrete-fourier-stieltjes-consistency}. Let \(A\subseteq Y\subseteq X\) be a uniformly discrete subpair and let \(F\subseteq K(Y,A)\) be finite and nonempty. By item~6 of that theorem, \(V_Y(f)\) is the pointwise closure of the span of the additive functions \(a_\lambda^Y\), the characters in \(\operatorname{supp}(\nu_\lambda^Y)\), and \(\mathbf1\) when some \(a_\lambda^Y\) is nonzero. Restriction to the finite set \(F\) therefore gives a spanning family for the finite-dimensional space \(T_F(V_Y(f))\). The next theorem reconstructs the values \(f(i_Y\eta)\), \(\eta\in F\), from evaluations at \(d=\dim T_F(V_Y(f))\) suitably chosen points.

\begin{theorem}\label{thm:finite-spectral-reconstruction}
Assume the hypotheses of Theorem~\ref{thm:uniformly-discrete-fourier-stieltjes-consistency}. Let \(A\subseteq Y\subseteq X\) be a uniformly discrete subpair, let \(F\subseteq K(Y,A)\) be finite and nonempty, and set \(U=T_F(V_Y(f))\) and \(d=\dim U\). Then the following statements hold.
\begin{enumerate}
\item If \(d=0\), then \(f(i_Y\eta)=0\) for every \(\eta\in F\).

\item If \(d\geq1\), there are functions \(\psi_1,\ldots,\psi_d\in V_Y(f)\) such that \(T_F\psi_1,\ldots,T_F\psi_d\) form a basis of \(U\). Each \(\psi_j\) is an additive function \(a_\lambda^Y\) for some \(\lambda\in E^*\), a character in \(\operatorname{supp}(\nu_\lambda^Y)\) for some \(\lambda\in E^*\), or \(\mathbf1\) when some \(a_\lambda^Y\) is nonzero.

\item If \(d\geq1\), there are \(\eta_1,\ldots,\eta_d\in F\) such that \(M=(\psi_j(\eta_i))_{i,j=1}^d\) is invertible. The unique vectors \(v_1,\ldots,v_d\in E\) satisfying
\begin{equation}\label{eq:finite-spectral-reconstruction}
M(v_1,\ldots,v_d)^\top
=
(f(i_Y\eta_1),\ldots,f(i_Y\eta_d))^\top
\end{equation}
satisfy
\begin{equation}\label{eq:finite-exponential-reconstruction}
f(i_Y\eta)
=
\sum_{j=1}^dv_j\psi_j(\eta)
\qquad
\eta\in F.
\end{equation}

\item The number \(d\) is the least cardinality of a subset \(S\subseteq F\) for which restriction from \(F\) to \(S\) is injective on \(U\).

\item If \(a_{n,\lambda}=0\) for every \(n\) and every \(\lambda\in E^*\), then we may choose every \(\psi_j\) as a character, with \(\psi_j\in\operatorname{supp}(\nu_{\lambda_j}^Y)\) for some \(\lambda_j\in E^*\).
\end{enumerate}
\end{theorem}

\begin{proof}
If \(d=0\), then \(T_F(\lambda\circ f\circ i_Y)=0\) for every \(\lambda\in E^*\). Since \(E^*\) separates points of \(E\), we have \(f(i_Y\eta)=0\) for every \(\eta\in F\).

Suppose that \(d\geq1\). By item~6 of Theorem~\ref{thm:uniformly-discrete-fourier-stieltjes-consistency}, \(V_Y(f)\) is the pointwise closure of the span of the additive functions \(a_\lambda^Y\), the characters in \(\operatorname{supp}(\nu_\lambda^Y)\), and \(\mathbf1\) when some \(a_\lambda^Y\) is nonzero. Restriction \(T_F\) is continuous for pointwise convergence, and every linear subspace of the finite-dimensional space \(\mathbb C^F\) is closed. Hence the restrictions of these functions span \(U\). We choose \(d\) such functions whose restrictions form a basis of \(U\).

The matrix with rows \((\psi_1(\eta),\ldots,\psi_d(\eta))\), indexed by \(\eta\in F\), has rank \(d\). We choose \(d\) linearly independent rows and denote the corresponding points by \(\eta_1,\ldots,\eta_d\). The resulting matrix \(M\) is invertible, so Equation~\eqref{eq:finite-spectral-reconstruction} determines \(v_1,\ldots,v_d\).

We fix \(\lambda\in E^*\). Both \(T_F(\lambda\circ f\circ i_Y)\) and \(T_F(\sum_{j=1}^d\lambda(v_j)\psi_j)\) belong to \(U\), and they agree at \(\eta_1,\ldots,\eta_d\). Since \(M\) is invertible, restriction to \(\{\eta_1,\ldots,\eta_d\}\) is injective on \(U\). Hence the two functions agree on \(F\). Since \(E^*\) separates points of \(E\), we obtain Equation~\eqref{eq:finite-exponential-reconstruction}.

The selected \(d\) points show that restriction to a subset of \(F\) of cardinality \(d\) can be injective on \(U\). Conversely, if \(S\subseteq F\) and restriction from \(F\) to \(S\) is injective on \(U\), then \(U\) embeds linearly into \(\mathbb C^S\), so \(d\leq|S|\). Hence \(d\) is the least possible cardinality.

If every approximating additive function vanishes, Corollary~\ref{cor:pure-fourier-stieltjes-finite-span} shows that restrictions of characters in the supports of the representing measures span \(U\). We may therefore choose the basis functions as stated in item~5.
\end{proof}

We next quantize elements whose generators lie in a totally bounded subset of \(X/A\) and whose reduced expressions have bounded total multiplicity by replacing their generators with points of a finite \(\varepsilon\)-net.

Let \(C\subseteq X/A\) be totally bounded. Let \(Z_C\subseteq X\) consist of \(A\) together with one representative in \(X\) of each non-basepoint element of \(C\). Every \(\xi\in K(Z_C,A)\) has a unique finitely supported integer coefficient function on \(Z_C/A\setminus\{[A]\}\). Its positive and negative parts determine a unique reduced expression \(\xi=\alpha-\beta\) with disjoint supports. For a positive integer \(R\), let \(\mathcal K_R(C)\) consist of the elements whose reduced expression satisfies \(N(\alpha)+N(\beta)\leq R\).

\begin{lemma}\label{lem:finite-diagram-quantization}
For every \(\varepsilon>0\), there are \(A\subseteq Y\subseteq Z_C\), with \(Y/A\) finite, and a homomorphism \(P\colon K(Z_C,A)\to K(Y,A)\) such that \(F=P(\mathcal K_R(C))\) is finite and
\begin{equation}\label{eq:finite-approximation-distance}
\sup_{\xi\in\mathcal K_R(C)}
\rho(i_{Z_C}\xi,i_YP\xi)
\leq
R\varepsilon.
\end{equation}
If \(m=|Y/A\setminus\{[A]\}|\), then
\begin{equation}\label{eq:finite-quantization-state-count}
|F|
\leq
\sum_{r=0}^{\min\{m,R\}}
2^r\binom mr\binom Rr.
\end{equation}
\end{lemma}

\begin{proof}
We choose a finite \(\varepsilon\)-net \(Q_0\subseteq C\), set \(Q=Q_0\cup\{[A]\}\), and choose \(q(c)\in Q\) with \(d_1(c,q(c))\leq\varepsilon\) for every \(c\in C\). For each non-basepoint \(q\in Q\), we choose \(y_q\in Z_C\) whose class is \(q\), and set \(Y=A\cup\{y_q:q\in Q\setminus\{[A]\}\}\).

We map a non-basepoint generator \([x]\) of \(K(Z_C,A)\) to \(0\) when \(q([x])=[A]\) and to \([y_{q([x])}]\) otherwise. Since \(K(Z_C,A)\) is free abelian on its non-basepoint classes, this assignment extends uniquely to a homomorphism \(P\). The same generator assignment induces a monoid homomorphism \(D(Z_C,A)\to D(Y,A)\).

Let \(\xi=\alpha-\beta\in\mathcal K_R(C)\) be reduced. Matching every occurrence in \(\alpha\) with its selected class gives \(W_1(i_YP\alpha,i_{Z_C}\alpha)\leq N(\alpha)\varepsilon\), and the same argument gives \(W_1(i_YP\beta,i_{Z_C}\beta)\leq N(\beta)\varepsilon\). Equation~\eqref{eq:groth-metric} and the triangle inequality therefore give
\[
\rho(i_{Z_C}\xi,i_YP\xi)
\leq
(N(\alpha)+N(\beta))\varepsilon
\leq
R\varepsilon.
\]
This proves Equation~\eqref{eq:finite-approximation-distance}.

Every reduced element of \(F\) has total absolute coefficient sum at most \(R\). Suppose that exactly \(r\) coefficients are nonzero. We choose their \(r\) support points in \(\binom mr\) ways and their signs in \(2^r\) ways. The number of choices for the positive integer absolute values of the \(r\) nonzero coefficients whose sum is at most \(R\) is \(\binom Rr\). Summing over \(r\) gives Equation~\eqref{eq:finite-quantization-state-count}.
\end{proof}

\begin{figure}[t]
    \centering
    \includegraphics[width=\textwidth]{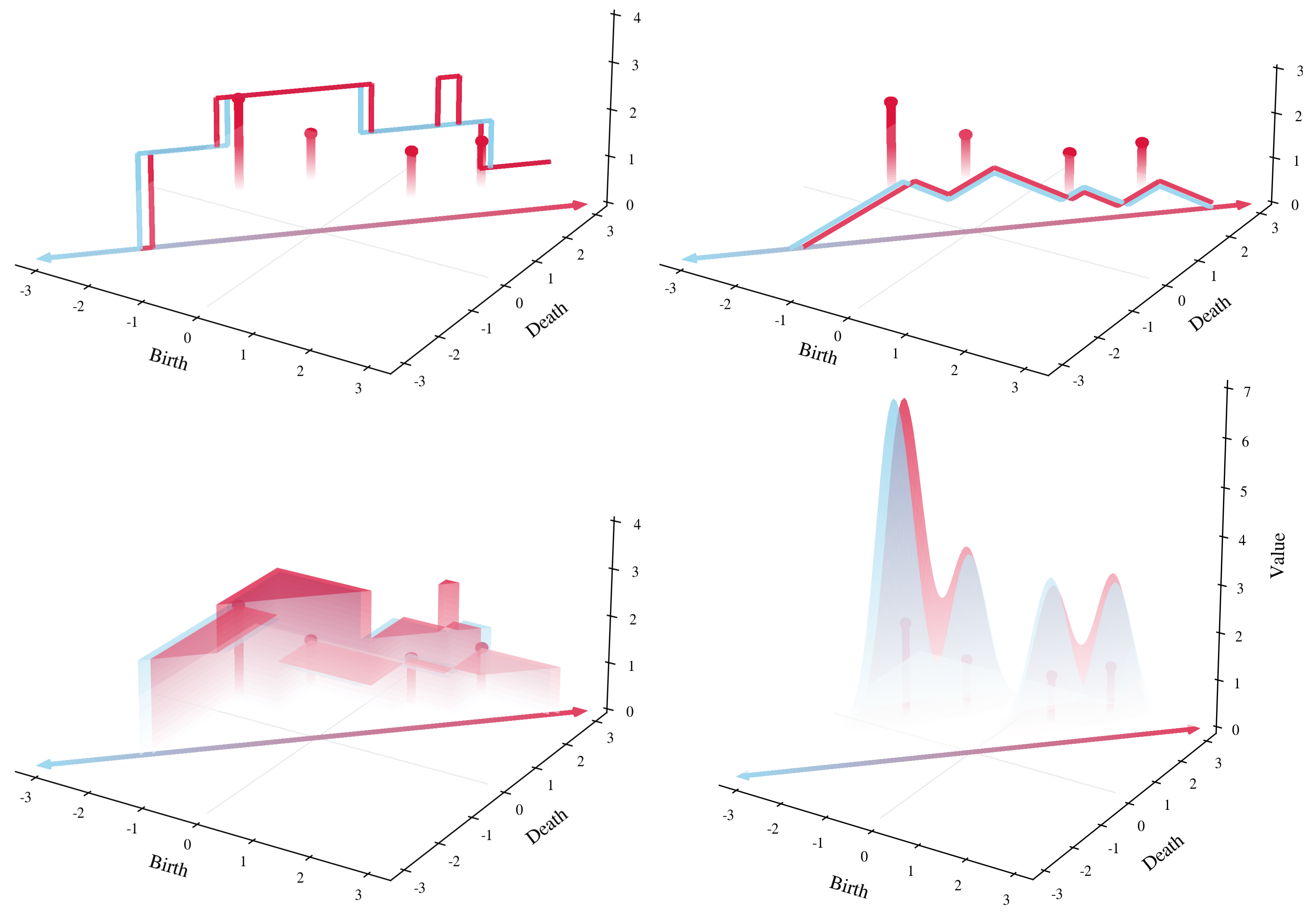}
    \caption{
    Finite quantization and approximation of an out-of-sample persistence diagram. The red pins show the original persistence diagram $\xi$, and the red surfaces show its vectorizations. The blue surfaces show the exact character expansions evaluated at the quantized virtual persistence diagram $P\xi$. The upper-left panel shows the Betti curve, the upper-right panel shows the first persistence landscape, the lower-left panel shows the rank function, and the lower-right panel shows the persistence image.
    }
    \label{fig:reconstruction}
\end{figure}

\begin{corollary}\label{cor:finite-quantization-approximation}
Assume the hypotheses of Theorem~\ref{thm:uniformly-discrete-fourier-stieltjes-consistency}, and suppose that \(f\) is Lipschitz. Let \(C\subseteq X/A\) be totally bounded, let \(R\) be a positive integer, let \(\varepsilon>0\), and let \(Y\), \(P\), and \(F=P(\mathcal K_R(C))\) satisfy Lemma~\ref{lem:finite-diagram-quantization}. Set \(d=\dim T_F(V_Y(f))\), and choose \(\psi_1,\ldots,\psi_d\) and \(v_1,\ldots,v_d\) according to Theorem~\ref{thm:finite-spectral-reconstruction}, with the empty sum when \(d=0\). Then
\begin{equation}\label{eq:finite-quantization-approximation}
\sup_{\xi\in\mathcal K_R(C)}
\left\|
f(i_{Z_C}\xi)
-
\sum_{j=1}^dv_j\psi_j(P\xi)
\right\|_E
\leq
\operatorname{Lip}(f)R\varepsilon.
\end{equation}
If \(a_{n,\lambda}=0\) for every \(n\) and every \(\lambda\in E^*\), then we may choose each \(\psi_j\) as a character in \(\operatorname{supp}(\nu_{\lambda_j}^Y)\) for some \(\lambda_j\in E^*\). Once these basis functions have been identified, Theorem~\ref{thm:finite-spectral-reconstruction} determines the coefficients using \(d\) evaluations of \(f\).
\end{corollary}

\begin{proof}
Since \(Y/A\) is finite, it is uniformly discrete. By Theorem~\ref{thm:finite-spectral-reconstruction}, \(f(i_YP\xi)=\sum_{j=1}^dv_j\psi_j(P\xi)\) for every \(\xi\in\mathcal K_R(C)\), with the empty sum when \(d=0\). Equation~\eqref{eq:finite-approximation-distance} and the Lipschitz property of \(f\) give Equation~\eqref{eq:finite-quantization-approximation}. The final statement follows from item~5 of Theorem~\ref{thm:finite-spectral-reconstruction}.
\end{proof}

Let \(N(C,\varepsilon)\) denote the least cardinality of a finite \(\varepsilon\)-net of \(C\). If we use a net of cardinality \(N(C,\varepsilon)\) in Lemma~\ref{lem:finite-diagram-quantization}, then \(m\leq N(C,\varepsilon)\), and
\begin{equation}\label{eq:finite-reconstruction-state-count}
d
\leq
|F|
\leq
\sum_{r=0}^{\min\{N(C,\varepsilon),R\}}
2^r
\binom{N(C,\varepsilon)}r
\binom Rr.
\end{equation}
Thus Equation~\eqref{eq:finite-reconstruction-state-count} bounds the number of evaluations used by the reconstruction in Theorem~\ref{thm:finite-spectral-reconstruction}, once the basis functions have been identified.

For every subpair \(A\subseteq Z\subseteq X\), we write \(\rho_Z\) for the metric on \(K(Z,A)\). For a fixed \(\xi\), enlarge \(Y\) by the points whose classes occur in the reduced expression of \(\xi\). Then \(\xi\) and the image of \(P\xi\) lie in the same abelian group, so compatibility of the representations allows us to subtract their additive and Fourier--Stieltjes representations there.

\begin{corollary}\label{cor:additive-fourier-stieltjes-approximation}
Assume the hypotheses and notation of Corollary~\ref{cor:finite-quantization-approximation}. For \(\xi\in\mathcal K_R(C)\), let \(Z_\xi\subseteq Z_C\) consist of \(Y\) together with the points whose classes occur in the reduced expression of \(\xi\). The inclusion \(K(Z_\xi,A)\to K(Z_C,A)\) is injective, and we identify \(K(Z_\xi,A)\) with its image. Thus \(\xi\in K(Z_\xi,A)\). Let \(j_{Z_\xi,Y}\colon K(Y,A)\to K(Z_\xi,A)\) denote inclusion. Then, for every \(\lambda\in E^*\) with \(\|\lambda\|\leq1\),
\begin{align}
\left|
\lambda\left(
f(i_{Z_C}\xi)
-
\sum_{j=1}^d v_j\psi_j(P\xi)
\right)
\right|
\notag
&\leq
\min\Bigg\{
\operatorname{Lip}(f)R\varepsilon,\,
\left|
a_\lambda^{Z_\xi}
\left(
\xi-j_{Z_\xi,Y}P\xi
\right)
\right|
\notag\\
&
+
\int_{\widehat{K(Z_\xi,A)}}
\left|
\chi\left(
\xi-j_{Z_\xi,Y}P\xi
\right)-1
\right|
\,d|\nu_\lambda^{Z_\xi}|(\chi)
\Bigg\}.
\label{eq:pointwise-additive-fourier-stieltjes-approximation}
\end{align}
\end{corollary}

\begin{proof}
By Theorem~\ref{thm:finite-spectral-reconstruction}, \(f(i_YP\xi)=\sum_{j=1}^dv_j\psi_j(P\xi)\). We fix \(\xi\in\mathcal K_R(C)\) and \(\lambda\in E^*\) with \(\|\lambda\|\leq1\). Since \(Z_\xi/A\) is finite, it is uniformly discrete. Under the identification in the statement, \(i_{Z_\xi}\xi=i_{Z_C}\xi\) and \(i_{Z_\xi}j_{Z_\xi,Y}P\xi=i_YP\xi\). By Equation~\eqref{eq:compatible-fourier-stieltjes-restriction}, the pair \(a_\lambda^{Z_\xi},\nu_\lambda^{Z_\xi}\) restricts to the pair on \(K(Y,A)\). Equation~\eqref{eq:compatible-fourier-stieltjes-representation} therefore gives
\[
\begin{aligned}
&\lambda\left(
f(i_{Z_C}\xi)-f(i_YP\xi)
\right)\\
&\quad=
a_\lambda^{Z_\xi}
\left(
\xi-j_{Z_\xi,Y}P\xi
\right)
+
\int_{\widehat{K(Z_\xi,A)}}
\chi(j_{Z_\xi,Y}P\xi)
\left[
\chi\left(
\xi-j_{Z_\xi,Y}P\xi
\right)-1
\right]
\,d\nu_\lambda^{Z_\xi}(\chi).
\end{aligned}
\]
Since \(|\chi(j_{Z_\xi,Y}P\xi)|=1\), the triangle inequality gives the final term in Equation~\eqref{eq:pointwise-additive-fourier-stieltjes-approximation}. Equation~\eqref{eq:finite-approximation-distance} and Lipschitz continuity also give \(\left|\lambda(f(i_{Z_C}\xi)-f(i_YP\xi))\right|\leq\operatorname{Lip}(f)R\varepsilon\). Combining the two bounds proves the result.
\end{proof}

Suppose now that \(a_{n,\lambda}=0\) for every \(n\) and every \(\lambda\in E^*\). Item~2 of Theorem~\ref{thm:uniformly-discrete-fourier-stieltjes-consistency} then gives \(a_\lambda^Z=0\) on every uniformly discrete subpair, so the right-hand side of Equation~\eqref{eq:pointwise-additive-fourier-stieltjes-approximation} depends only on \(\nu_\lambda^Z\).

\begin{corollary}\label{cor:quantitative-spectral-approximation}
Assume the hypotheses and notation of Corollary~\ref{cor:uniform-approximating-measure-bound}. Suppose additionally that \(f\) is Lipschitz and \(a_{n,\lambda}=0\) for every \(n\) and every \(\lambda\in E^*\). Let \(C\subseteq X/A\) be totally bounded, let \(R\) be a positive integer, let \(\varepsilon>0\), and let \(Y\), \(P\), and \(F=P(\mathcal K_R(C))\) satisfy Lemma~\ref{lem:finite-diagram-quantization}. Set \(d=\dim T_F(V_Y(f))\), and choose characters \(\chi_1,\ldots,\chi_d\) and vectors \(v_1,\ldots,v_d\) as in Corollary~\ref{cor:finite-quantization-approximation}. For every \(\xi\in\mathcal K_R(C)\), let \(Z_\xi\subseteq Z_C\) consist of \(Y\) together with the points whose classes occur in the reduced expression of \(\xi\). We identify \(K(Z_\xi,A)\) with its image in \(K(Z_C,A)\), so \(\xi\in K(Z_\xi,A)\), and let \(j_{Z_\xi,Y}\colon K(Y,A)\to K(Z_\xi,A)\) denote inclusion.

Let \(\omega\colon[0,\infty)\to[0,2]\) be nondecreasing with \(\omega(0)=0\). Suppose that
\begin{equation}\label{eq:quantization-spectral-bound}
\left|
\chi\left(
\xi-j_{Z_\xi,Y}P\xi
\right)-1
\right|
\leq
\omega\left(
\rho_{Z_\xi}
\left(
\xi,j_{Z_\xi,Y}P\xi
\right)
\right)
\end{equation}
for every \(\xi\in\mathcal K_R(C)\), every \(\lambda\in E^*\) with \(\|\lambda\|\leq1\), and every \(\chi\in\operatorname{supp}(\nu_\lambda^{Z_\xi})\). Then
\begin{equation}\label{eq:quantitative-spectral-approximation}
\sup_{\xi\in\mathcal K_R(C)}
\left\|
f(i_{Z_C}\xi)
-
\sum_{j=1}^dv_j\chi_j(P\xi)
\right\|_E
\leq
\min\{
\operatorname{Lip}(f)R\varepsilon,
M\omega(R\varepsilon)
\},
\end{equation}
where \(M\) is defined in Equation~\eqref{eq:uniform-approximating-measure-bound}.
\end{corollary}

\begin{proof}
Since \(a_{n,\lambda}=0\), we have \(a_{n,\lambda}^{Z_\xi}=a_{n,\lambda}\circ P_n^{Z_\xi}=0\) for every \(n\). Item~2 of Theorem~\ref{thm:uniformly-discrete-fourier-stieltjes-consistency} therefore gives \(a_\lambda^{Z_\xi}=0\). Hence Corollary~\ref{cor:additive-fourier-stieltjes-approximation} gives
\[
\left|
\lambda\left(
f(i_{Z_C}\xi)-f(i_YP\xi)
\right)
\right|
\leq
\int_{\widehat{K(Z_\xi,A)}}
\left|
\chi\left(
\xi-j_{Z_\xi,Y}P\xi
\right)-1
\right|
\,d|\nu_\lambda^{Z_\xi}|(\chi)
\]
for every \(\xi\in\mathcal K_R(C)\) and every \(\lambda\in E^*\) with \(\|\lambda\|\leq1\). The matching used in Lemma~\ref{lem:finite-diagram-quantization} uses only classes in \(Z_\xi/A\), so \(\rho_{Z_\xi}(\xi,j_{Z_\xi,Y}P\xi)\leq R\varepsilon\). By Equation~\eqref{eq:quantization-spectral-bound}, monotonicity of \(\omega\), and Equation~\eqref{eq:uniform-local-measure-bound},
\[
\int_{\widehat{K(Z_\xi,A)}}
\left|
\chi\left(
\xi-j_{Z_\xi,Y}P\xi
\right)-1
\right|
\,d|\nu_\lambda^{Z_\xi}|(\chi)
\leq
M\omega(R\varepsilon).
\]
Lipschitz continuity also gives the bound \(\operatorname{Lip}(f)R\varepsilon\). Taking the supremum over the unit ball of \(E^*\) and then over \(\xi\in\mathcal K_R(C)\) proves Equation~\eqref{eq:quantitative-spectral-approximation}.
\end{proof}

Equation~\eqref{eq:quantization-spectral-bound} follows, for example, if the characters in \(\operatorname{supp}(\nu_\lambda^{Z_\xi})\) satisfy a uniform Lipschitz bound with respect to \(d_1\). Suppose that there is \(c>0\) such that \(|\chi([p]-[q])-1|\leq c\,d_1(p,q)\) for every \(\xi\in\mathcal K_R(C)\), every \(\lambda\in E^*\) with \(\|\lambda\|\leq1\), every \(\chi\in\operatorname{supp}(\nu_\lambda^{Z_\xi})\), and every \(p,q\in Z_\xi\). Let \(\sigma=\sum_{k=1}^m(p_k,q_k)\) be a matching between the two persistence diagrams whose \(1\)-Wasserstein distance appears in Equation~\eqref{eq:groth-metric} for \(\rho_{Z_\xi}(\xi,j_{Z_\xi,Y}P\xi)\). Then
\[
\xi-j_{Z_\xi,Y}P\xi
=
\sum_{k=1}^m([p_k]-[q_k]).
\]
Repeated application of \(|zw-1|\leq|z-1|+|w-1|\) for \(z,w\in\mathbb T\) gives
\[
\left|
\chi\left(
\xi-j_{Z_\xi,Y}P\xi
\right)-1
\right|
\leq
c\sum_{k=1}^m d_1(p_k,q_k).
\]
Taking the infimum over all such matchings gives
\[
\left|
\chi\left(
\xi-j_{Z_\xi,Y}P\xi
\right)-1
\right|
\leq
c\,
\rho_{Z_\xi}
\left(
\xi,j_{Z_\xi,Y}P\xi
\right).
\]
Since every character takes values in \(\mathbb T\), we also have \(|\chi(\eta)-1|\leq2\) for every \(\eta\in K(Z_\xi,A)\). Therefore Equation~\eqref{eq:quantization-spectral-bound} holds with \(\omega(t)=\min\{2,ct\}\).

We now combine minimum-cardinality nets with Equation~\eqref{eq:finite-reconstruction-state-count} to bound the number of evaluations required for a prescribed approximation error, assuming that the characters whose restrictions form the chosen basis have already been identified.

\begin{corollary}\label{cor:covering-number-evaluation-bound}
Assume the hypotheses and notation of Corollary~\ref{cor:uniform-approximating-measure-bound}. Suppose that \(f\) is Lipschitz and \(a_{n,\lambda}=0\) for every \(n\) and every \(\lambda\in E^*\). Let \(C\subseteq X/A\) be totally bounded and let \(R\) be a positive integer. Let \(\omega\colon[0,\infty)\to[0,2]\) be nondecreasing with \(\omega(0)=0\), and let \(\varepsilon_*>0\).

For every \(0<\varepsilon\leq\varepsilon_*\), choose an \(\varepsilon\)-net \(Q_\varepsilon\subseteq C\) with \(|Q_\varepsilon|=N(C,\varepsilon)\), and construct \(Y_\varepsilon\), \(P_\varepsilon\), and \(F_\varepsilon=P_\varepsilon(\mathcal K_R(C))\) from \(Q_\varepsilon\) as in Lemma~\ref{lem:finite-diagram-quantization}. For every \(\xi\in\mathcal K_R(C)\), let \(Z_{\xi,\varepsilon}\subseteq Z_C\) consist of \(Y_\varepsilon\) together with the points whose classes occur in the reduced expression of \(\xi\). We identify \(K(Z_{\xi,\varepsilon},A)\) with its image in \(K(Z_C,A)\) and let \(j_{Z_{\xi,\varepsilon},Y_\varepsilon}\colon K(Y_\varepsilon,A)\to K(Z_{\xi,\varepsilon},A)\) denote inclusion. Suppose that
\[
\left|
\chi\left(
\xi-j_{Z_{\xi,\varepsilon},Y_\varepsilon}P_\varepsilon\xi
\right)-1
\right|
\leq
\omega\left(
\rho_{Z_{\xi,\varepsilon}}
\left(
\xi,j_{Z_{\xi,\varepsilon},Y_\varepsilon}P_\varepsilon\xi
\right)
\right)
\]
for every \(0<\varepsilon\leq\varepsilon_*\), every \(\xi\in\mathcal K_R(C)\), every \(\lambda\in E^*\) with \(\|\lambda\|\leq1\), and every \(\chi\in\operatorname{supp}(\nu_\lambda^{Z_{\xi,\varepsilon}})\). Set \(d_\varepsilon=\dim T_{F_\varepsilon}(V_{Y_\varepsilon}(f))\).

If \(\delta>0\) and \(t>0\) satisfy \(\min\{\operatorname{Lip}(f)t,M\omega(t)\}\leq\delta\) and \(t/R\leq\varepsilon_*\), set \(\varepsilon=t/R\) and choose characters \(\chi_{1,\varepsilon},\ldots,\chi_{d_\varepsilon,\varepsilon}\) and vectors \(v_{1,\varepsilon},\ldots,v_{d_\varepsilon,\varepsilon}\) as in Corollary~\ref{cor:quantitative-spectral-approximation}. Then
\[
\sup_{\xi\in\mathcal K_R(C)}
\left\|
f(i_{Z_C}\xi)
-
\sum_{j=1}^{d_\varepsilon}
v_{j,\varepsilon}
\chi_{j,\varepsilon}(P_\varepsilon\xi)
\right\|_E
\leq
\delta.
\]
Once the characters whose restrictions form the chosen basis have been identified, Theorem~\ref{thm:finite-spectral-reconstruction} determines the coefficients using \(d_{t/R}\) evaluations of \(f\), and
\begin{equation}\label{eq:covering-number-evaluation-bound}
d_{t/R}
\leq
\sum_{r=0}^{\min\{N(C,t/R),R\}}
2^r
\binom{N(C,t/R)}r
\binom Rr.
\end{equation}

If there are \(c_0,s>0\) and \(\varepsilon_0>0\) such that \(N(C,\varepsilon)\leq c_0\varepsilon^{-s}\) for every \(0<\varepsilon\leq\varepsilon_0\), then, for fixed \(R\),
\begin{equation}\label{eq:covering-number-restriction-rate}
d_\varepsilon
=
O(\varepsilon^{-sR})
\qquad
(\varepsilon\downarrow0).
\end{equation}
If \(\operatorname{Lip}(f)>0\), then for sufficiently small \(\delta>0\), the choice \(t=\delta/\operatorname{Lip}(f)\) gives
\begin{equation}\label{eq:lipschitz-evaluation-rate}
d_{\delta/(R\operatorname{Lip}(f))}
=
O(\delta^{-sR})
\qquad
(\delta\downarrow0).
\end{equation}

The bound \(M\omega(t)\leq\delta\) may permit a larger value of \(t\) than the Lipschitz bound \(\operatorname{Lip}(f)t\leq\delta\), and hence a coarser net. More precisely, if \(t'>t\), \(t'/R\leq\varepsilon_*\), and \(M\omega(t')\leq\delta\), then replacing \(t\) by \(t'\) weakly decreases the right-hand side of Equation~\eqref{eq:covering-number-evaluation-bound}.
\end{corollary}

\begin{proof}
We choose \(\varepsilon=t/R\). By Corollary~\ref{cor:quantitative-spectral-approximation}, the approximation error is at most \(\min\{\operatorname{Lip}(f)t,M\omega(t)\}\leq\delta\). Once the characters whose restrictions form the chosen basis have been identified, Theorem~\ref{thm:finite-spectral-reconstruction} determines the coefficients using \(d_{t/R}\) evaluations of \(f\). Since \(|Q_\varepsilon|=N(C,\varepsilon)\), Equation~\eqref{eq:finite-reconstruction-state-count} gives Equation~\eqref{eq:covering-number-evaluation-bound}.

If \(N(C,\varepsilon)\leq c_0\varepsilon^{-s}\), then
\[
\begin{aligned}
d_\varepsilon
&\leq
\sum_{r=0}^{\min\{N(C,\varepsilon),R\}}
2^r
\binom{N(C,\varepsilon)}r
\binom Rr\\
&=
O\left((1+N(C,\varepsilon))^R\right)
=
O(\varepsilon^{-sR}).
\end{aligned}
\]
Hence Equation~\eqref{eq:covering-number-restriction-rate} holds. If \(\operatorname{Lip}(f)>0\), we substitute \(\varepsilon=\delta/(R\operatorname{Lip}(f))\) and obtain Equation~\eqref{eq:lipschitz-evaluation-rate} for sufficiently small \(\delta\).

Finally, \(N(C,\varepsilon)\) is nonincreasing in \(\varepsilon\). Hence increasing an admissible \(t\) increases the approximation scale \(\varepsilon=t/R\) and weakly decreases the right-hand side of Equation~\eqref{eq:covering-number-evaluation-bound}. This proves the final assertion.
\end{proof}

\section{Classification of Persistence Diagram Vectorizations and Kernels}
\label{sec:examples}

We classify persistence diagram vectorizations and persistence kernels by whether they have uniformly discrete spectral synthesis. In every example below, whenever a Lipschitz extension exists, one exists with uniformly discrete spectral synthesis. This suggests the following conjecture.

\begin{conjecture}\label{conj:lipschitz-extension-spectral-synthesis}
Let \((X,d,A)\) be a metric pair with \(X/A\) separable, let \(E\) be a complex Banach space, and let \(f\colon K(X,A)\to E\) be Lipschitz. Then \(f\) has uniformly discrete spectral synthesis as in Definition~\ref{def:vectorization-spectral-synthesis}.
\end{conjecture}

The \texttt{Lipschitz Extension} column shows whether a map has a Lipschitz extension to \(K(X,A)\), and the \texttt{Spectral Synthesis} column shows whether it has such an extension with uniformly discrete spectral synthesis. For classes of vectorizations, \cmark\ means that every Lipschitz member has the corresponding extension. We use \xmark\ if some Lipschitz member has no such extension or if the class has no Lipschitz members.

\subsection{Vectorizations}
\label{sec:vectorizations}

\begin{table}[htbp]
\centering
\small
\setlength{\tabcolsep}{4pt}
\begin{tabular}{llll}
\hline
Vectorization
&
Codomain
&
\begin{tabular}{c}
Lipschitz\\
Extension
\end{tabular}
&
\begin{tabular}{c}
Spectral\\
Synthesis
\end{tabular}
\\
\hline

Lipschitz-free embedding\textsuperscript{a}
&
\(\mathcal F(K(X,A))\)
&
\cmark
&
\cmark
\\

Additive functions\textsuperscript{b}
&
\(\mathbb C\)
&
\cmark
&
\cmark
\\

Characters\textsuperscript{b}
&
\(\mathbb C\)
&
\cmark
&
\cmark
\\

Persistence statistics\textsuperscript{c}
&
\(\mathbb R^{38}\)
&
\xmark
&
\xmark
\\

Persistent entropy\textsuperscript{d}
&
\(\mathbb R\)
&
\xmark
&
\xmark
\\

Entropy summary function, \(p=1\)\textsuperscript{e}
&
\(L^1(\mathbb R)\)
&
\cmark
&
\cmark
\\

\begin{tabular}[c]{@{}l@{}}
Entropy summary function,\\
\(1<p\leq\infty\)\textsuperscript{e}
\end{tabular}
&
\(L^p(\mathbb R)\)
&
\xmark
&
\xmark
\\

Betti curve\textsuperscript{f}
&
\(L^1(\mathbb R)\)
&
\cmark
&
\cmark
\\

\begin{tabular}[c]{@{}l@{}}
Lifespan curve,\\
\(1\leq p\leq\infty\)\textsuperscript{g}
\end{tabular}
&
\(L^p(\mathbb R)\)
&
\xmark
&
\xmark
\\

\begin{tabular}[c]{@{}l@{}}
Persistence landscape,\\
\(1\leq p<\infty\)\textsuperscript{h}
\end{tabular}
&
\(L^p(\mathbb N\times\mathbb R)\)
&
\xmark
&
\xmark
\\

Persistence landscape, \(p=\infty\)\textsuperscript{h}
&
\(L^\infty(\mathbb N\times\mathbb R)\)
&
\cmark
&
\cmark
\\

\begin{tabular}[c]{@{}l@{}}
Persistence silhouette,\\
\(1\leq p<\infty\)\textsuperscript{i}
\end{tabular}
&
\(L^p(\mathbb R)\)
&
\xmark
&
\xmark
\\

Persistence silhouette, \(p=\infty\)\textsuperscript{i}
&
\(L^\infty(\mathbb R)\)
&
\cmark
&
\cmark
\\

Persistence image\textsuperscript{j}
&
\(\mathbb R^N\)
&
\cmark
&
\cmark
\\

Template functions\textsuperscript{k}
&
\(\mathbb R^n\)
&
\cmark
&
\cmark
\\

Tropical coordinates\textsuperscript{l}
&
\(\mathbb R^7\)
&
\xmark
&
\xmark
\\

Schauder basis\textsuperscript{m}
&
\(\ell^1\)
&
\cmark
&
\cmark
\\

\begin{tabular}[c]{@{}l@{}}
Truncated signature of an\\
integrated persistence landscape, \(N,M\geq1\)\textsuperscript{n}
\end{tabular}
&
\(\displaystyle\bigoplus_{m=0}^{M}(\mathbb C^N)^{\otimes m}\)
&
\xmark
&
\xmark
\\

\begin{tabular}[c]{@{}l@{}}
Linear functionals of truncated signatures\\
of integrated persistence landscapes
\textsuperscript{n,\(\dagger\)}
\end{tabular}
&
\begin{tabular}[c]{@{}l@{}}
finite-dimensional complex\\
normed space
\end{tabular}
&
\cmark
&
\cmark
\\

\hline
\end{tabular}

\begin{minipage}{\linewidth}
\footnotesize
\textsuperscript{a}\cite{bubenik2022virtual}. Here \(\mathcal F(K(X,A))\) denotes the Lipschitz-free space over \((K(X,A),0)\).
\textsuperscript{b}\cite{szekelyhidi2015spectral}.
\textsuperscript{c}\cite[Definition~3.1]{ali2023survey}, \cite{asaad2022persistent,chung2022persistence,pun2018persistent}.
\textsuperscript{d}\cite{chintakunta2015entropy,atienza2020stability}.
\textsuperscript{e}\cite{atienza2020stability}.
\textsuperscript{f}\cite{berry2018functional}.
\textsuperscript{g}\cite{asaad2022persistent}.
\textsuperscript{h}\cite{bubenik2015statistical}.
\textsuperscript{i}\cite{chazal2014stochastic}.
\textsuperscript{j}\cite{adams2017persistence}.
\textsuperscript{k}\cite{perea2023approximating}.
\textsuperscript{l}\cite[Definition~3.4 and the subsequent seven coordinates]{ali2023survey}, \cite[Section~8]{kalivsnik2019tropical}.
\textsuperscript{m}\cite[Definition~5.1 and Theorem~5.2]{bubenik2025schauder}. For this classification, we assume the hypotheses of Proposition~\ref{prop:schauder-basis-synthesis}: \((X,A)\) is a polyhedral pair, \(B\) is a totally ordered Schauder basis of \((\operatorname{Lip}_c(X,A),\ell^\infty)\), and \(B\) satisfies the locally Lipschitz finite property.
\textsuperscript{n}\cite{chevyrev2018persistence,giusti2023signatures}.

\textsuperscript{\(\dagger\)}
Proposition~\ref{prop:integrated-landscape-signature-classification} shows that only the constant functions are Lipschitz among the functions obtained by applying linear functionals to truncated signatures of integrated persistence landscapes. Every such function has a constant Lipschitz extension with uniformly discrete spectral synthesis.
\end{minipage}

\caption{Persistence diagram vectorizations classified by whether they have uniformly discrete spectral synthesis.}
\label{tab:persistence-vectorizations}
\end{table}

The proofs of the classifications in Table~\ref{tab:persistence-vectorizations} are in Appendix~\ref{sec:vectorization-synthesis-derivations}.

\subsection{Learnable Vectorizations}
\label{sec:learnable-vectorizations}

\begin{table}[!htbp]
\centering
\small
\begin{tabular}{lllll}
\hline
Vectorization
&
Variant
&
Codomain
&
\begin{tabular}{c}
Lipschitz\\
Extension
\end{tabular}
&
\begin{tabular}{c}
Spectral\\
Synthesis
\end{tabular}
\\
\hline

\multirow{5}{*}{
\begin{tabular}[c]{@{}l@{}}
Persistence\\
codebooks\textsuperscript{a}
\end{tabular}}
&
PBoW\textsuperscript{\(\ast\)}
&
\(\mathbb R^N\)
&
\xmark
&
\xmark
\\

&
sPBoW
&
\(\mathbb R^N\)
&
\xmark
&
\xmark
\\

&
PVLAD
&
\(\mathbb R^{2N}\)
&
\xmark
&
\xmark
\\

&
sPVLAD
&
\(\mathbb R^{2N}\)
&
\xmark
&
\xmark
\\

&
PFV\textsuperscript{\(\dagger\)}
&
\(\mathbb R^{4N}\)
&
\xmark
&
\xmark
\\

\hline

ATOL\textsuperscript{b,\(\ddagger\)}
&
&
\(\mathbb R^b\)
&
\xmark
&
\xmark
\\

\hline
\end{tabular}

\begin{minipage}{\linewidth}
\footnotesize
\textsuperscript{a}\cite{zielinski2021persistence}.
\textsuperscript{b}\cite{royer2021atol}.

\textsuperscript{\(\ast\)}
For the PBoW classification, we assume that the closures of the regions in which two different codewords are nearest intersect at an off-diagonal point, as in Proposition~\ref{prop:pbow-no-lipschitz-subclass}.

\textsuperscript{\(\dagger\)}
For the PFV classification, we assume that the diagonal Fisher-information entries appearing in Proposition~\ref{prop:pfv-no-lipschitz-subclass} are finite and strictly positive.

\textsuperscript{\(\ddagger\)}
For the ATOL classification, we assume \(\sigma_i(\widehat c_n)>0\) for every codeword, as in Proposition~\ref{prop:atol-no-lipschitz-subclass}.
\end{minipage}

\caption{Persistence codebook and ATOL vectorizations classified by whether they have uniformly discrete spectral synthesis.}
\label{tab:persistence-codebooks}
\end{table}

The proofs of the classifications in this subsection are in Appendix~\ref{sec:learnable-vectorization-synthesis-derivations}.

The stability results of Zieli\'nski et al.~\cite[Sections~3.4, 3.6--3.7]{zielinski2021persistence} use assumptions on matchings between persistence diagrams that differ from those required for Lipschitz continuity here. Their results for sPVLAD and PFV assume that the persistence diagrams lie in a fixed compact subset of \(\mathbb R^2\) and do not establish Lipschitz continuity under the hypotheses of our classifications. 

Similarly, Proposition~\ref{prop:atol-no-lipschitz-subclass} shows that ATOL vectorizations are not Lipschitz with respect to the \(1\)-Wasserstein distance on persistence diagrams.

\begin{table}[!htbp]
\centering
\small
\setlength{\tabcolsep}{4pt}
\begin{tabular}{llllll}
\hline
Vectorization
&
\begin{tabular}{c}
Point\\
Transformation
\end{tabular}
&
\begin{tabular}{c}
Permutation Invariant\\
Operation
\end{tabular}
&
Codomain
&
\begin{tabular}{c}
Lipschitz\\
Extension
\end{tabular}
&
\begin{tabular}{c}
Spectral\\
Synthesis
\end{tabular}
\\
\hline

\multirow{12}{*}{PersLay\textsuperscript{a}}
&
\multirow{4}{*}{\(\phi_\Lambda\)}
&
Minimum
&
\(\mathbb R^q\)
&
\cmark
&
\cmark
\\

&
&
Maximum
&
\(\mathbb R^q\)
&
\cmark
&
\cmark
\\

&
&
Sum
&
\(\mathbb R^q\)
&
\cmark
&
\cmark
\\

&
&
\(k\)-th largest
&
\(\mathbb R^q\)
&
\cmark
&
\cmark
\\

\cmidrule(lr){2-6}

&
\multirow{4}{*}{\(\phi_\Gamma\)}
&
Minimum
&
\(\mathbb R^q\)
&
\cmark
&
\cmark
\\

&
&
Maximum
&
\(\mathbb R^q\)
&
\cmark
&
\cmark
\\

&
&
Sum
&
\(\mathbb R^q\)
&
\cmark
&
\cmark
\\

&
&
\(k\)-th largest
&
\(\mathbb R^q\)
&
\cmark
&
\cmark
\\

\cmidrule(lr){2-6}

&
\multirow{4}{*}{\(\phi_L\)}
&
Minimum
&
\(\mathbb R^q\)
&
\cmark
&
\cmark
\\

&
&
Maximum
&
\(\mathbb R^q\)
&
\cmark
&
\cmark
\\

&
&
Sum
&
\(\mathbb R^q\)
&
\cmark
&
\cmark
\\

&
&
\(k\)-th largest
&
\(\mathbb R^q\)
&
\cmark
&
\cmark
\\

\hline
\end{tabular}

\begin{minipage}{\linewidth}
\footnotesize
\textsuperscript{a}\cite{Carrire2019PersLayAN}.
\end{minipage}

\caption{PersLay vectorizations classified by whether they have uniformly discrete spectral synthesis, with the conventions of Appendix~\ref{subsec:perslay}.}
\label{tab:perslay-vectorizations}
\end{table}

Theorem~\ref{thm:perslay-synthesis} in Appendix~\ref{subsec:perslay} extends this classification to Lipschitz PersLay layers with the operations specified there, beyond the point transformations in Table~\ref{tab:perslay-vectorizations}. These layers are defined in \cite[Equation~(4)]{Carrire2019PersLayAN}.

\subsection{Persistence Kernels}

The persistence kernels in Table~\ref{tab:persistence-kernels} have different classifications. In particular, the persistence weighted Gaussian kernel has uniformly discrete spectral synthesis for \(p\geq1\), but not for \(0<p<1\). For the sliced Wasserstein and Persistence Fisher kernels, we use the feature maps into their reproducing kernel Hilbert spaces. We prove the classifications in Table~\ref{tab:persistence-kernels} in Appendix~\ref{sec:kernel-synthesis-derivations}.

\begin{table}[htbp]
\centering
\small
\setlength{\tabcolsep}{5pt}
\renewcommand{\arraystretch}{1.15}

\begin{tabular}{llcc}
\toprule
Kernel
&
Codomain
&
\begin{tabular}{c}
Lipschitz\\
Extension
\end{tabular}
&
\begin{tabular}{c}
Spectral\\
Synthesis
\end{tabular}
\\
\midrule

Sliced Wasserstein\textsuperscript{a}
&
\(\mathcal H_{k_{SW}^{\sigma}}\)
&
\xmark
&
\xmark
\\

Persistence Fisher\textsuperscript{b}
&
\(\mathcal H_{k_{PF}^{t,\sigma}}\)
&
\xmark
&
\xmark
\\
\begin{tabular}[c]{@{}l@{}}
Persistence weighted \\
Gaussian kernel, \(0<p<1\)\textsuperscript{c}
\end{tabular}
&
\(\mathcal H_{k_G}\)
&
\xmark
&
\xmark
\\
\begin{tabular}[c]{@{}l@{}}
Persistence weighted \\
Gaussian kernel, \(p\geq 1\)\textsuperscript{c}
\end{tabular}
&
\(\mathcal H_{k_G}\)
&
\cmark
&
\cmark
\\

Persistence scale-space kernel\textsuperscript{d}
&
\(L^2(\Omega)\)
&
\cmark
&
\cmark
\\

\bottomrule
\end{tabular}

\begin{minipage}{\linewidth}
\footnotesize
\textsuperscript{a}\cite{carriere2017sliced}. The classification holds for every \(\sigma>0\).
\textsuperscript{b}\cite{le2018persistence}. The classification holds for \(\Theta=\mathbb R^2\) and every \(t,\sigma>0\).
\textsuperscript{c}\cite{kusano2016persistence}. The classifications hold for every \(C,\sigma>0\).
\textsuperscript{d}\cite{reininghaus2015stable}. The classification holds for every \(\sigma>0\).
\end{minipage}

\caption{Persistence kernels classified by whether they have uniformly discrete spectral synthesis.}
\label{tab:persistence-kernels}
\end{table}

\FloatBarrier

\section{Conclusion}\label{sec13}

In this paper, we introduced and established sufficient conditions for the uniformly discrete spectral synthesis of persistence diagram vectorizations. We also obtained error bounds for approximating Lipschitz vectorizations on virtual persistence diagrams with bounded total multiplicity and persistence intervals in totally bounded sets. We then classified persistence diagram vectorizations and persistence kernels by whether they have uniformly discrete spectral synthesis. In every classified example, whenever a Lipschitz extension exists, one exists with uniformly discrete spectral synthesis. For separable \(X/A\), we conjecture that every Lipschitz map \(K(X,A)\to E\) into a complex Banach space has uniformly discrete spectral synthesis (Conjecture~\ref{conj:lipschitz-extension-spectral-synthesis}). The main limitations are that we have only established sufficient conditions for uniformly discrete spectral synthesis. Future work should establish necessary and sufficient conditions for uniformly discrete spectral synthesis and develop an algorithm which computes the additive functions and characters required for an explicit approximation algorithm.

\section*{Declarations}

\begin{itemize}
\item \textbf{Competing Interests} The authors declare that they have no competing interests.
\item \textbf{Funding} This research received no external funding.
\item \textbf{Data Availability} Not applicable.
\item \textbf{Authors' Contributions}
C.F. developed the theoretical framework, proved the main results, and wrote the manuscript. M.E.A. advised the project and provided feedback on the framework and manuscript.
\item \textbf{AI Usage} The authors used an AI language model to rephrase dense analytic descriptions, to discuss and check mathematical arguments, and to proofread and correct the language and grammar during manuscript preparation. The authors assume responsibility for all content in this publication.
\end{itemize}

\bibliography{sn-bibliography}

\begin{appendices}

\section{Derivations for the Spectral Synthesis of Vectorizations}
\label{sec:vectorization-synthesis-derivations}


The tables in Section~\ref{sec:examples} use the following conventions:
\begin{enumerate}
\item For an individual persistence diagram vectorization \(f\colon D(X,A)\to E\), the Lipschitz Extension column contains \cmark\ precisely when there exists a Lipschitz map \(\overline f\colon K(X,A)\to E\) such that \(\overline f|_{D(X,A)}=f\). The column contains \xmark\ precisely when no such map exists.

\item For an individual persistence diagram vectorization \(f\colon D(X,A)\to E\), the Spectral Synthesis column contains \cmark\ precisely when \(f\) has a Lipschitz extension with uniformly discrete spectral synthesis in the sense of Definition~\ref{def:vectorization-spectral-synthesis}. The column contains \xmark\ precisely when no such extension exists.

\item For a class containing at least one persistence diagram vectorization that is Lipschitz with respect to \(W_1\), we restrict attention to its Lipschitz members. The Lipschitz Extension column contains \cmark\ precisely when every such member has a Lipschitz extension to \(K(X,A)\). The Spectral Synthesis column contains \cmark\ precisely when every such member has a Lipschitz extension with uniformly discrete spectral synthesis.

\item For a class of persistence diagram vectorizations, the Lipschitz Extension column contains \xmark\ precisely when either the class contains no member that is Lipschitz with respect to \(W_1\) or some Lipschitz member has no Lipschitz extension to \(K(X,A)\). The Spectral Synthesis column contains \xmark\ precisely when either the class contains no member that is Lipschitz with respect to \(W_1\) or some Lipschitz member has no Lipschitz extension with uniformly discrete spectral synthesis. Thus \xmark\ in the Lipschitz Extension column implies \xmark\ in the Spectral Synthesis column.
\end{enumerate}

For a real Banach-valued vectorization or extension, we use the complexification of its codomain when forming the varieties in Definition~\ref{def:vectorization-spectral-synthesis}. For the interval-based examples below, we take \(X=\{(p,q)\in\mathbb R^2:p\leq q\}\), \(A=\{(t,t):t\in\mathbb R\}\), and \(d([p,q],[p',q'])=|p-p'|+|q-q'|\).

\begin{lemma}\label{lem:no-lipschitz-extension-obstruction}
Let \(E\) be a Banach space, and let \(f\colon D(X,A)\to E\) be a persistence diagram vectorization. Suppose that there exist \(\alpha_n,\beta_n\in D(X,A)\) such that \(W_1(\alpha_n,\beta_n)>0\) and \(\|f(\alpha_n)-f(\beta_n)\|_E/W_1(\alpha_n,\beta_n)\to\infty\). Then \(f\) is not Lipschitz with respect to \(W_1\), and no Lipschitz extension of \(f\) to \(K(X,A)\) exists.
\end{lemma}

\begin{proof}
Suppose that \(\overline f\colon K(X,A)\to E\) is a Lipschitz extension of \(f\). By \cite[Corollary~7.10]{bubenik2022virtual}, the canonical embedding \(D(X,A)\hookrightarrow K(X,A)\) is isometric, so \(\|f(\alpha_n)-f(\beta_n)\|_E/W_1(\alpha_n,\beta_n)\leq\operatorname{Lip}(\overline f)\) for every \(n\), contradicting the hypothesis.
\end{proof}

\begin{lemma}\label{lem:additive-extension-synthesis}
Let \(E\) be a real or complex Banach space. Let \(f\colon D(X,A)\to E\) be additive, and suppose that there exists \(C\geq0\) such that \(\|f(\alpha)-f(\beta)\|_E\leq C W_1(\alpha,\beta)\) for all \(\alpha,\beta\in D(X,A)\). Then \(\overline f(\alpha-\beta)=f(\alpha)-f(\beta)\) defines the unique group homomorphism \(\overline f\colon K(X,A)\to E\) that extends \(f\). Moreover, \(\operatorname{Lip}(\overline f)\leq C\), and \(\overline f\) has uniformly discrete spectral synthesis.
\end{lemma}

\begin{proof}
Additivity gives \(f(0)=0\). Suppose that \(\alpha-\beta=\alpha'-\beta'\) in \(K(X,A)\). Then some \(\theta\in D(X,A)\) satisfies \(\alpha+\beta'+\theta=\alpha'+\beta+\theta\). As observed after Definition~\ref{def:persistence-diagram-monoid}, \(D(X,A)\cong D(X\setminus A)\). Since the free commutative monoid \(D(X\setminus A)\) is cancellative, \(\alpha+\beta'=\alpha'+\beta\). Additivity therefore gives \(f(\alpha)-f(\beta)=f(\alpha')-f(\beta')\), so \(\overline f\) is well-defined.

For \(\alpha,\beta,\gamma,\delta\in D(X,A)\),
\begin{align*}
\overline f\bigl((\alpha-\beta)+(\gamma-\delta)\bigr)
&=\overline f\bigl((\alpha+\gamma)-(\beta+\delta)\bigr)\\
&=f(\alpha+\gamma)-f(\beta+\delta)\\
&=\overline f(\alpha-\beta)+\overline f(\gamma-\delta),
\end{align*}
so \(\overline f\) is a group homomorphism. It extends \(f\) because \(\overline f(\alpha-0)=f(\alpha)\). Any group homomorphism \(T\colon K(X,A)\to E\) extending \(f\) satisfies \(T(\alpha-\beta)=f(\alpha)-f(\beta)=\overline f(\alpha-\beta)\), so uniqueness follows.

For \(\xi=\alpha-\beta\) and \(\eta=\gamma-\delta\),
\begin{align*}
\|\overline f(\xi)-\overline f(\eta)\|_E
&=\|f(\alpha+\delta)-f(\gamma+\beta)\|_E\\
&\leq C W_1(\alpha+\delta,\gamma+\beta)\\
&\overset{\eqref{eq:groth-metric}}{=}C\rho(\xi,\eta).
\end{align*}
Thus \(\operatorname{Lip}(\overline f)\leq C\).

Fix a uniformly discrete subpair \(A\subseteq Z\subseteq X\). By Theorem~\ref{thm:vpd-local-compactness}, \(K(Z,A)\) is discrete. If \(E\) is complex, every scalarization in \(E^*\) of \(\overline f\circ i_Z\) is additive. If \(E\) is real, we complexify \(E\) and \(\overline f\) and use bounded complex-linear scalarizations of the complexification. Lemma~\ref{lem:additive-fourier-stieltjes-synthesis}, with every representing measure equal to zero, shows that \(V_Z(\overline f)\) is synthesizable. Since \(Z\) was arbitrary, Definition~\ref{def:vectorization-spectral-synthesis} gives uniformly discrete spectral synthesis.
\end{proof}

\begin{lemma}\label{lem:collapsed-subset-pseudometric}
Let \((H,\rho)\) be a metric space, and let \(M\subseteq H\) be nonempty. Then \(\delta_M(x,y):=\min\{\rho(x,y),\operatorname{dist}(x,M)+\operatorname{dist}(y,M)\}\) defines a pseudometric on \(H\).
\end{lemma}

\begin{proof}
Symmetry and \(\delta_M(x,x)=0\) are immediate. Fix \(x,y,z\in H\), and set \(r_x:=\operatorname{dist}(x,M)\), \(r_y:=\operatorname{dist}(y,M)\), and \(r_z:=\operatorname{dist}(z,M)\). Since the distance-to-\(M\) map is \(1\)-Lipschitz, the four possible branch choices for \(\delta_M(x,y)\) and \(\delta_M(y,z)\) give
\begin{align*}
\delta_M(x,z)
&\leq\rho(x,z)\leq\rho(x,y)+\rho(y,z),\\
\delta_M(x,z)
&\leq r_x+r_z\leq\rho(x,y)+r_y+r_z,\\
\delta_M(x,z)
&\leq r_x+r_z\leq r_x+r_y+\rho(y,z),\\
\delta_M(x,z)
&\leq r_x+r_z\leq r_x+2r_y+r_z.
\end{align*}
Thus \(\delta_M\) satisfies the triangle inequality.
\end{proof}

\begin{lemma}\label{lem:separable-diagrams-and-completion}
If \(X/A\) is separable, then \(D(X,A)\) and \(K(X,A)\) are separable.
\end{lemma}

\begin{proof}
Choose a countable dense subset \(S_0\subseteq X/A\) containing \([A]\). Choose \(r\colon S_0\setminus\{[A]\}\to X\setminus A\) such that \(q(r(s))=s\), and set \(X_0:=r(S_0\setminus\{[A]\})\).

Let \(\alpha=\sum_{j=1}^N x_j\in D(X,A)\) and \(\varepsilon>0\). For each \(j\), choose \(s_j\in S_0\) such that \(\overline d_1(q(x_j),s_j)<\varepsilon/N\), and define
\[
\alpha':=\sum_{\substack{1\leq j\leq N\\s_j\neq[A]}}r(s_j).
\]
Match \(x_j\) with \(r(s_j)\) when \(s_j\neq[A]\). When \(s_j=[A]\), choose any \(a_j\in A\) and match \(x_j\) with \(a_j\). Then \(d_1(x_j,a_j)=d(x_j,A)=\overline d_1(q(x_j),[A])<\varepsilon/N\). Hence \(W_1(\alpha,\alpha')<\varepsilon\). Since finite multisets in the countable set \(X_0\) form a countable set, the persistence diagrams supported on \(X_0\) form a countable dense subset of \(D(X,A)\).

Let \(D_0\subseteq D(X,A)\) be countable and dense, and set \(K_0:=\{\alpha-\beta:\alpha,\beta\in D_0\}\). For \(\xi=\alpha-\beta\in K(X,A)\), choose \(\alpha_n,\beta_n\in D_0\) such that \(W_1(\alpha_n,\alpha)\to0\) and \(W_1(\beta_n,\beta)\to0\). Then
\begin{align*}
\rho(\alpha_n-\beta_n,\alpha-\beta)
&\overset{\eqref{eq:groth-metric}}{=}W_1(\alpha_n+\beta,\alpha+\beta_n)\\
&\leq W_1(\alpha_n+\beta,\alpha+\beta)+W_1(\alpha+\beta,\alpha+\beta_n)\\
&\overset{\eqref{eq:wasserstein-translation-invariant}}{=}W_1(\alpha_n,\alpha)+W_1(\beta,\beta_n).
\end{align*}
Thus \(K_0\) is dense in \(K(X,A)\).
\end{proof}

\begin{definition}\label{def:complex-lipschitz-free-space}
Let \((Q,d_Q,q_0)\) be a pointed metric space. Let \(\operatorname{Lip}_0(Q;\mathbb C)\) denote the complex Banach space of complex-valued Lipschitz functions on \(Q\) that vanish at \(q_0\), with norm \(\|\cdot\|_{\operatorname{Lip}}\). The complex Lipschitz-free space over \(Q\), denoted by \(\mathcal F_{\mathbb C}(Q)\), is the norm-closed complex linear span in \(\operatorname{Lip}_0(Q;\mathbb C)^*\) of the evaluation functionals \(\delta_Q(x)\), where \(\delta_Q(x)(u)=u(x)\).
\end{definition}

The evaluation map \(\delta_Q\colon Q\to\mathcal F_{\mathbb C}(Q)\) is isometric because
\[
\|\delta_Q(x)-\delta_Q(y)\|
=\sup_{\|u\|_{\operatorname{Lip}}\leq1}|u(x)-u(y)|
=d_Q(x,y).
\]
Every \(u\in\operatorname{Lip}_0(Q;\mathbb C)\) defines \(\widehat u\in\mathcal F_{\mathbb C}(Q)^*\) by \(\widehat u(v)=v(u)\), with \(\|\widehat u\|=\|u\|_{\operatorname{Lip}}\).

\begin{lemma}\label{lem:point-masses-give-synthesis}
Let \(G\) be a discrete abelian group, and let \(V\subseteq C(G)\) be a variety. If \(1_{\{x\}}\in V\) for every \(x\in G\), then \(V=C(G)\), and \(V\) is synthesizable.
\end{lemma}

\begin{proof}
The indicator functions of singletons span \(c_{00}(G)\), which is dense in \(C(G)\) for pointwise convergence, so pointwise closedness gives \(V=C(G)\).

For distinct \(x,y\in G\), the Pontryagin duality theorem provides \(\chi\in\widehat G\) such that \(\chi(x)\neq\chi(y)\) \cite[Theorem~4.32]{folland2016course}. For every finite \(F\subseteq G\), the complex span of the restrictions of the characters of \(G\) to \(F\) is a unital self-adjoint subalgebra of \(\mathbb C^F\) that separates points. The Stone--Weierstrass theorem therefore gives the whole space \(\mathbb C^F\). Hence finite linear combinations of characters are dense in \(C(G)\) for pointwise convergence. Since characters are exponential monomials, Theorem~\ref{thm:synthesis-equivalent-characterizations} gives synthesis.
\end{proof}

\begin{lemma}\label{lem:subpair-grothendieck-isometry}
For every subpair \(A\subseteq Z\subseteq X\), the group homomorphism \(i_Z\colon K(Z,A)\to K(X,A)\) induced by the inclusion \(Z\hookrightarrow X\) is isometric.
\end{lemma}

\begin{proof}
By Equation~\eqref{eq:one-strengthening}, the strengthened distances computed in \(Z\) and \(X\) agree on \(Z\), so the quotient distances agree on \(Z/A\). Every matching between \(\alpha,\beta\in D(Z,A)\) remains a matching between \(i_Z\alpha\) and \(i_Z\beta\) in \(D(X,A)\) with the same cost.

Conversely, let \(\sigma\in D(X\times X)\) match \(i_Z\alpha\) and \(i_Z\beta\). A summand with first coordinate in \(X\setminus Z\) would contribute that point to \((\pi_1)_*\sigma\) outside \(A\), contradicting \((\pi_1)_*\sigma=i_Z\alpha\pmod{D(A)}\). The same argument applies to the second coordinate. Hence \(\sigma\in D(Z\times Z)\), with unchanged cost. Thus the two \(1\)-Wasserstein distances agree under \(i_Z\).

For \(\alpha,\beta,\gamma,\delta\in D(Z,A)\),
\begin{align*}
\rho\bigl(i_Z(\alpha-\beta),i_Z(\gamma-\delta)\bigr)
&\overset{\eqref{eq:groth-metric}}{=}W_1(i_Z\alpha+i_Z\delta,i_Z\gamma+i_Z\beta)\\
&=W_1(\alpha+\delta,\gamma+\beta)\\
&\overset{\eqref{eq:groth-metric}}{=}\rho(\alpha-\beta,\gamma-\delta).
\end{align*}
Thus \(i_Z\) is isometric.
\end{proof}

\begin{lemma}\label{lem:linfty-real-line-extension-synthesis}
Suppose that \(X/A\) is separable, and let \(f\colon D(X,A)\to L^\infty(\mathbb R;\mathbb C)\) be Lipschitz with \(f(0)=0\). Then there exists a Lipschitz extension \(\overline f\colon K(X,A)\to L^\infty(\mathbb R;\mathbb C)\) with uniformly discrete spectral synthesis.
\end{lemma}

\begin{proof}
By Lemma~\ref{lem:separable-diagrams-and-completion}, \(D(X,A)\) and \(K(X,A)\) are separable. Choose a countable dense subset \(D_0\subseteq D(X,A)\) with \(0\in D_0\). For each \(\alpha\in D_0\), choose a measurable representative \(F_\alpha\) of \(f(\alpha)\), with \(F_0=0\), and define
\[
N:=\bigcup_{\alpha,\beta\in D_0}
\left\{
t\in\mathbb R:
|F_\alpha(t)-F_\beta(t)|
>
\operatorname{Lip}(f)W_1(\alpha,\beta)
\right\}.
\]
Each set in the union is null, so \(N\) is null. Redefine every \(F_\alpha\) to equal zero on \(N\). Then \(|F_\alpha(t)-F_\beta(t)|\leq\operatorname{Lip}(f)W_1(\alpha,\beta)\) for every \(\alpha,\beta\in D_0\) and \(t\in\mathbb R\).

For \(x\in K(X,A)\) and \(t\in\mathbb R\), define
\begin{align*}
g_{\operatorname{Re}}(x)(t)
&:=\inf_{\alpha\in D_0}
\left\{\operatorname{Re}F_\alpha(t)+\operatorname{Lip}(f)\rho(x,\alpha)\right\},\\
g_{\operatorname{Im}}(x)(t)
&:=\inf_{\alpha\in D_0}
\left\{\operatorname{Im}F_\alpha(t)+\operatorname{Lip}(f)\rho(x,\alpha)\right\},
\end{align*}
and set \(g(x):=g_{\operatorname{Re}}(x)+i g_{\operatorname{Im}}(x)\). Each coordinate is measurable because it is the infimum of a countable family of measurable functions.

For \(\alpha\in D_0\), the pointwise estimate with \(\beta=0\) gives \(|F_\alpha(t)|\leq\operatorname{Lip}(f)\rho(\alpha,0)\). Since \(\rho(x,\alpha)\geq\rho(\alpha,0)-\rho(x,0)\), we obtain \(-\operatorname{Lip}(f)\rho(x,0)\leq g_{\operatorname{Re}}(x)(t)\leq\operatorname{Lip}(f)\rho(x,0)\), where the upper bound follows by taking \(\alpha=0\). The same argument applies to the imaginary part, so \(|g(x)(t)|\leq2\operatorname{Lip}(f)\rho(x,0)\). Hence \(g(x)\in L^\infty(\mathbb R;\mathbb C)\).

For \(x,y\in K(X,A)\), the triangle inequality gives \(|g_{\operatorname{Re}}(x)(t)-g_{\operatorname{Re}}(y)(t)|\leq\operatorname{Lip}(f)\rho(x,y)\), and the same estimate holds for the imaginary part. Hence \(\|g(x)-g(y)\|_{L^\infty}\leq2\operatorname{Lip}(f)\rho(x,y)\).

Fix \(\beta\in D_0\). The pointwise Lipschitz estimate gives \(\operatorname{Re}F_\alpha(t)+\operatorname{Lip}(f)\rho(\beta,\alpha)\geq\operatorname{Re}F_\beta(t)\) for every \(\alpha\in D_0\). Taking the infimum and then taking \(\alpha=\beta\) gives \(g_{\operatorname{Re}}(\beta)=\operatorname{Re}F_\beta\), and the same argument gives \(g_{\operatorname{Im}}(\beta)=\operatorname{Im}F_\beta\). Thus \(g(\beta)=F_\beta\).

For \(\beta\in D(X,A)\), choose \(\beta_n\in D_0\) with \(W_1(\beta_n,\beta)\to0\). Since \(g(\beta_n)=f(\beta_n)\),
\[
\|g(\beta)-f(\beta)\|_{L^\infty}
\leq
\|g(\beta)-g(\beta_n)\|_{L^\infty}
+
\|f(\beta_n)-f(\beta)\|_{L^\infty}
\longrightarrow0.
\]
Hence \(g|_{D(X,A)}=f\).

We now add a Lipschitz map that vanishes on \(D(X,A)\) and whose scalarizations recover Lipschitz functions on the metric quotient obtained by collapsing \(D(X,A)\).

Let \(E_0:=\overline{\operatorname{span}_{\mathbb C}g(K(X,A))}\). Since \(g\) is Lipschitz and \(K(X,A)\) is separable, \(E_0\) is separable. Choose a dense sequence \((y_j)_{j\geq1}\subseteq E_0\) and measurable representatives. We write \(|E|\) for the Lebesgue measure of a measurable set \(E\subseteq\mathbb R\).

For each \(n\geq1\), define \(F_n(t):=(y_1(t),\ldots,y_n(t))\) on \([n,n+1)\). Its essential range is bounded. Choose a bounded Borel set \(K_n\subseteq\mathbb C^n\) containing this essential range and partition \(K_n\) into finitely many Borel sets of diameter at most \(1/n\). One inverse image \(E_n\subseteq[n,n+1)\) has positive measure. Choose disjoint measurable subsets \(E_n^+,E_n^-\subseteq E_n\) of positive measure. For \(1\leq j\leq n\), the averages of \(y_j\) over \(E_n^+\) and \(E_n^-\) differ by at most \(1/n\).

Define \(A_0\colon L^\infty(\mathbb R;\mathbb C)\to\ell^\infty\) by
\[
A_0(F)_n
:=
\frac12
\left(
\frac{1}{|E_n^+|}\int_{E_n^+}F(t)\,dt
-
\frac{1}{|E_n^-|}\int_{E_n^-}F(t)\,dt
\right).
\]
Then \(\|A_0\|\leq1\). For \(n\geq j\), \(|A_0(y_j)_n|\leq1/(2n)\), so \(A_0(y_j)\in c_0\). Density of \((y_j)\) and closedness of \(c_0\) give \(A_0(E_0)\subseteq c_0\).

Define \(J\colon\ell^\infty\to L^\infty(\mathbb R;\mathbb C)\) by
\[
J(b)(t)
:=
\begin{cases}
b_n,&t\in E_n^+,\\
-b_n,&t\in E_n^-,\\
0,&t\notin\displaystyle\bigcup_{n\geq1}(E_n^+\cup E_n^-).
\end{cases}
\]
The sets \(E_n^\pm\) are pairwise disjoint, so \(\|J(b)\|_{L^\infty}=\|b\|_{\ell^\infty}\) and \(A_0J=\operatorname{id}_{\ell^\infty}\).

By Lemma~\ref{lem:collapsed-subset-pseudometric}, the formula \(\delta(x,y):=\min\{\rho(x,y),\operatorname{dist}(x,D(X,A))+\operatorname{dist}(y,D(X,A))\}\) defines a pseudometric on \(K(X,A)\). Let \(Q\) be its metric quotient, let \(\kappa\colon K(X,A)\to Q\) be the quotient map, and take \(\kappa(D(X,A))\) as the base point. Then
\begin{equation}
\label{eq:linfty-quotient-metric}
d_Q(\kappa(x),\kappa(y))
=
\min\left\{
\rho(x,y),
\operatorname{dist}(x,D(X,A))
+
\operatorname{dist}(y,D(X,A))
\right\}.
\end{equation}
Thus \(\kappa\) is \(1\)-Lipschitz, and separability of \(K(X,A)\) gives separability of \(Q\).

The space \(\mathcal F_{\mathbb C}(Q)\) is separable, so the closed unit ball of \(\mathcal F_{\mathbb C}(Q)^*\) is weak-* compact and metrizable. Choose a weak-* dense sequence \((\phi_j)_{j\geq1}\) in this ball and define \(J_0(v):=(\phi_j(v))_{j\geq1}\). Then
\begin{align*}
\|J_0(v)\|_{\ell^\infty}
&=\sup_j|\phi_j(v)|\\
&=\sup_{\phi\in B_{\mathcal F_{\mathbb C}(Q)^*}}|\phi(v)|\\
&=\|v\|,
\end{align*}
so \(J_0\) is a complex-linear isometry.

Choose \(r\colon\mathbb N\to\mathbb N\) with infinite fibers and define \(\mathcal R(b)_n:=b_{r(n)}\). Then \(\|\mathcal R(b)\|_{\ell^\infty}=\|b\|_{\ell^\infty}\). If \(\pi\colon\ell^\infty\to\ell^\infty/c_0\) denotes the quotient map, then \(\|\pi\mathcal R(b)\|=\limsup_n|\mathcal R(b)_n|=\|b\|_{\ell^\infty}\). Hence \(T:=\pi\mathcal R J_0\colon\mathcal F_{\mathbb C}(Q)\to\ell^\infty/c_0\) is a complex-linear isometry.

Define \(h(x):=J(\mathcal R J_0\delta_Q(\kappa(x)))\). Since \(\kappa(D(X,A))\) is the base point, \(h|_{D(X,A)}=0\), and
\begin{align*}
\|h(x)-h(y)\|_{L^\infty}
&=
\left\|
\mathcal R J_0
\bigl(
\delta_Q(\kappa(x))-\delta_Q(\kappa(y))
\bigr)
\right\|_{\ell^\infty}\\
&=
\left\|
\delta_Q(\kappa(x))-\delta_Q(\kappa(y))
\right\|_{\mathcal F_{\mathbb C}(Q)}\\
&=
d_Q(\kappa(x),\kappa(y))
\leq\rho(x,y).
\end{align*}
Thus \(h\) is \(1\)-Lipschitz.

Set \(\overline f:=g+h\). Then \(\overline f|_{D(X,A)}=f\) and \(\operatorname{Lip}(\overline f)\leq2\operatorname{Lip}(f)+1\). Since \(A_0(g(x))\in c_0\) and \(A_0J=\operatorname{id}_{\ell^\infty}\), we have \(\pi A_0(g(x))=0\) and \(\pi A_0(h(x))=T\delta_Q(\kappa(x))\).

Fix \(u\in\operatorname{Lip}_0(Q;\mathbb C)\), and let \(\widehat u\in\mathcal F_{\mathbb C}(Q)^*\) be its linearization. Define \(\theta_u(Tv):=\widehat u(v)\). Since \(T\) is an isometry, \(\theta_u\) is bounded on \(T(\mathcal F_{\mathbb C}(Q))\). By the complex Hahn--Banach theorem, extend it to \(\widetilde\theta_u\in(\ell^\infty/c_0)^*\), and set \(\lambda_u:=\widetilde\theta_u\circ\pi\circ A_0\). Then
\begin{equation}
\label{eq:linfty-scalar-recovery}
\begin{aligned}
\lambda_u(\overline f(x))
&=\widetilde\theta_u(\pi A_0(\overline f(x)))\\
&=\theta_u(T\delta_Q(\kappa(x)))\\
&=\widehat u(\delta_Q(\kappa(x)))\\
&=u(\kappa(x)).
\end{aligned}
\end{equation}

Fix a uniformly discrete subpair \(A\subseteq Z\subseteq X\). If \(K(Z,A)=\{0\}\), then \(V_Z(\overline f)\) is synthesizable. Assume \(K(Z,A)\neq\{0\}\). Choose \(\eta>0\) such that \(\overline d_1([v],[w])\geq\eta\) for distinct classes in \(Z/A\), so \(d(v,A)\geq\eta\) for \(v\in Z\setminus A\).

Choose \(z\in Z\setminus A\), identify it with its singleton persistence diagram, and set \(a:=0-z\in K(Z,A)\). By Lemma~\ref{lem:subpair-grothendieck-isometry}, \(i_Z\) is isometric. For every \(\gamma\in D(X,A)\), Equation~\eqref{eq:groth-metric} gives \(\rho(i_Za,\gamma)=W_1(0,z+\gamma)\geq d(z,A)\geq\eta\). Hence \(\operatorname{dist}(i_Za,D(X,A))\geq\eta\).

By Theorem~\ref{thm:vpd-local-compactness}, \(K(Z,A)\) is discrete, so some \(\delta>0\) satisfies \(\rho(i_Za,i_Zx)\geq\delta\) for \(x\neq a\). Equation~\eqref{eq:linfty-quotient-metric} then gives
\[
d_Q(\kappa(i_Za),\kappa(i_Zx))
\geq
\min\{\delta,\eta\}
\qquad
(x\neq a).
\]
Thus \(\kappa(i_Za)\) is isolated in \((\kappa\circ i_Z)(K(Z,A))\).

Let \(u_a\) be the indicator of \(\{\kappa(i_Za)\}\) on \((\kappa\circ i_Z)(K(Z,A))\). The isolation estimate makes \(u_a\) Lipschitz. Since \(\kappa(i_Za)\) differs from the base point, \(u_a\) vanishes there. The McShane extension theorem gives a real-valued Lipschitz extension to \(Q\), which we regard as an element of \(\operatorname{Lip}_0(Q;\mathbb C)\). Equation~\eqref{eq:linfty-scalar-recovery} gives \(\lambda_{u_a}\circ\overline f\circ i_Z=1_{\{a\}}\). Translation invariance then gives the indicator function of every singleton in \(V_Z(\overline f)\), so Lemma~\ref{lem:point-masses-give-synthesis} gives synthesis. Since \(Z\) was arbitrary, \(\overline f\) has uniformly discrete spectral synthesis.
\end{proof}

\subsection{Lipschitz-free embedding}

In Table~\ref{tab:persistence-vectorizations}, \(\mathcal F(K(X,A))\) denotes the Lipschitz-free space \(\widehat V(K(X,A),0)\) used below. The listed embedding is the composition of \(D(X,A)\hookrightarrow K(X,A)\) with \(\delta\colon K(X,A)\to\widehat V(K(X,A),0)\).

\begin{proposition}\label{prop:lipschitz-free-synthesis}
The Lipschitz-free embedding \(\delta\colon K(X,A)\to\widehat V(K(X,A),0)\) has uniformly discrete spectral synthesis.
\end{proposition}

\begin{proof}
By \cite[Theorem~7.9]{bubenik2022virtual}, \(\delta\) is an isometric embedding of \((K(X,A),\rho,0)\). We form its spectral varieties after complexifying \(\widehat V(K(X,A),0)\) and extend bounded real-linear functionals complex-linearly.

Fix a uniformly discrete subpair \(A\subseteq Z\subseteq X\). By Lemma~\ref{lem:subpair-grothendieck-isometry}, \(i_Z\colon K(Z,A)\to K(X,A)\) is isometric. Theorem~\ref{thm:vpd-local-compactness} implies that \(K(Z,A)\) is discrete, so some \(\eta>0\) satisfies \(\rho(x,y)\geq\eta\) for distinct \(x,y\in K(Z,A)\).

Fix \(a\in K(Z,A)\setminus\{0\}\). Then \(1_{\{a\}}\) is \(\eta^{-1}\)-Lipschitz and vanishes at \(0\). Since \(i_Z\) is isometric, the function \(i_Zx\mapsto1_{\{a\}}(x)\) is \(\eta^{-1}\)-Lipschitz on \(i_Z(K(Z,A))\). The McShane extension theorem gives \(u_a\colon K(X,A)\to\mathbb R\) with \(u_a\circ i_Z=1_{\{a\}}\) and \(u_a(0)=0\).

By \cite[Theorem~7.8]{bubenik2022virtual}, there exists a unique bounded real-linear functional \(\widetilde u_a\colon\widehat V(K(X,A),0)\to\mathbb R\) such that \(\widetilde u_a\circ\delta=u_a\). Its complex-linear extension satisfies \(\widetilde u_a^{\mathbb C}\circ\delta\circ i_Z=1_{\{a\}}\), so \(1_{\{a\}}\in V_Z(\delta)\). Translation invariance gives the indicator function of every singleton, and Lemma~\ref{lem:point-masses-give-synthesis} gives synthesis. Since \(Z\) was arbitrary, \(\delta\) has uniformly discrete spectral synthesis.
\end{proof}

\subsection{Additive Functions}

\begin{proposition}\label{prop:additive-functions-synthesis}
Every Lipschitz additive function \(a\colon K(X,A)\to\mathbb C\) has uniformly discrete spectral synthesis.
\end{proposition}

\begin{proof}
Fix a uniformly discrete subpair \(A\subseteq Z\subseteq X\). By Theorem~\ref{thm:vpd-local-compactness}, \(K(Z,A)\) is discrete. Since \(a\circ i_Z\in\operatorname{Hom}(K(Z,A),\mathbb C)\), Lemma~\ref{lem:additive-fourier-stieltjes-synthesis}, with zero representing measure, gives synthesis of \(V_Z(a)\). Since \(Z\) was arbitrary, \(a\) has uniformly discrete spectral synthesis.
\end{proof}

\subsection{Characters}

\begin{proposition}\label{prop:multiplicative-characters-synthesis}
Every Lipschitz character \(\chi\colon K(X,A)\to\mathbb C^\times\) has uniformly discrete spectral synthesis.
\end{proposition}

\begin{proof}
Fix a uniformly discrete subpair \(A\subseteq Z\subseteq X\). The function \(\chi\circ i_Z\colon K(Z,A)\to\mathbb C^\times\) is an exponential, which we regard as complex-valued when forming \(V_Z(\chi)\). Multiplicativity gives \(\tau_y(\chi\circ i_Z)=(\chi\circ i_Z)(y)(\chi\circ i_Z)\) for every \(y\in K(Z,A)\). Hence \(V_Z(\chi)=\operatorname{span}_{\mathbb C}\{\chi\circ i_Z\}\). Since \(\chi\circ i_Z\) is an exponential monomial, Theorem~\ref{thm:synthesis-equivalent-characterizations} gives synthesis. Since \(Z\) was arbitrary, \(\chi\) has uniformly discrete spectral synthesis.
\end{proof}

\subsection{Persistence Statistics}

We use the \(\mathbb R^{38}\)-valued persistence statistics vector of \cite[Definition~3.1]{ali2023survey}. One coordinate counts the total number of bars with multiplicity.

\begin{proposition}\label{prop:statistics-no-extension}
The persistence statistics vector is not Lipschitz with respect to \(W_1\) on the nonempty persistence diagrams. No extension of this vector to \(D(X,A)\) is Lipschitz with respect to \(W_1\) or admits a Lipschitz extension to \(K(X,A)\).
\end{proposition}

\begin{proof}
Let \(f\) denote the persistence statistics vector on the nonempty persistence diagrams. Choose intervals \(I_n=[p_n,q_n]\) with lifespans \(\ell_n=q_n-p_n>0\) such that \(\ell_n\to0\), and set \(D_n=\{I_n\}\) and \(E_n=\{I_n,I_n\}\). Every matching between \(D_n\) and \(E_n\) must match one copy of \(I_n\) in \(E_n\) to \(A\), which has cost \(d(I_n,A)=\ell_n\). We may match the other copy of \(I_n\) in \(E_n\) to the copy in \(D_n\) at zero cost, so \(W_1(D_n,E_n)=\ell_n\).

The bar-count coordinate equals \(1\) on \(D_n\) and \(2\) on \(E_n\), so \(\|f(D_n)-f(E_n)\|_2\geq1\). Hence \(\|f(D_n)-f(E_n)\|_2/W_1(D_n,E_n)\geq1/\ell_n\to\infty\), so \(f\) is not Lipschitz with respect to \(W_1\) on the nonempty persistence diagrams.

Let \(\widetilde f\colon D(X,A)\to\mathbb R^{38}\) be any extension of \(f\). Since \(D_n\) and \(E_n\) are nonempty, we have \(\widetilde f(D_n)=f(D_n)\) and \(\widetilde f(E_n)=f(E_n)\), so the same unbounded Lipschitz quotient holds for \(\widetilde f\). By Lemma~\ref{lem:no-lipschitz-extension-obstruction}, \(\widetilde f\) is not Lipschitz with respect to \(W_1\) and has no Lipschitz extension to \(K(X,A)\). Thus no extension of the persistence statistics vector to \(D(X,A)\) has either property.
\end{proof}

\subsection{Persistent Entropy}

We use the entropy of \cite[Definition~3.1(3)]{ali2023survey}. For a nonempty barcode with multiplicity function \(\mu\colon B\to\mathbb Z_{>0}\), set \(L_\mu=\sum_{[p,q]\in B}\mu_{p,q}(q-p)\). Its persistent entropy is
\[
E_\mu
:=
-
\sum_{[p,q]\in B}
\mu_{p,q}
\left(
\frac{q-p}{L_\mu}
\right)
\log\left(
\frac{q-p}{L_\mu}
\right),
\]
where \(\mu_{p,q}\) denotes the multiplicity of \([p,q]\).

\begin{proposition}\label{prop:entropy-no-extension}
Persistent entropy is not Lipschitz with respect to \(W_1\) on the nonempty persistence diagrams. Moreover, no extension of persistent entropy to \(D(X,A)\) is Lipschitz with respect to \(W_1\), and no such extension has a Lipschitz extension to \(K(X,A)\).
\end{proposition}

\begin{proof}
Let \(E\) denote persistent entropy on the nonempty persistence diagrams. Choose an interval \(I=[p,q]\) with lifespan \(L=q-p>0\), and choose intervals \(J_n\) with lifespans \(\varepsilon_n>0\) such that \(\varepsilon_n\downarrow0\). Set \(D=\{I\}\) and \(D_n=\{I,J_n\}\). Matching the copy of \(I\) in \(D_n\) to the copy of \(I\) in \(D\) has zero cost, while matching \(J_n\) to the diagonal has cost \(\varepsilon_n\). Hence \(W_1(D,D_n)=\varepsilon_n\).

Set \(r_n=\varepsilon_n/(L+\varepsilon_n)\). The term in \(E(D_n)\) corresponding to \(J_n\) is \(-r_n\log r_n\), and every term in the entropy sum is nonnegative. Hence \(E(D_n)\geq-r_n\log r_n\). Since the unique normalized lifespan in \(D\) equals \(1\), we have \(E(D)=0\).

For all sufficiently large \(n\), we have \(\varepsilon_n\leq L\), so \(r_n\geq\varepsilon_n/(2L)\). We also have \(-\log r_n=-\log\varepsilon_n+\log(L+\varepsilon_n)\), while \(|\log\varepsilon_n|\to\infty\). Hence, for all sufficiently large \(n\), \(-\log r_n\geq\frac12|\log\varepsilon_n|\). Therefore
\[
\frac{E(D_n)-E(D)}{W_1(D,D_n)}
\geq
\frac{1}{4L}|\log\varepsilon_n|
\longrightarrow
\infty.
\]
Persistent entropy is therefore not Lipschitz with respect to \(W_1\) on the nonempty persistence diagrams.

Let \(\widetilde E\colon D(X,A)\to\mathbb R\) be any extension of persistent entropy. The same persistence diagrams \(D,D_n\) show that \(\widetilde E\) is not Lipschitz with respect to \(W_1\). By Lemma~\ref{lem:no-lipschitz-extension-obstruction}, \(\widetilde E\) has no Lipschitz extension to \(K(X,A)\). Since \(\widetilde E\) was arbitrary, no extension of persistent entropy to \(D(X,A)\) has a Lipschitz extension to \(K(X,A)\).
\end{proof}

\subsection{Entropy Summary}

Atienza et al.~\cite[Definition~4.1]{atienza2020stability} introduced the entropy summary function. We use the natural-logarithm normalization of the formula in \cite[Definition~3.2]{ali2023survey}. For a nonempty persistence diagram \(\mu\colon B\to\mathbb Z_{>0}\), with \(L_\mu\) as above, its entropy summary function \(S_\mu\colon\mathbb R\to\mathbb R\) is
\[
S_\mu(t)
=
-
\sum_{[p,q]\in B}
\mathbf1_{[p,q)}(t)\,
\mu_{p,q}
\left(
\frac{q-p}{L_\mu}
\right)
\log\left(
\frac{q-p}{L_\mu}
\right).
\]
We extend the entropy summary function to the empty persistence diagram by \(S_0=0\).

\begin{figure}[t]
    \centering
    \includegraphics[width=0.72\textwidth]{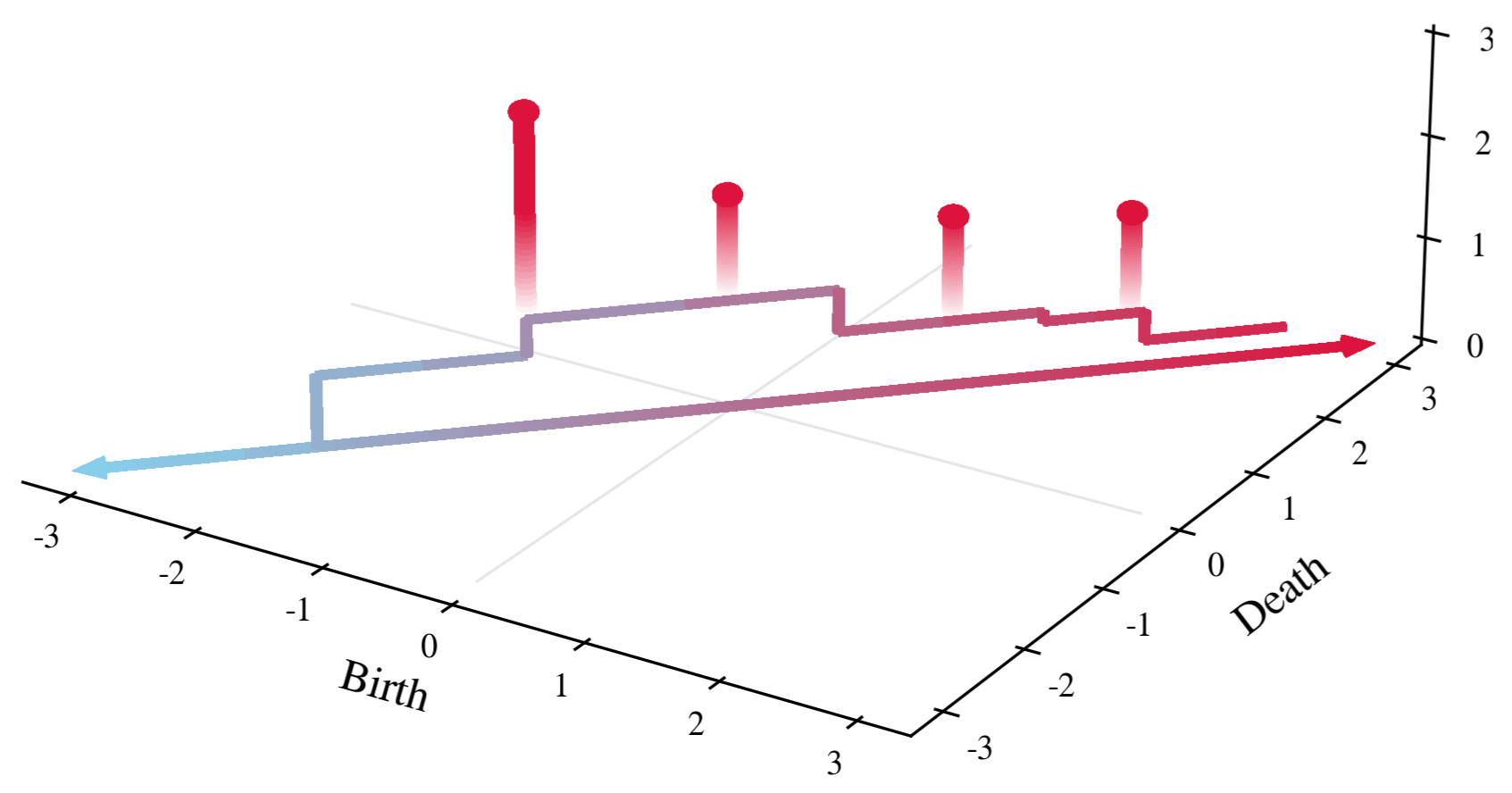}
    \caption{Entropy summary function of a persistence diagram.}
    \label{fig:entropy-summary}
\end{figure}

\begin{proposition}\label{prop:entropy-summary-classification}
For every \(1<p\leq\infty\), the entropy summary function \(S\colon D(X,A)\to L^p(\mathbb R)\) is not Lipschitz with respect to \(W_1\) and has no Lipschitz extension to \(K(X,A)\). For \(p=1\),
\begin{equation}\label{eq:entropy-summary-l1-lipschitz}
\|S_\alpha-S_\beta\|_{L^1}
\leq
\left(2+\frac1e\right)W_1(\alpha,\beta)
\end{equation}
for all \(\alpha,\beta\in D(X,A)\), and \(S\) has a Lipschitz extension to \(K(X,A)\). If \(X/A\) is separable, then \(S\) has a Lipschitz extension to \(K(X,A)\) with uniformly discrete spectral synthesis.
\end{proposition}

\begin{proof}
For \(R>0\), let \(D_R\) send each interval \([b,d]\) to \([Rb,Rd]\) and preserve multiplicities. The normalized lifespan weights remain unchanged, so \(S_{D_R\alpha}(t)=S_\alpha(t/R)\). Since \(d_1(D_Ru,D_Rv)=R\,d_1(u,v)\), applying \(D_R\) and \(D_{1/R}\) to matchings shows that \(W_1(D_R\alpha,D_R\beta)=R\,W_1(\alpha,\beta)\). For \(1\leq p<\infty\), substitution gives \(\|S_{D_R\alpha}-S_{D_R\beta}\|_{L^p}=R^{1/p}\|S_\alpha-S_\beta\|_{L^p}\), while \(\|S_{D_R\alpha}-S_{D_R\beta}\|_{L^\infty}=\|S_\alpha-S_\beta\|_{L^\infty}\).

Let \(\alpha\) consist of two equal intervals of positive lifespan and let \(\beta=0\). Then \(S_\alpha\neq0=S_\beta\) and \(W_1(\alpha,\beta)>0\). For \(1<p<\infty\), the Lipschitz quotient for \(D_R\alpha,D_R\beta\) equals \(R^{1/p-1}\|S_\alpha\|_{L^p}/W_1(\alpha,0)\), which tends to infinity as \(R\downarrow0\). For \(p=\infty\), the quotient equals \(R^{-1}\|S_\alpha\|_{L^\infty}/W_1(\alpha,0)\) and has the same conclusion. Hence \(S\) is not Lipschitz with respect to \(W_1\) for \(1<p\leq\infty\). By Lemma~\ref{lem:no-lipschitz-extension-obstruction}, \(S\) has no Lipschitz extension to \(K(X,A)\) for these values of \(p\).

We next prove Equation~\eqref{eq:entropy-summary-l1-lipschitz}. Suppose first that \(\alpha\) and \(\beta\) are nonempty. Fix a matching with matched intervals \(I_i,J_i\), where diagonal intervals have zero length. Let \(\ell_i,m_i\) denote their lengths, set \(L=\sum_i\ell_i\), \(M=\sum_i m_i\), \(p_i=\ell_i/L\), and \(q_i=m_i/M\), and interchange \(\alpha,\beta\) if necessary so that \(M\leq L\). Define \(h(s)=-s\log s\) for \(s>0\) and \(h(0)=0\). Then \(S_\alpha=\sum_i h(p_i)\mathbf1_{I_i}\) and \(S_\beta=\sum_i h(q_i)\mathbf1_{J_i}\).

For \(p,q\in[0,1]\),
\begin{equation}\label{eq:entropy-summary-coefficient-bound}
q|h(p)-h(q)|
\leq
|p-q|.
\end{equation}
If \(p\leq q\), write \(p=rq\) with \(r\in[0,1]\). Then \(h(rq)-h(q)=-(1-r)h(q)+qr(-\log r)\). Since \(0\leq h(q)\leq1/e\leq1\) and \(0\leq r(-\log r)\leq1-r\), both \(h(q)\) and \(qr(-\log r)/(1-r)\), when \(r<1\), lie in \([0,1]\). Hence \(q|h(p)-h(q)|\leq q(1-r)=|p-q|\). If \(q\leq p\), write \(q=rp\). The identity \(h(p)-h(rp)=(1-r)h(p)-pr(-\log r)\) gives \(q|h(p)-h(q)|\leq rp(1-r)\leq p(1-r)=|p-q|\). This proves Equation~\eqref{eq:entropy-summary-coefficient-bound}.

For intervals \(I=[a,b]\) and \(J=[c,d]\), we have \(|I\triangle J|\leq|a-c|+|b-d|=d_1(I,J)\) and \(||I|-|J||\leq d_1(I,J)\). Using
\(h(p_i)\mathbf1_{I_i}-h(q_i)\mathbf1_{J_i}
=h(p_i)(\mathbf1_{I_i}-\mathbf1_{J_i})+(h(p_i)-h(q_i))\mathbf1_{J_i}\),
Equation~\eqref{eq:entropy-summary-coefficient-bound}, and \(h(p_i)\leq1/e\), we obtain
\[
\|S_\alpha-S_\beta\|_{L^1}
\leq
\frac1e\sum_i d_1(I_i,J_i)
+
M\sum_i|p_i-q_i|.
\]
Moreover, \(M\sum_i|p_i-q_i|=\sum_i|(M/L)\ell_i-m_i|\leq(M/L)\sum_i|\ell_i-m_i|+(M/L)(L-M)\leq2\sum_i|\ell_i-m_i|\leq2\sum_i d_1(I_i,J_i)\). Hence \(\|S_\alpha-S_\beta\|_{L^1}\leq(2+1/e)\sum_i d_1(I_i,J_i)\). Taking the infimum over matchings proves Equation~\eqref{eq:entropy-summary-l1-lipschitz} when both diagrams are nonempty. If \(\beta=0\), then \(\|S_\alpha\|_{L^1}=\sum_i h(p_i)\ell_i\leq e^{-1}\sum_i\ell_i=e^{-1}W_1(\alpha,0)\), and symmetry gives the case \(\alpha=0\).

We first construct a Lipschitz extension of \(S\) to \(K(X,A)\). If \(X/A\) is separable, we then add a Lipschitz map that vanishes on \(D(X,A)\) and provides scalarizations that agree with indicator functions of singletons on finite subsets of \(K(Z,A)\) for every uniformly discrete subpair \(A\subseteq Z\subseteq X\).

Extend \(h\) from \([0,1]\) to \(H\colon[0,\infty)\to\mathbb R\) by
\[
H(s)
:=
\begin{cases}
-s\log s,&0<s\leq1,\\
-\left(u+\dfrac32u^2\right)e^{-u},&s\geq1,\quad u=s-1,\\
0,&s=0.
\end{cases}
\]
The two formulas and their first two derivatives agree at \(s=1\), so \(H\in C([0,\infty))\cap C^2((0,\infty))\). Set \(A_H:=\sup_{s\geq0}|H(s)|\), \(B_H:=\sup_{s>0}s|H'(s)|\), and \(C_H:=\sup_{s>0}s|(sH'(s))'|\). These quantities are finite. Define \(K_H:=8\max\{1,A_H,B_H,C_H\}\).

For \(r>0\) and an interval \(I=[p,q]\) of length \(\ell=q-p\), set \(\psi_r(I):=H(\ell/r)\mathbf1_{[p,q)}\), and set \(\psi_r(I)=0\) for diagonal intervals. We claim that
\begin{equation}\label{eq:entropy-fixed-scale-generator}
\|\psi_r(I)-\psi_r(J)\|_{L^1}
\leq
K_Hd_1(I,J).
\end{equation}
For \(p,q\geq0\), we first show \(q|H(p)-H(q)|\leq(K_H/2)|p-q|\). The claim is immediate for \(q=0\). If \(q/2\leq p\leq2q\), the mean-value theorem applied to \(u\mapsto H(e^u)\) gives \(|H(p)-H(q)|\leq B_H|\log(p/q)|\leq2B_H|p-q|/q\). If \(p<q/2\), then \(q|H(p)-H(q)|\leq2A_Hq\leq4A_H|p-q|\). If \(p>2q\), then \(q|H(p)-H(q)|\leq2A_Hq\leq2A_H|p-q|\). The definition of \(K_H\) proves the claimed coefficient bound.

If \(I,J\) have lengths \(\ell,m\), respectively, then
\(\|\psi_r(I)-\psi_r(J)\|_{L^1}\leq A_H|I\triangle J|+(K_H/2)|\ell-m|\leq K_Hd_1(I,J)\),
which proves Equation~\eqref{eq:entropy-fixed-scale-generator}.

Define \(T_r(\alpha):=\sum_{I\in\alpha}H(|I|/r)\mathbf1_I\), counting multiplicities. This map is additive. Summing Equation~\eqref{eq:entropy-fixed-scale-generator} over any matching and taking the infimum gives \(\|T_r(\alpha)-T_r(\beta)\|_{L^1}\leq K_HW_1(\alpha,\beta)\). By Lemma~\ref{lem:additive-extension-synthesis}, \(T_r\) has an additive extension, also denoted by \(T_r\colon K(X,A)\to L^1(\mathbb R)\), with \(\operatorname{Lip}(T_r)\leq K_H\).

We next compare two scales. For \(1\leq a\leq2\), set \(G_a(s):=H(s/a)-H(s)\) and \(\widetilde H(u):=H(e^u)\). We have \(\widetilde H'(u)=e^uH'(e^u)\) and \(\widetilde H''(u)=e^u(sH'(s))'|_{s=e^u}\). Hence \(\sup_{s>0}|G_a(s)|\leq B_H\log a\) and \(\sup_{s>0}s|G_a'(s)|\leq C_H\log a\). The coefficient estimate used in Equation~\eqref{eq:entropy-fixed-scale-generator}, applied to \(G_a\), therefore gives
\begin{equation}\label{eq:entropy-scale-comparison}
\operatorname{Lip}(T_r-T_s)
\leq
K_H\log(r/s)
\qquad
(s\leq r\leq2s).
\end{equation}

Define \(g(0):=0\) and \(g(\xi):=T_{\rho(\xi,0)}(\xi)\) for \(\xi\neq0\). For a nonzero persistence diagram \(\alpha\), we have \(\rho(\alpha,0)=W_1(\alpha,0)=\sum_{I\in\alpha}|I|=L_\alpha\). Hence \(|I|/L_\alpha\leq1\), so \(H(|I|/L_\alpha)=h(|I|/L_\alpha)\), and therefore \(g(\alpha)=S_\alpha\).

We prove that \(g\) is Lipschitz. Fix \(\xi,\eta\in K(X,A)\), set \(r:=\rho(\xi,0)\), \(s:=\rho(\eta,0)\), and \(\delta:=\rho(\xi,\eta)\), and assume \(r\geq s\). Then \(0\leq r-s\leq\delta\). If \(s=0\), then \(\eta=0\) and \(r=\delta\), so \(\|g(\xi)-g(\eta)\|_{L^1}=\|T_r(\xi)\|_{L^1}\leq K_H\delta\). If \(0<s\leq r\leq2s\), Equation~\eqref{eq:entropy-scale-comparison} gives
\(\|g(\xi)-g(\eta)\|_{L^1}\leq K_H\delta+K_Hs\log(r/s)\leq K_H\delta+K_H(r-s)\leq2K_H\delta\).
If \(r>2s\), then \(s<r-s\leq\delta\), and
\(\|g(\xi)-g(\eta)\|_{L^1}\leq K_H\delta+(\operatorname{Lip}(T_r)+\operatorname{Lip}(T_s))s\leq3K_H\delta\).
Thus \(g\) is a Lipschitz extension of \(S\) to \(K(X,A)\).

Assume now that \(X/A\) is separable. Every \(g(\xi)\) is a finite real linear combination of interval indicators. Set \(\mathcal S_{\mathbb R}:=\operatorname{span}_{\mathbb R}\{\mathbf1_{[p,q)}:p<q\}\subseteq L^1(\mathbb R)\), so \(g(K(X,A))\subseteq\mathcal S_{\mathbb R}\).

We construct a Lipschitz map that vanishes on \(D(X,A)\) and whose scalarizations agree with indicator functions of singletons on finite subsets of \(K(Z,A)\) for every uniformly discrete subpair \(A\subseteq Z\subseteq X\). By Lemma~\ref{lem:collapsed-subset-pseudometric}, \(\delta_D(x,y):=\min\{\rho(x,y),\operatorname{dist}(x,D(X,A))+\operatorname{dist}(y,D(X,A))\}\) defines a pseudometric on \(K(X,A)\). Let \(Q\) be its metric quotient, let \(\kappa\colon K(X,A)\to Q\) be the quotient map, and set \(q_0:=\kappa(D(X,A))\). Then \(\kappa\) is \(1\)-Lipschitz, and Lemma~\ref{lem:separable-diagrams-and-completion} implies that \(Q\) is separable.

Choose a dense sequence \((q_n)_{n\geq1}\subseteq Q\) and define
\[
j(q)
:=
\left(
2^{-n}
\bigl(
d_Q(q,q_n)-d_Q(q_0,q_n)
\bigr)
\right)_{n\geq1}
\in\ell^2(\mathbb N;\mathbb R).
\]
The reverse triangle inequality gives \(\|j(q)-j(r)\|_{\ell^2}\leq(\sum_{n\geq1}4^{-n})^{1/2}d_Q(q,r)\), so \(j\) is Lipschitz and \(j(q_0)=0\). If \(q\neq r\), choose \(n\) such that \(d_Q(q,q_n)<d_Q(q,r)/2\). Then the triangle inequality gives \(d_Q(r,q_n)>d_Q(q,r)/2>d_Q(q,q_n)\), so the corresponding coordinates of \(j(q)\) and \(j(r)\) differ. Hence \(j\) is injective.

Let \(k(u,v):=\exp(-\|u-v\|_{\ell^2}^2)\) be the Gaussian kernel on \(\ell^2(\mathbb N;\mathbb R)\), let \(\mathcal H_k\) be its real reproducing kernel Hilbert space, and set \(\Phi(u):=k(u,\cdot)\). By \cite[Theorem~3.1]{guella2022gaussian}, \(k\) is strictly positive definite. Hence the vectors \(\Phi(u_1),\ldots,\Phi(u_m)\) are linearly independent whenever \(u_1,\ldots,u_m\) are distinct.

Define \(\Psi(q):=\Phi(j(q))-\Phi(j(q_0))\). Then \(\Psi(q_0)=0\), and \(\|\Psi(q)-\Psi(r)\|_{\mathcal H_k}^2=2(1-e^{-\|j(q)-j(r)\|_{\ell^2}^2})\leq2\|j(q)-j(r)\|_{\ell^2}^2\), so \(\Psi\) is Lipschitz. If distinct \(q_1,\ldots,q_m\) differ from \(q_0\) and \(\sum_i c_i\Psi(q_i)=0\), then \(\sum_i c_i\Phi(j(q_i))-(\sum_i c_i)\Phi(j(q_0))=0\). Strict positive definiteness and injectivity of \(j\) imply \(c_1=\cdots=c_m=0\). Thus the vectors \(\Psi(q)\), \(q\neq q_0\), are finitely linearly independent.

Let \(\mathcal H_0\) be the closed real linear span of \(\Psi(Q)\). Since \(Q\) is separable and \(\Psi\) is continuous, \(\mathcal H_0\) is separable. Choose an isometric real-linear embedding \(U\colon\mathcal H_0\to\ell^2(\mathbb N;\mathbb R)\). For each \(n\geq1\), choose \(\varphi_n\in C_c((n,n+1);\mathbb R)\setminus\{0\}\) with \(\|\varphi_n\|_{L^1}=1\), and define the bounded injective real-linear map \(J\colon\ell^2(\mathbb N;\mathbb R)\to L^1(\mathbb R)\) by
\[
J(c)
:=
\sum_{n=1}^\infty2^{-n}c_n\varphi_n.
\]
Since the supports are disjoint, \(\|J(c)\|_{L^1}=\sum_{n\geq1}2^{-n}|c_n|\leq(\sum_{n\geq1}4^{-n})^{1/2}\|c\|_{\ell^2}\). Every element of \(J(\ell^2(\mathbb N;\mathbb R))\) has a continuous representative, and
\begin{equation}\label{eq:entropy-step-continuous-intersection}
\mathcal S_{\mathbb R}\cap J(\ell^2(\mathbb N;\mathbb R))
=
\{0\}.
\end{equation}
Indeed, every element of \(\mathcal S_{\mathbb R}\) is a compactly supported finite step function. If such a function agrees almost everywhere with a continuous function, continuity forces the constant values on adjacent intervals to agree at every endpoint. The function vanishes outside a compact set, so all of these constant values equal zero.

Define \(\omega(x):=JU\Psi(\kappa(x))\) and \(\overline S:=g+\omega\). The map \(\omega\) is Lipschitz and vanishes on \(D(X,A)\), so \(\overline S\colon K(X,A)\to L^1(\mathbb R)\) is a Lipschitz extension of \(S\).

We form \(V_Z(\overline S)\) after complexifying \(L^1(\mathbb R)\). Taking real and imaginary parts shows that Equation~\eqref{eq:entropy-step-continuous-intersection} remains valid after complexification. The same argument shows that the vectors \(\Psi(q)\), \(q\neq q_0\), remain finitely linearly independent over \(\mathbb C\).

Fix a uniformly discrete subpair \(A\subseteq Z\subseteq X\). If \(K(Z,A)=\{0\}\), then \(V_Z(\overline S)\) is synthesizable. Assume \(K(Z,A)\neq\{0\}\). Choose \(\eta>0\) such that \(\overline d_1([v],[w])\geq\eta\) for distinct classes in \(Z/A\). Choose \(z\in Z\setminus A\), identify \(z\) with its singleton persistence diagram, and set \(a:=0-z\in K(Z,A)\). By Lemma~\ref{lem:subpair-grothendieck-isometry}, \(i_Z\) is isometric. For every \(\gamma\in D(X,A)\), Equation~\eqref{eq:groth-metric} gives
\[
\rho(i_Za,\gamma)
=
W_1(0,z+\gamma)
\geq
d(z,A)
\geq
\eta.
\]
Hence \(\operatorname{dist}(i_Za,D(X,A))\geq\eta\), so \(\kappa(i_Za)\neq q_0\). If \(x\neq a\), then \(\rho(i_Za,i_Zx)>0\), while the second quantity in the minimum defining \(\delta_D(i_Za,i_Zx)\) is at least \(\eta\). Hence \(\kappa(i_Za)\neq\kappa(i_Zx)\).

Let \(\mathcal F\subseteq K(Z,A)\) be finite and contain \(a\). We claim that
\begin{equation}\label{eq:entropy-finite-linear-separation}
\overline S(i_Za)
\notin
\operatorname{span}_{\mathbb C}
\{
\overline S(i_Zx):x\in\mathcal F\setminus\{a\}
\}.
\end{equation}
Suppose that \(c_a\overline S(i_Za)+\sum_{x\in\mathcal F\setminus\{a\}}c_x\overline S(i_Zx)=0\) with \(c_a\neq0\). The complexified version of Equation~\eqref{eq:entropy-step-continuous-intersection} separates the \(\mathcal S_{\mathbb C}\)-valued and \(J(\ell^2_{\mathbb C})\)-valued summands, so
\[
c_a\omega(i_Za)
+
\sum_{x\in\mathcal F\setminus\{a\}}
c_x\omega(i_Zx)
=
0.
\]
Injectivity of the complexifications of \(J\) and \(U\) gives a relation among the vectors \(\Psi(\kappa(i_Zx))\). Group the terms with equal values of \(\kappa(i_Zx)\). Since \(\kappa(i_Za)\) differs from \(q_0\) and from every \(\kappa(i_Zx)\) with \(x\neq a\), finite linear independence of the vectors \(\Psi(q)\), \(q\neq q_0\), implies \(c_a=0\), a contradiction. This proves Equation~\eqref{eq:entropy-finite-linear-separation}.

For each such finite \(\mathcal F\), Equation~\eqref{eq:entropy-finite-linear-separation} defines a complex-linear functional on \(\operatorname{span}_{\mathbb C}\{\overline S(i_Zx):x\in\mathcal F\}\) that takes value \(1\) at \(\overline S(i_Za)\) and \(0\) at \(\overline S(i_Zx)\) for \(x\in\mathcal F\setminus\{a\}\). Every linear functional on this finite-dimensional normed space is bounded. By the complex Hahn--Banach theorem, it extends to a bounded complex-linear functional \(\lambda_{\mathcal F}\) on the complexification of \(L^1(\mathbb R)\), and
\[
(\lambda_{\mathcal F}\circ\overline S\circ i_Z)|_{\mathcal F}
=
1_{\{a\}}|_{\mathcal F}.
\]

Direct the finite subsets \(\mathcal F\subseteq K(Z,A)\) containing \(a\) by inclusion. The corresponding scalarizations converge pointwise to \(1_{\{a\}}\). Since \(V_Z(\overline S)\) is pointwise closed, \(1_{\{a\}}\in V_Z(\overline S)\). Translation invariance gives the indicator function of every singleton. By Lemma~\ref{lem:point-masses-give-synthesis}, \(V_Z(\overline S)\) is synthesizable. Since \(Z\) was arbitrary, \(\overline S\) has uniformly discrete spectral synthesis.
\end{proof}

\subsection{Betti Curves}

\begin{figure}[t]
    \centering
    \includegraphics[width=0.72\textwidth]{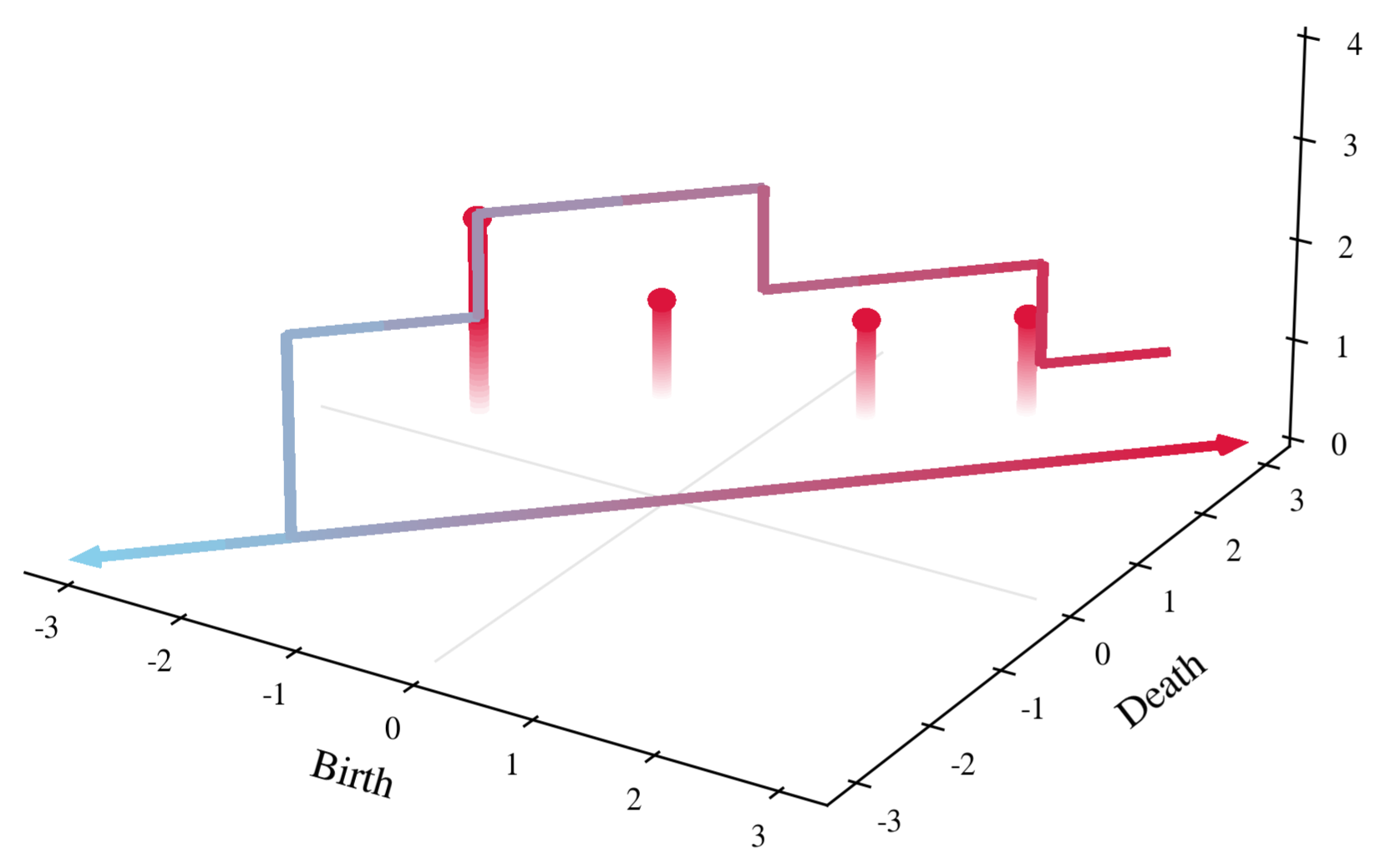}
    \caption{Betti curve of a persistence diagram.}
    \label{fig:betti-curve}
\end{figure}

\begin{proposition}\label{prop:betti-synthesis}
The Betti curve vectorization from \cite[Definition~3.6]{ali2023survey}, introduced in \cite{berry2018functional}, has a Lipschitz extension to the virtual persistence diagram group with uniformly discrete spectral synthesis.
\end{proposition}

\begin{proof}
Let \(b\colon D(X,A)\to L^1(\mathbb R)\) denote the Betti curve vectorization. By \cite[Definition~3.6]{ali2023survey},
\[
b\left(\sum_{k=1}^n m_k[p_k,q_k]\right)
=
\sum_{k=1}^n m_k\mathbf1_{[p_k,q_k)}.
\]
Let \(I=[p,q]\) and \(J=[p',q']\). Since the symmetric difference of \([p,q)\) and \([p',q')\) has measure at most \(|p-p'|+|q-q'|\), while \(\|\mathbf1_{[p,q)}\|_{L^1}=q-p=d(I,A)\),
\begin{align*}
\|\mathbf1_{[p,q)}-\mathbf1_{[p',q')}\|_{L^1}
&\leq
\min\{
|p-p'|+|q-q'|,
d(I,A)+d(J,A)
\}\\
&=
d_1(I,J).
\end{align*}
If \(\sigma=\sum_i(I_i,J_i)\) matches \(\alpha,\gamma\), then
\begin{align*}
\|b(\alpha)-b(\gamma)\|_{L^1}
&\leq
\sum_i
\|\mathbf1_{I_i}-\mathbf1_{J_i}\|_{L^1}\\
&\leq
\sum_i d_1(I_i,J_i).
\end{align*}
Taking the infimum gives \(\|b(\alpha)-b(\gamma)\|_{L^1}\leq W_1(\alpha,\gamma)\). Since \(b\) is additive, Lemma~\ref{lem:additive-extension-synthesis} gives the asserted extension and synthesis.
\end{proof}

\subsection{Lifespan Curves}

\begin{figure}[t!b]
    \centering
    \includegraphics[width=0.72\textwidth]{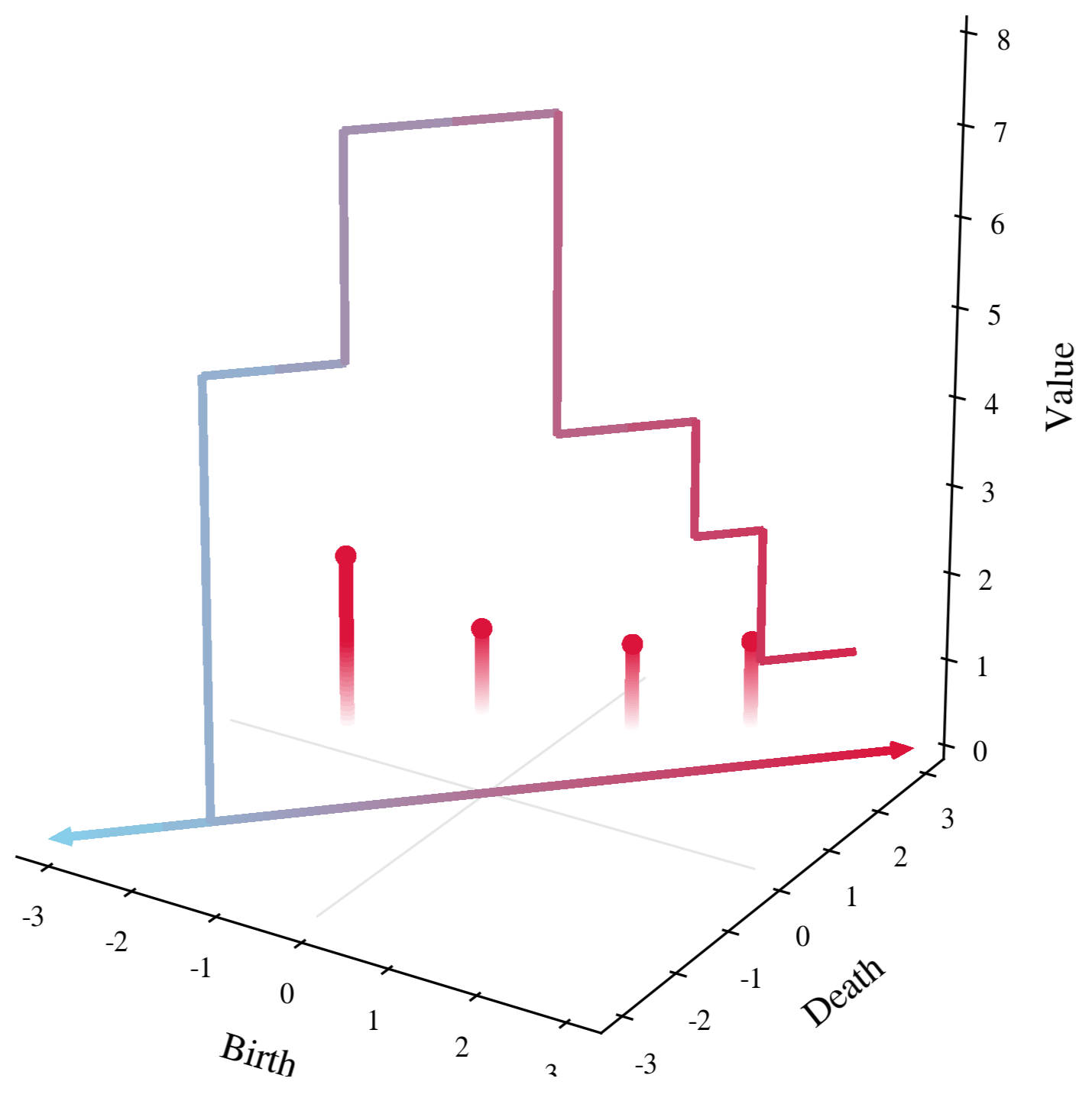}
    \caption{Lifespan curve of a persistence diagram.}
    \label{fig:lifespan-curve}
\end{figure}

\begin{proposition}\label{prop:lifespan-no-extension}
For \(1\leq p\leq\infty\), the lifespan curve vectorization \(D(X,A)\to L^p(\mathbb R)\) does not have a Lipschitz extension to \(K(X,A)\).
\end{proposition}

\begin{proof}
Fix \(p\in[1,\infty]\), and let \(L\colon D(X,A)\to L^p(\mathbb R)\) denote the lifespan curve vectorization. By \cite[Definition~3.7]{ali2023survey},
\[
L\left(
\sum_{k=1}^n m_k[p_k,q_k]
\right)(t)
=
\sum_{k=1}^n
m_k(q_k-p_k)\mathbf1_{[p_k,q_k)}(t).
\]

Suppose first that \(1\leq p<\infty\). Set \(R_n:=n\) and \(\varepsilon_n:=R_n^{-2p}\). Matching \([0,R_n]\) with \([\varepsilon_n,R_n+\varepsilon_n]\) gives \(0<W_1\leq2\varepsilon_n\), while the two lifespan curves are \(R_n\mathbf1_{[0,R_n)}\) and \(R_n\mathbf1_{[\varepsilon_n,R_n+\varepsilon_n)}\). Hence
\[
\frac{
\|L([0,R_n])-L([\varepsilon_n,R_n+\varepsilon_n])\|_{L^p}
}{
W_1([0,R_n],[\varepsilon_n,R_n+\varepsilon_n])
}
\geq
2^{1/p-1}R_n^{2p-1}
\longrightarrow\infty.
\]

For \(p=\infty\), set \(R_n:=n\) and \(\varepsilon_n:=R_n^{-2}\). The same matching has cost at most \(2\varepsilon_n\), while the \(L^\infty\)-difference equals \(R_n\), so
\[
\frac{
\|L([0,R_n])-L([\varepsilon_n,R_n+\varepsilon_n])\|_{L^\infty}
}{
W_1([0,R_n],[\varepsilon_n,R_n+\varepsilon_n])
}
\geq
\frac12R_n^3
\longrightarrow\infty.
\]
Lemma~\ref{lem:no-lipschitz-extension-obstruction} shows that the lifespan curve is not Lipschitz with respect to \(W_1\) and has no Lipschitz extension to \(K(X,A)\).
\end{proof}

\subsection{Persistence Landscapes}

\begin{figure}[t!b]
    \centering
    \includegraphics[width=0.72\textwidth]{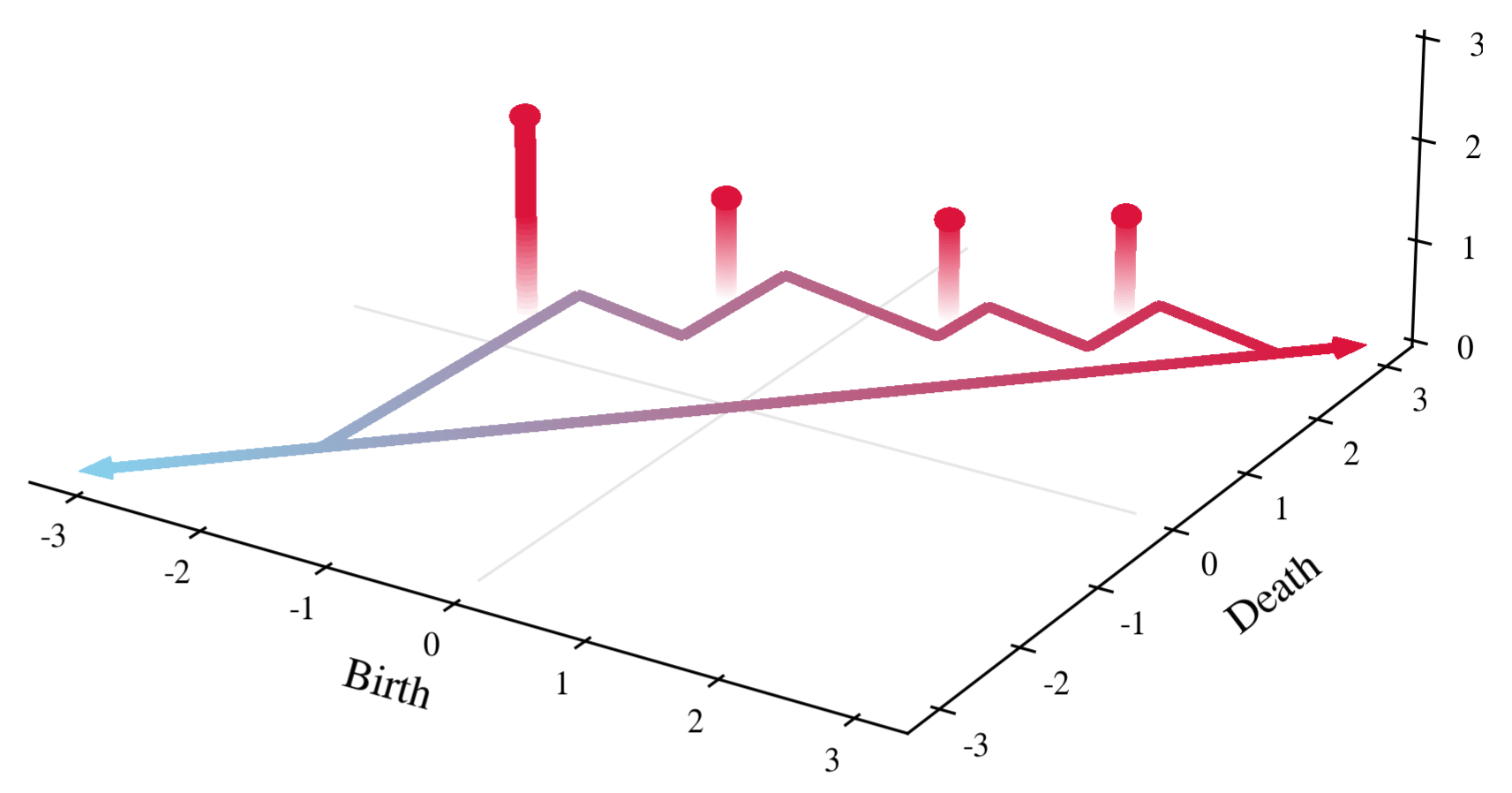}
    \caption{First persistence landscape of a persistence diagram.}
    \label{fig:persistence-landscape}
\end{figure}

\begin{proposition}\label{prop:landscape-synthesis}
For \(1\leq p\leq\infty\), use the persistence landscape vectorization from \cite[Definition~3.8]{ali2023survey}, introduced in \cite{bubenik2015statistical}, with the \(L^p(\mathbb N\times\mathbb R;\mathbb C)\)-norm. The resulting vectorization is Lipschitz with respect to \(W_1\) if and only if \(p=\infty\). If \(X/A\) is separable, then the \(L^\infty(\mathbb N\times\mathbb R;\mathbb C)\)-valued persistence landscape has a Lipschitz extension with uniformly discrete spectral synthesis.
\end{proposition}

\begin{proof}
We use counting measure on \(\mathbb N\) and Lebesgue measure on \(\mathbb R\).

Fix \(1\leq p<\infty\). Set \(R_n:=n\), \(\varepsilon_n:=R_n^{-2}\), \(B_n:=[0,R_n]\), and \(C_n:=[\varepsilon_n,R_n+\varepsilon_n]\). Matching the persistence intervals gives \(0<W_1(B_n,C_n)\leq2\varepsilon_n\). By \cite[Definition~3.8]{ali2023survey}, the persistence diagram consisting of only the interval \([a,b]\) satisfies \(\Lambda_{[a,b],1}(t)=\max\{\min(t-a,b-t),0\}\) and \(\Lambda_{[a,b],k}=0\) for \(k\geq2\). Thus \(\Lambda_{B_n,1}(t)-\Lambda_{C_n,1}(t)=\varepsilon_n\) on \([\varepsilon_n,R_n/2]\), and
\begin{align*}
\|\Lambda^{(p)}(B_n)-\Lambda^{(p)}(C_n)\|_{L^p}
&\geq
\varepsilon_n(R_n/2-\varepsilon_n)^{1/p},\\
\frac{
\|\Lambda^{(p)}(B_n)-\Lambda^{(p)}(C_n)\|_{L^p}
}{
W_1(B_n,C_n)
}
&\geq
\frac12
\left(
\frac{R_n}{2}-\varepsilon_n
\right)^{1/p}
\longrightarrow\infty.
\end{align*}
Hence \(\Lambda^{(p)}\) is not Lipschitz with respect to \(W_1\) for \(p<\infty\).

Suppose \(p=\infty\). By \cite[Theorem~13]{bubenik2015statistical}, \(\|\Lambda^{(\infty)}(\alpha)-\Lambda^{(\infty)}(\beta)\|_{L^\infty}\leq d_B(\alpha,\beta)\). For off-diagonal intervals \(I=[a,b]\) and \(J=[c,d]\), set \(d_\infty(I,J):=\max\{|a-c|,|b-d|\}\) and \(d_\infty(I,A):=(b-a)/2\).

Fix a matching between \(\alpha\) and \(\beta\). Retain a matched off-diagonal pair \(I,J\) when \(d_1(I,J)=d(I,J)\), and otherwise replace it by matches from \(I\) and \(J\) to the diagonal. In the first case, \(d_\infty(I,J)\leq d_1(I,J)\). In the second, \(d_\infty(I,A)\leq d(I,A)\leq d_1(I,J)\) and similarly for \(J\). Thus the resulting bottleneck matching has maximum cost at most the total \(d_1\)-cost of the original matching. Taking infima gives \(d_B(\alpha,\beta)\leq W_1(\alpha,\beta)\), hence \(\|\Lambda^{(\infty)}(\alpha)-\Lambda^{(\infty)}(\beta)\|_{L^\infty}\leq W_1(\alpha,\beta)\).

Assume that \(X/A\) is separable. For \(x\in K(X,A)\), \(k\in\mathbb N\), and \(t\in\mathbb R\), define
\[
g(x)_k(t)
:=
\inf_{\alpha\in D(X,A)}
\left\{
\Lambda^{(\infty)}(\alpha)_k(t)+\rho(x,\alpha)
\right\}.
\]
For fixed \(x,k\), this function is upper semicontinuous and hence Borel measurable. Nonnegativity of the landscape coordinates and the choice \(\alpha=0\) give \(0\leq g(x)_k(t)\leq\rho(x,0)\). Hence \(g(x)\in L^\infty(\mathbb N\times\mathbb R;\mathbb C)\). The triangle inequality gives \(|g(x)_k(t)-g(y)_k(t)|\leq\rho(x,y)\), hence \(\|g(x)-g(y)\|_{L^\infty}\leq\rho(x,y)\).

For \(\beta\in D(X,A)\), the \(1\)-Lipschitz estimate for \(\Lambda^{(\infty)}\) gives \(\Lambda^{(\infty)}(\alpha)_k(t)+\rho(\beta,\alpha)\geq\Lambda^{(\infty)}(\beta)_k(t)\). Taking the infimum and then choosing \(\alpha=\beta\) gives \(g|_{D(X,A)}=\Lambda^{(\infty)}\).

For \(k\geq1\), let \(D_{<k}\) consist of the persistence diagrams with fewer than \(k\) persistence intervals, counted with multiplicity. Since the \(k\)-th landscape coordinate vanishes on \(D_{<k}\),
\[
\operatorname{dist}(x,D(X,A))
\leq
g(x)_k(t)
\leq
\operatorname{dist}(x,D_{<k}).
\]
Since \(D_{<k}\) increases to \(D(X,A)\),
\[
\sup_{t\in\mathbb R}
\left|
g(x)_k(t)-\operatorname{dist}(x,D(X,A))
\right|
\leq
\operatorname{dist}(x,D_{<k})
-
\operatorname{dist}(x,D(X,A))
\longrightarrow0.
\]

We now add a Lipschitz map that vanishes on \(D(X,A)\) and whose scalarizations recover indicator functions of singletons on uniformly discrete restrictions.

By Lemma~\ref{lem:collapsed-subset-pseudometric}, \(\delta(x,y):=\min\{\rho(x,y),\operatorname{dist}(x,D(X,A))+\operatorname{dist}(y,D(X,A))\}\) defines a pseudometric. Let \(Q\) be its metric quotient, let \(\kappa\colon K(X,A)\to Q\) be the quotient map, and take \(\kappa(D(X,A))\) as the base point. Then
\begin{equation}
\label{eq:landscape-quotient-pseudometric}
d_Q(\kappa(x),\kappa(y))
=
\min\left\{
\rho(x,y),
\operatorname{dist}(x,D(X,A))
+
\operatorname{dist}(y,D(X,A))
\right\}.
\end{equation}
Thus \(\kappa\) is \(1\)-Lipschitz. Lemma~\ref{lem:separable-diagrams-and-completion} gives separability of \(Q\), so \(\mathcal F_{\mathbb C}(Q)\) is separable.

Choose a weak-* dense sequence \((\phi_j)_{j\geq1}\) in the closed unit ball of \(\mathcal F_{\mathbb C}(Q)^*\), and define \(J_0(v):=(\phi_j(v))_{j\geq1}\). Then
\begin{align*}
\|J_0(v)\|_{\ell^\infty}
&=\sup_j|\phi_j(v)|\\
&=\sup_{\phi\in B_{\mathcal F_{\mathbb C}(Q)^*}}|\phi(v)|\\
&=\|v\|,
\end{align*}
so \(J_0\) is a complex-linear isometry.

Choose \(r\colon\mathbb N\to\mathbb N\) with infinite fibers and define \(\mathcal R(v)_n:=v_{r(n)}\). Then \(\|\mathcal R(v)\|_{\ell^\infty}=\|v\|_{\ell^\infty}\). For the quotient map \(\pi\colon\ell^\infty\to\ell^\infty/c_0\), we also have \(\|\pi\mathcal R(v)\|=\|v\|_{\ell^\infty}\). Thus \(T:=\pi\mathcal RJ_0\colon\mathcal F_{\mathbb C}(Q)\to\ell^\infty/c_0\) is a complex-linear isometry.

For \(x\in K(X,A)\), set \(c(x):=\mathcal RJ_0\delta_Q(\kappa(x))\), and define
\[
h(x)(2n,t):=c(x)_n,
\qquad
h(x)(2n-1,t):=-c(x)_n.
\]
Then
\begin{align*}
\|h(x)-h(y)\|_{L^\infty}
&=
\|c(x)-c(y)\|_{\ell^\infty}\\
&=
\left\|
\delta_Q(\kappa(x))-\delta_Q(\kappa(y))
\right\|_{\mathcal F_{\mathbb C}(Q)}\\
&=
d_Q(\kappa(x),\kappa(y))
\leq\rho(x,y).
\end{align*}
Hence \(h\) is \(1\)-Lipschitz and vanishes on \(D(X,A)\). Set \(\overline\Lambda:=g+h\). Then \(\overline\Lambda|_{D(X,A)}=\Lambda^{(\infty)}\) and \(\operatorname{Lip}(\overline\Lambda)\leq2\).

For \(F\in L^\infty(\mathbb N\times\mathbb R;\mathbb C)\), define
\[
A_0(F)_n
:=
\frac12
\int_0^1
\bigl(
F_{2n}(t)-F_{2n-1}(t)
\bigr)\,dt.
\]
Then \(\|A_0\|\leq1\). Since \(g(x)_k\) converges uniformly to \(\operatorname{dist}(x,D(X,A))\), we have \(|A_0(g(x))_n|\leq\frac12\|g(x)_{2n}-g(x)_{2n-1}\|_{L^\infty([0,1])}\to0\). Thus \(A_0(g(x))\in c_0\), while \(A_0(h(x))=c(x)\). Hence \(\pi A_0(g(x))=0\) and \(\pi A_0(h(x))=T\delta_Q(\kappa(x))\).

For \(u\in\operatorname{Lip}_0(Q;\mathbb C)\), let \(\widehat u\in\mathcal F_{\mathbb C}(Q)^*\) be its linearization. Define \(\theta_u(Tv):=\widehat u(v)\), extend \(\theta_u\) to \(\widetilde\theta_u\in(\ell^\infty/c_0)^*\) by the complex Hahn--Banach theorem, and set \(\lambda_u:=\widetilde\theta_u\circ\pi\circ A_0\). Then
\begin{equation}
\label{eq:landscape-scalar-recovery}
\begin{aligned}
\lambda_u(\overline\Lambda(x))
&=
\widetilde\theta_u\bigl(
\pi A_0(\overline\Lambda(x))
\bigr)\\
&=
\theta_u\bigl(
T\delta_Q(\kappa(x))
\bigr)\\
&=
\widehat u\bigl(
\delta_Q(\kappa(x))
\bigr)\\
&=
u(\kappa(x)).
\end{aligned}
\end{equation}

Fix a uniformly discrete subpair \(A\subseteq Z\subseteq X\). If \(K(Z,A)=\{0\}\), then every function on \(K(Z,A)\) is constant, so \(V_Z(\overline\Lambda)\) is synthesizable. Assume \(K(Z,A)\neq\{0\}\). Choose \(\eta>0\) such that \(\overline d_1([v],[w])\geq\eta\) for distinct classes in \(Z/A\), choose \(z\in Z\setminus A\), and set \(a:=0-z\). For every \(\gamma\in D(X,A)\), Equation~\eqref{eq:groth-metric} gives \(\rho(i_Za,\gamma)=W_1(0,z+\gamma)\geq d(z,A)\geq\eta\), so \(\operatorname{dist}(i_Za,D(X,A))\geq\eta\). Since \(K(Z,A)\) is discrete and \(i_Z\) is isometric, some \(\delta>0\) satisfies \(\rho(i_Za,i_Zx)\geq\delta\) for \(x\neq a\). Equation~\eqref{eq:landscape-quotient-pseudometric} gives
\[
d_Q(\kappa(i_Za),\kappa(i_Zx))
\geq
\min\{\delta,\eta\}
\qquad
(x\neq a).
\]
Thus \(\kappa(i_Za)\) is isolated and differs from the base point. Its indicator on \((\kappa\circ i_Z)(K(Z,A))\) is Lipschitz and vanishes at the base point, so the McShane extension theorem gives \(\widetilde u_a\in\operatorname{Lip}_0(Q;\mathbb C)\). Equation~\eqref{eq:landscape-scalar-recovery} gives \(\lambda_{\widetilde u_a}\circ\overline\Lambda\circ i_Z=1_{\{a\}}\). Translation invariance gives the indicator function of every singleton, and Lemma~\ref{lem:point-masses-give-synthesis} gives synthesis. Since \(Z\) was arbitrary, \(\overline\Lambda\) has uniformly discrete spectral synthesis.
\end{proof}

\subsection{Persistence Silhouettes}

\begin{figure}[t]
    \centering
    \includegraphics[width=0.72\textwidth]{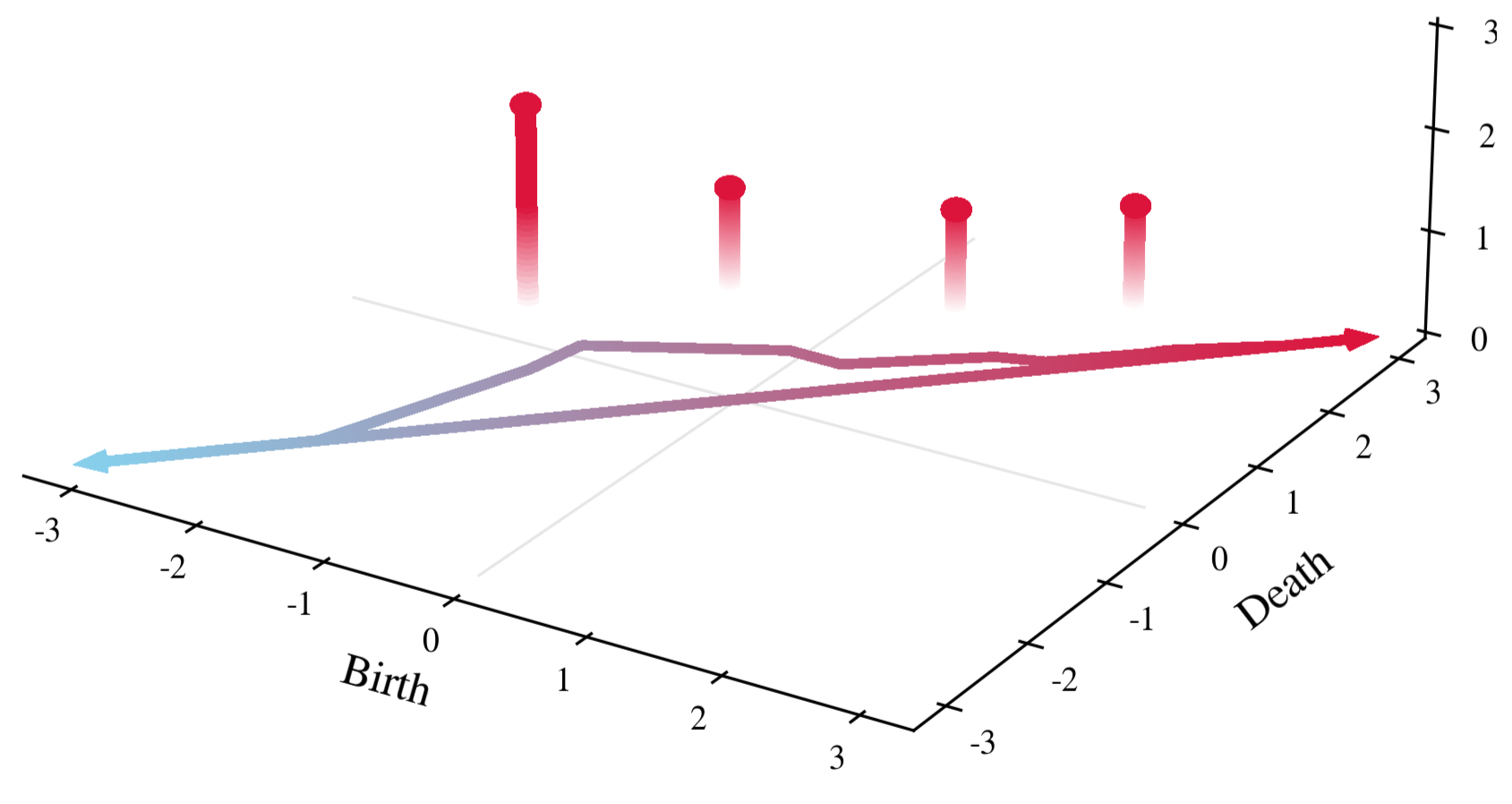}
    \caption{Persistence silhouette of a persistence diagram.}
    \label{fig:persistence-silhouette}
\end{figure}

\begin{proposition}\label{prop:silhouette-synthesis}
For \(1\leq p\leq\infty\), use the \(w\)-weighted persistence silhouette vectorization from \cite[Definition~3.9]{ali2023survey}, introduced in \cite{chazal2014stochastic}, with the \(L^p(\mathbb R;\mathbb C)\)-norm. For every positive weight function \(w\colon X\setminus A\to\mathbb R_{>0}\), the resulting vectorization is not Lipschitz with respect to \(W_1\) when \(1\leq p<\infty\). If \(X/A\) is separable and the \(L^\infty(\mathbb R;\mathbb C)\)-valued silhouette is Lipschitz with respect to \(W_1\), then it has a Lipschitz extension with uniformly discrete spectral synthesis.
\end{proposition}

\begin{proof}
Fix \(1\leq p<\infty\). Set \(R_n:=n\), \(\varepsilon_n:=R_n^{-2}\), \(B_n:=[0,R_n]\), and \(C_n:=[\varepsilon_n,R_n+\varepsilon_n]\). Since each persistence diagram contains one persistence interval, the weight cancels in the normalized average. By \cite[Definition~3.9]{ali2023survey}, the silhouettes are the corresponding tent functions, so \(\phi^w_{B_n}(t)-\phi^w_{C_n}(t)=\varepsilon_n\) on \([\varepsilon_n,R_n/2]\). Matching the persistence intervals gives \(0<W_1(B_n,C_n)\leq2\varepsilon_n\), and
\begin{align*}
\|\phi^w_{B_n}-\phi^w_{C_n}\|_{L^p}
&\geq
\varepsilon_n(R_n/2-\varepsilon_n)^{1/p},\\
\frac{
\|\phi^w_{B_n}-\phi^w_{C_n}\|_{L^p}
}{
W_1(B_n,C_n)
}
&\geq
\frac12
\left(
\frac{R_n}{2}-\varepsilon_n
\right)^{1/p}
\longrightarrow\infty.
\end{align*}
Thus the \(L^p\)-valued silhouette is not Lipschitz with respect to \(W_1\) for \(p<\infty\).

We define the silhouette of the empty persistence diagram to be the zero function. If \(X/A\) is separable and the \(L^\infty(\mathbb R;\mathbb C)\)-valued silhouette is Lipschitz with respect to \(W_1\), Lemma~\ref{lem:linfty-real-line-extension-synthesis} gives a Lipschitz extension with uniformly discrete spectral synthesis.
\end{proof}

\subsection{Persistence Images}

\begin{figure}[t]
    \centering
    \includegraphics[width=0.72\textwidth]{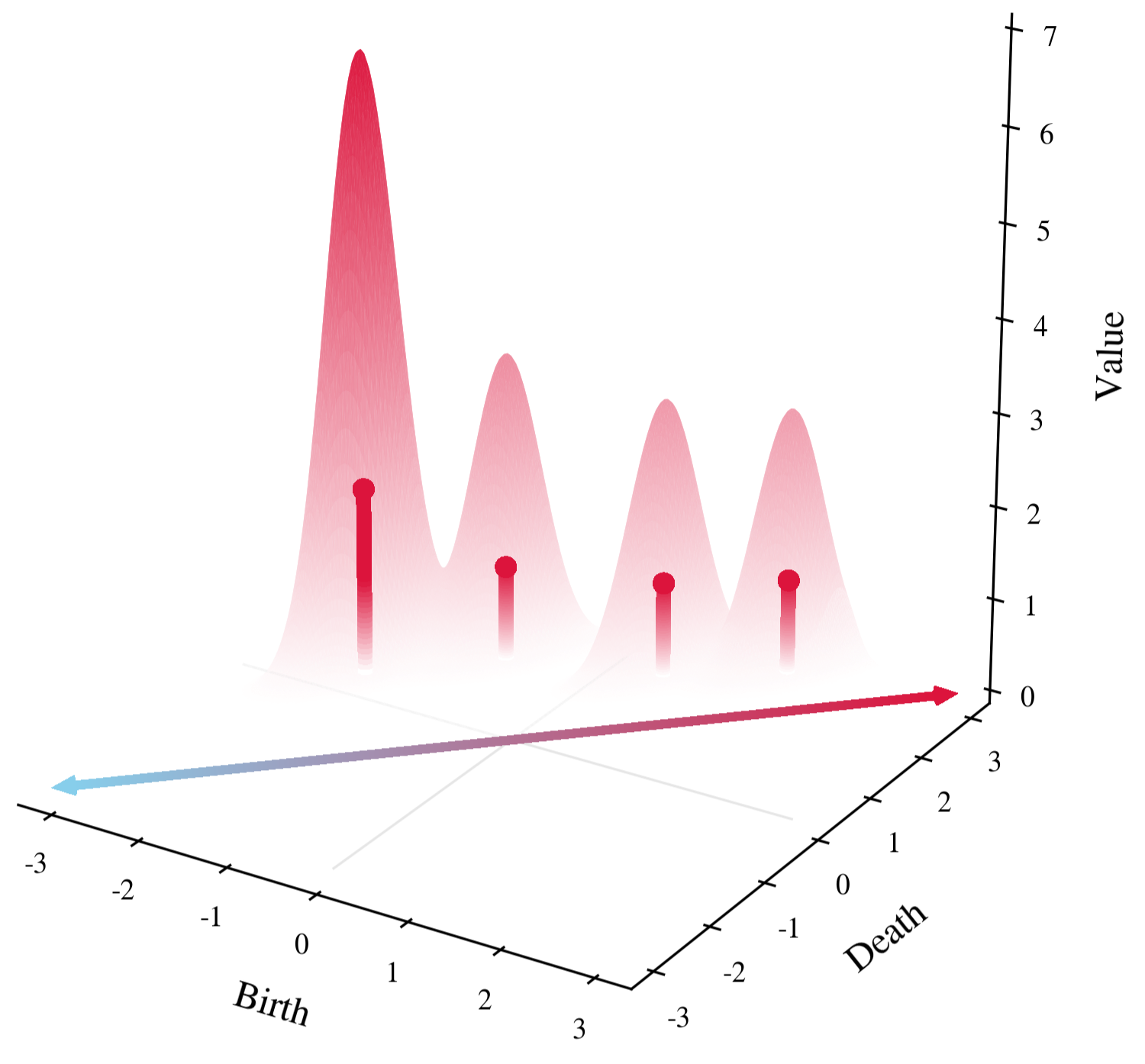}
    \caption{Persistence image of a persistence diagram.}
    \label{fig:persistence-image}
\end{figure}

\begin{proposition}\label{prop:persistence-image-synthesis}
The persistence image vectorization from \cite[Definition~3.10]{ali2023survey}, introduced in \cite{adams2017persistence}, has a Lipschitz extension to the virtual persistence diagram group with uniformly discrete spectral synthesis whenever it is Lipschitz with respect to the \(1\)-Wasserstein distance.
\end{proposition}

\begin{proof}
Let \(I_{f,\Psi,\mathcal P}\) be a persistence image vectorization that is Lipschitz with respect to \(W_1\). Definition~3.10 of \cite{ali2023survey} gives \(\rho^{\alpha+\beta}_{f,\Psi}=\rho^\alpha_{f,\Psi}+\rho^\beta_{f,\Psi}\). Integration over each pixel gives \(I_{f,\Psi,\mathcal P}(\alpha+\beta)=I_{f,\Psi,\mathcal P}(\alpha)+I_{f,\Psi,\mathcal P}(\beta)\). Thus \(I_{f,\Psi,\mathcal P}\) is additive, and Lemma~\ref{lem:additive-extension-synthesis} gives the asserted extension and synthesis.
\end{proof}


\subsection{Template Functions}

\begin{figure}[t]
    \centering
    \includegraphics[width=0.72\textwidth]{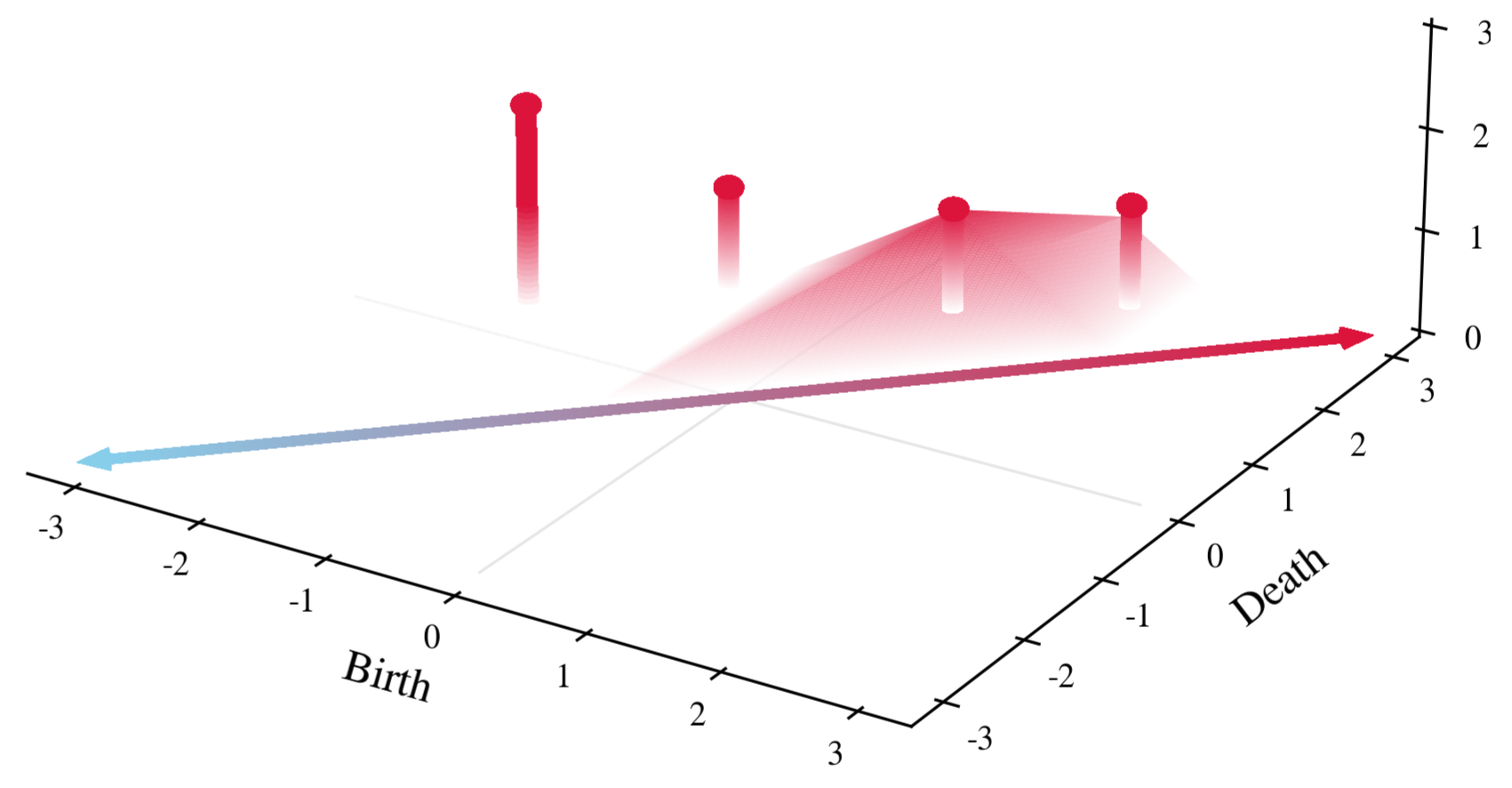}
    \caption{Tent template function on $\mathbb R^2$.}
    \label{fig:template-function}
\end{figure}

\begin{proposition}\label{prop:template-synthesis}
Every Lipschitz template-function vectorization from \cite[Definition~3.11]{ali2023survey} has a Lipschitz extension to the virtual persistence diagram group with uniformly discrete spectral synthesis.
\end{proposition}

\begin{proof}
Let \(\boldsymbol{\varphi}=(\varphi_1,\ldots,\varphi_n)\), and let \(\Phi_{\boldsymbol{\varphi}}\colon D(X,A)\to\mathbb C^n\) be the resulting Lipschitz vectorization. By \cite[Equation~(3)]{ali2023survey},
\[
V_\gamma(\varphi_j)
=
\sum_{[p,q]\in\operatorname{supp}\gamma}
m_\gamma([p,q])\varphi_j(p,q-p).
\]
Hence \(V_{\alpha+\beta}(\varphi_j)=V_\alpha(\varphi_j)+V_\beta(\varphi_j)\) for every \(j\), so \(\Phi_{\boldsymbol{\varphi}}\) is additive. Lemma~\ref{lem:additive-extension-synthesis} gives the asserted extension and synthesis.
\end{proof}

\subsection{Tropical Coordinates}

For \(r\in\mathbb Z_{>0}\), let \(\Phi_r=(F_{1,r},\ldots,F_{7,r})\) denote the seven tropical coordinates listed following \cite[Definition~3.4]{ali2023survey}, which give the \(r\)-parameterized version of the seven max-plus type coordinates used in \cite[Section~8]{kalivsnik2019tropical}. We evaluate these coordinates on the reduced off-diagonal representative of each persistence diagram, counting multiplicities, and use the Euclidean norm on \(\mathbb R^7\). For the first four coordinates, when fewer than four off-diagonal intervals are present, we use the zero-length padding convention for the corresponding ordered-bar formulas. The proposition below concerns only this seven-coordinate feature vector under these conventions and makes no assertion about the general tropical rational-function family studied in \cite{kalivsnik2019tropical}.

\begin{proposition}\label{prop:tropical-no-extension}
For every \(r\in\mathbb Z_{>0}\), the seven-coordinate feature vector \(\Phi_r\) is not Lipschitz with respect to \(W_1\) on the nonempty persistence diagrams. No extension of \(\Phi_r\) to \(D(X,A)\) is Lipschitz with respect to \(W_1\), and no such extension has a Lipschitz extension to \(K(X,A)\).
\end{proposition}

\begin{proof}
Fix \(r\in\mathbb Z_{>0}\). For an interval \(u=[p,q]\), set \(z_r(u):=\min\{r(q-p),p\}+q-p\). The seventh tropical coordinate is
\[
F_{7,r}(\alpha)
=
\sum_{u\in\operatorname{supp}\alpha}
m_\alpha(u)
\left(
\max_{v\in\operatorname{supp}\alpha}z_r(v)-z_r(u)
\right)
\]
for every nonempty persistence diagram \(\alpha\).

Let \(a:=[0,1]\), so \(z_r(a)=1\). For \(1\leq k\leq n\), set \(\delta_{n,k}:=k/n^4\) and \(e_{n,k}:=[\delta_{n,k},\delta_{n,k}+n^{-2}]\). Since \(\delta_{n,k}\leq n^{-3}\leq rn^{-2}\), we have \(z_r(e_{n,k})=\delta_{n,k}+n^{-2}\leq n^{-3}+n^{-2}<1\).

Set
\[
\alpha_n:=4a,
\qquad
\beta_n:=4a+\sum_{k=1}^n e_{n,k}.
\]
Since \(z_r(a)=1\) and \(z_r(e_{n,k})<1\), we have \(\max_{v\in\operatorname{supp}\beta_n}z_r(v)=1\). Each copy of \(a\) contributes zero to \(F_{7,r}\), so \(F_{7,r}(\alpha_n)=0\), while
\[
F_{7,r}(\beta_n)
=
\sum_{k=1}^n
\left(
1-z_r(e_{n,k})
\right)
\geq
n\left(
1-n^{-2}-n^{-3}
\right).
\]

Each \(e_{n,k}\) has lifespan \(n^{-2}\), hence distance \(n^{-2}\) from \(A\). We match the four copies of \(a\) in \(\alpha_n\) with the four copies of \(a\) in \(\beta_n\) and match each \(e_{n,k}\) to \(A\), obtaining \(0<W_1(\alpha_n,\beta_n)\leq n^{-1}\). Since the Euclidean norm dominates the absolute difference of the seventh coordinates,
\[
\frac{
\|\Phi_r(\beta_n)-\Phi_r(\alpha_n)\|_2
}{
W_1(\alpha_n,\beta_n)
}
\geq
n^2
\left(
1-n^{-2}-n^{-3}
\right)
\longrightarrow
\infty.
\]
Hence \(\Phi_r\) is not Lipschitz with respect to \(W_1\) on the nonempty persistence diagrams.

Let \(\widetilde\Phi_r\colon D(X,A)\to\mathbb R^7\) be any extension of \(\Phi_r\). Since every \(\alpha_n\) and \(\beta_n\) is nonempty, the preceding quotient remains unchanged for \(\widetilde\Phi_r\), so \(\widetilde\Phi_r\) is not Lipschitz with respect to \(W_1\). By Lemma~\ref{lem:no-lipschitz-extension-obstruction}, \(\widetilde\Phi_r\) has no Lipschitz extension to \(K(X,A)\).
\end{proof}

\subsection{Vectorization by a Schauder Basis}

\begin{proposition}\label{prop:schauder-basis-synthesis}
Let \((X,A)\) be a polyhedral pair, and let \(B\) be a totally ordered Schauder basis of \((\operatorname{Lip}_c(X,A),\ell^\infty)\) that satisfies the locally Lipschitz finite property with upper bound \(M\), as in \cite[Theorem~5.2]{bubenik2025schauder}. Let \(F_B\colon D(X,A)\to\ell^1\) be the restriction to \(D(X,A)\) of the vectorization by \(B\) from \cite[Definition~5.1]{bubenik2025schauder}. Then \(F_B\) has a Lipschitz extension to \(K(X,A)\) with uniformly discrete spectral synthesis.
\end{proposition}

\begin{proof}
By \cite[Definition~5.1]{bubenik2025schauder}, \(F_B(\alpha)=(\alpha(f))_{f\in B}\), where \(\alpha(f)=\sum_{x\in\alpha}f(x)\), counted with multiplicity. For \(\alpha,\beta\in D(X,A)\), linearity of evaluation gives \(F_B(\alpha+\beta)=F_B(\alpha)+F_B(\beta)\), so \(F_B\) is additive. By \cite[Theorem~5.2]{bubenik2025schauder}, \(\|F_B(\alpha)-F_B(\beta)\|_1\leq2M W_1(\alpha,\beta)\) for all \(\alpha,\beta\in D(X,A)\). Lemma~\ref{lem:additive-extension-synthesis} therefore gives a Lipschitz extension of \(F_B\) to \(K(X,A)\) with uniformly discrete spectral synthesis.
\end{proof}

\subsection{Signatures of Integrated Persistence Landscapes}

\begin{proposition}
\label{prop:integrated-landscape-signature-classification}
Let \(\iota_{iL}^{(N)}\) denote the projection of the integrated landscape embedding from \cite[Definition~3.3]{chevyrev2018persistence} onto its first \(N\) coordinates. For \(N,M\geq1\), define
\[
\Phi_{iL}^{(N,M)}(\alpha)
:=
\left(
S_m\!\left(
\iota_{iL}^{(N)}(\alpha)
\right)
\right)_{m=0}^{M}
\]
with codomain \(\bigoplus_{m=0}^{M}(\mathbb C^N)^{\otimes m}\), and give this codomain an arbitrary norm, where \((\mathbb C^N)^{\otimes0}:=\mathbb C\). Then \(\Phi_{iL}^{(N,M)}\) is not Lipschitz with respect to \(W_1\) and does not have a Lipschitz extension to \(K(X,A)\).

Let \(N\geq1\), \(M\in\mathbb Z_{\geq0}\), let \(E\) be a finite-dimensional complex normed space, and let \(T_m\colon(\mathbb C^N)^{\otimes m}\to E\) be complex-linear for \(0\leq m\leq M\). Define
\[
f(\alpha)
:=
\sum_{m=0}^{M}
T_m\!\left(
S_m\!\left(
\iota_{iL}^{(N)}(\alpha)
\right)
\right).
\]
Then \(f\) is Lipschitz with respect to \(W_1\) if and only if \(f\) is constant. Consequently, every Lipschitz function of this form has a Lipschitz extension to \(K(X,A)\) with uniformly discrete spectral synthesis.
\end{proposition}

\begin{proof}
For \(R>0\), define \(D_R([b,d]):=[Rb,Rd]\) and extend \(D_R\) multiplicitywise. Since \(d(D_Ru,D_Rv)=R\,d(u,v)\), \(D_R(A)=A\), and \(d(D_Ru,A)=R\,d(u,A)\), Equation~\eqref{eq:one-strengthening} gives \(d_1(D_Ru,D_Rv)=R\,d_1(u,v)\). Applying \(D_R\) and \(D_{1/R}\) to matchings gives
\begin{equation}
\label{eq:integrated-landscape-wasserstein-scaling}
W_1(D_R\alpha,D_R\beta)
=
R\,W_1(\alpha,\beta).
\end{equation}

By \cite[Definition~3.2]{chevyrev2018persistence}, \(\lambda_k^{D_R\alpha}(t)=R\lambda_k^\alpha(t/R)\). Definition~3.3 and substitution therefore give \(\iota_{iL}^{(N)}(D_R\alpha)(t)=R^2\iota_{iL}^{(N)}(\alpha)(t/R)\). Each persistence landscape has compact support, so each integrated landscape path is constant outside a compact interval. Increasing reparametrization preserves iterated Riemann--Stieltjes integrals, while multiplication of a path by \(R^2\) multiplies its \(m\)-th signature level by \(R^{2m}\). Hence
\begin{equation}
\label{eq:integrated-landscape-signature-scaling}
S_m\!\left(
\iota_{iL}^{(N)}(D_R\alpha)
\right)
=
R^{2m}
S_m\!\left(
\iota_{iL}^{(N)}(\alpha)
\right)
\end{equation}
for every \(m\geq0\).

Set \(\alpha:=[0,1]\). Since \(\lambda_1^\alpha\) is positive on \((0,1)\), the first coordinate of \(S_1(\iota_{iL}^{(N)}(\alpha))\) equals \(\int_{\mathbb R}\lambda_1^\alpha(t)\,dt>0\), so this first signature level is nonzero. Let \(P_1\) denote projection onto the first signature level. Since \(P_1\) is bounded,
\[
\frac{
\left\|
\Phi_{iL}^{(N,M)}(D_R\alpha)
-
\Phi_{iL}^{(N,M)}(0)
\right\|
}{
W_1(D_R\alpha,0)
}
\geq
R\,
\frac{
\left\|
S_1(\iota_{iL}^{(N)}(\alpha))
\right\|
}{
\|P_1\|W_1(\alpha,0)
}
\longrightarrow\infty
\]
by Equations~\eqref{eq:integrated-landscape-wasserstein-scaling} and~\eqref{eq:integrated-landscape-signature-scaling}. Lemma~\ref{lem:no-lipschitz-extension-obstruction} shows that \(\Phi_{iL}^{(N,M)}\) is not Lipschitz with respect to \(W_1\) and has no Lipschitz extension to \(K(X,A)\).

Now suppose that \(f\) is Lipschitz, with \(L:=\operatorname{Lip}(f)\), and fix \(\alpha\in D(X,A)\). Since \(S_0=1\), Equations~\eqref{eq:integrated-landscape-signature-scaling} and~\eqref{eq:integrated-landscape-wasserstein-scaling} give, for every \(\lambda\in E^*\),
\begin{align*}
f(D_R\alpha)-f(0)
&=
\sum_{m=1}^{M}
R^{2m}
T_m\!\left(
S_m\!\left(
\iota_{iL}^{(N)}(\alpha)
\right)
\right),\\
\left|
\sum_{m=1}^{M}
R^{2m}
\lambda\!\left(
T_m\!\left(
S_m\!\left(
\iota_{iL}^{(N)}(\alpha)
\right)
\right)
\right)
\right|
&\leq
\|\lambda\|LR\,W_1(\alpha,0).
\end{align*}
If some coefficient with \(m\geq1\) is nonzero, let \(m_0\) be the largest such index. Division by \(R^{2m_0}\) and passage to \(R\to\infty\) give a nonzero limit on the left and zero on the right because \(2m_0>1\), a contradiction. Thus
\[
\lambda\!\left(
T_m\!\left(
S_m\!\left(
\iota_{iL}^{(N)}(\alpha)
\right)
\right)
\right)
=
0
\qquad
(1\leq m\leq M).
\]
Since \(E^*\) separates points, \(T_m(S_m(\iota_{iL}^{(N)}(\alpha)))=0\) for every \(m\geq1\). Hence \(f(\alpha)=T_0(1)\) for every \(\alpha\), so \(f\) is constant. Conversely, every constant function has Lipschitz constant zero.

Set \(c:=T_0(1)\) and define \(\overline f(\xi):=c\) on \(K(X,A)\). Then \(\overline f\) has Lipschitz constant zero. For every uniformly discrete subpair \(A\subseteq Z\subseteq X\), each scalarization of \(\overline f\circ i_Z\) is constant, so \(V_Z(\overline f)\) equals either \(\{0\}\) or \(\operatorname{span}_{\mathbb C}\{\mathbf1\}\). Theorem~\ref{thm:synthesis-equivalent-characterizations} gives synthesis, and \(\overline f\) has uniformly discrete spectral synthesis.
\end{proof}

\section{Derivations for the Spectral Synthesis of Learnable Vectorizations}
\label{sec:learnable-vectorization-synthesis-derivations}

We use the conventions and general results from Section~\ref{sec:vectorization-synthesis-derivations}.

\subsection{Persistence Codebooks}
\label{subsec:persistence-codebooks}

We classify the selected persistence codebook vectorizations below using the formulas of \cite{zielinski2021persistence}. For an off-diagonal interval \(I=[p,q]\), let \(\xi(I):=(p,q-p)\) denote its birth-persistence coordinate. For \(b\in\mathbb R\) and \(n\geq1\), set \(I_n(b):=[b,b+n^{-1}]\), and let \(\alpha_n(b)\) be the singleton persistence diagram consisting of \(I_n(b)\). Then \(W_1(\alpha_n(b),0)=n^{-1}\) and \(\xi(I_n(b))\to(b,0)\). We use the Euclidean norm on the finite-dimensional target spaces below.

In the stability estimates of \cite[Sections~3.4, 3.6--3.7]{zielinski2021persistence}, the source compares the codebook formulas at locations paired by matchings between persistence diagrams. In \(D(X,A)=D(X)/D(A)\), a match to \(A\) removes the corresponding generator. The results below classify the source formulas under this quotient convention.

For a finite codebook \(\mu_1,\ldots,\mu_N\in\mathbb R^2\), let \(NN(z)\) denote a selected index of a codeword nearest to \(z\), with a fixed choice when several codewords are nearest, and set \(R_i:=\{z\in\mathbb R^2:NN(z)=i\}\). For a nonempty persistence diagram \(B\), \cite[Equation~(1)]{zielinski2021persistence} first forms the persistence bag of words (PBoW) count vector \(c(B)=(c_i(B))_{i=1}^N\), where \(c_i(B)\) counts the points \(z\in B\) satisfying \(NN(z)=i\), and then applies the square-root and vector normalization specified in \cite[Section~3.1]{zielinski2021persistence}. We denote the resulting normalized PBoW by \(F_{\mathrm{PBoW}}\).

\begin{proposition}\label{prop:pbow-no-lipschitz-subclass}
Suppose that there exist distinct \(i,j\in\{1,\ldots,N\}\) and an off-diagonal point \(z\in\overline{R_i}\cap\overline{R_j}\). Then the normalized PBoW of \cite[Section~3.1]{zielinski2021persistence} is not Lipschitz with respect to \(W_1\) on the nonempty persistence diagrams. No extension of this PBoW to \(D(X,A)\) is Lipschitz with respect to \(W_1\), and no such extension has a Lipschitz extension to \(K(X,A)\).
\end{proposition}

\begin{proof}
Choose sequences \(z_n^{(i)}\in R_i\) and \(z_n^{(j)}\in R_j\) such that \(z_n^{(i)}\to z\) and \(z_n^{(j)}\to z\). Since \(z\) is off diagonal, after discarding finitely many terms we may assume that all points in these sequences are off diagonal. Let \(D_n^{(i)}\) and \(D_n^{(j)}\) be the singleton persistence diagrams whose birth-persistence coordinates are \(z_n^{(i)}\) and \(z_n^{(j)}\), respectively. The corresponding interval endpoints converge to the same off-diagonal interval, so \(W_1(D_n^{(i)},D_n^{(j)})\to0\).

For every \(n\), the normalized PBoW vectors of \(D_n^{(i)}\) and \(D_n^{(j)}\) are \(e_i\) and \(e_j\), respectively. Hence
\[
\frac{
\|F_{\mathrm{PBoW}}(D_n^{(i)})-F_{\mathrm{PBoW}}(D_n^{(j)})\|
}{
W_1(D_n^{(i)},D_n^{(j)})
}
=
\frac{\|e_i-e_j\|}{W_1(D_n^{(i)},D_n^{(j)})}
\longrightarrow\infty.
\]
Therefore \(F_{\mathrm{PBoW}}\) is not Lipschitz with respect to \(W_1\) on the nonempty persistence diagrams.

Let \(\widetilde F\colon D(X,A)\to\mathbb R^N\) be any extension of \(F_{\mathrm{PBoW}}\). The same pairs \(D_n^{(i)},D_n^{(j)}\) show that \(\widetilde F\) is not Lipschitz with respect to \(W_1\). By Lemma~\ref{lem:no-lipschitz-extension-obstruction}, \(\widetilde F\) has no Lipschitz extension to \(K(X,A)\).
\end{proof}

For the Gaussian-mixture constructions, fix \(\lambda=\{(w_i,\mu_i,\Sigma_i):1\leq i\leq N\}\), where \(w_i>0\), \(\sum_{i=1}^Nw_i=1\), and each \(\Sigma_i\) is positive definite. Let \(p_i(z\mid\lambda)\) denote the Gaussian likelihood from \cite[Section~3.4]{zielinski2021persistence}, and define the soft assignment \(\gamma_i(z):=w_ip_i(z\mid\lambda)/\sum_{j=1}^Nw_jp_j(z\mid\lambda)\), which we use below for the stable persistence VLAD and persistence Fisher vector constructions. Each \(p_i(\,\cdot\mid\lambda)\) and \(\gamma_i\) is continuous and strictly positive on \(\mathbb R^2\).

\begin{lemma}\label{lem:codebook-diagonal-obstruction}
Let \(E\) be a finite-dimensional Banach space and \(F\colon D(X,A)\to E\) satisfy \(F(0)=0\). Suppose that, for some \(b\in\mathbb R\) and some subsequence \((n_k)\), the vectors \(F(\alpha_{n_k}(b))\) converge to a nonzero vector \(v\in E\). Then \(F\) is not locally Lipschitz with respect to \(W_1\) at \(0\), is not Lipschitz with respect to \(W_1\), and has no Lipschitz extension to \(K(X,A)\).
\end{lemma}

\begin{proof}
Since \(W_1(\alpha_{n_k}(b),0)=n_k^{-1}\) and \(F(\alpha_{n_k}(b))\to v\neq0\),
\[
\frac{\|F(\alpha_{n_k}(b))-F(0)\|}
{W_1(\alpha_{n_k}(b),0)}
=
n_k\|F(\alpha_{n_k}(b))\|
\longrightarrow\infty.
\]
Hence \(F\) is not locally Lipschitz with respect to \(W_1\) at \(0\), and therefore \(F\) is not Lipschitz with respect to \(W_1\). By Lemma~\ref{lem:no-lipschitz-extension-obstruction}, \(F\) has no Lipschitz extension to \(K(X,A)\).
\end{proof}

\begin{proposition}\label{prop:spbow-no-lipschitz-subclass}
No stable persistence bag of words vectorization from \cite[Equation~(4)]{zielinski2021persistence} is locally Lipschitz with respect to \(W_1\) at the empty persistence diagram under the quotient convention used here. In particular, no such vectorization is Lipschitz with respect to \(W_1\) or has a Lipschitz extension to \(K(X,A)\).
\end{proposition}

\begin{proof}
Fix \(\lambda\) as above and define \(\kappa_\lambda(z):=(w_ip_i(z\mid\lambda))_{i=1}^N\). By \cite[Equation~(4)]{zielinski2021persistence}, the value of the stable persistence bag of words vectorization \(F_\lambda\) on a singleton persistence diagram with birth-persistence coordinate \(z\) is \(\kappa_\lambda(z)\), while \(F_\lambda(0)=0\).

Fix \(b\in\mathbb R\). The vector \(\kappa_\lambda(b,0)=(w_ip_i((b,0)\mid\lambda))_{i=1}^N\) has strictly positive coordinates and is therefore nonzero. By continuity,
\[
F_\lambda(\alpha_n(b))
=
\kappa_\lambda(\xi(I_n(b)))
\longrightarrow
\kappa_\lambda(b,0).
\]
By Lemma~\ref{lem:codebook-diagonal-obstruction}, the result follows.
\end{proof}

\begin{proposition}\label{prop:pvlad-no-lipschitz-subclass}
No persistence VLAD vectorization from \cite[Equation~(5)]{zielinski2021persistence} is locally Lipschitz with respect to \(W_1\) at the empty persistence diagram under the quotient convention used here. In particular, no such vectorization is Lipschitz with respect to \(W_1\) or has a Lipschitz extension to \(K(X,A)\).
\end{proposition}

\begin{proof}
Fix a finite codebook \(\mu_1,\ldots,\mu_N\in\mathbb R^2\) and a choice of \(NN\). By \cite[Equation~(5)]{zielinski2021persistence}, the singleton contribution at \(z\in\mathbb R^2\) is \(\kappa(z):=(\mathbf1_{\{NN(z)=j\}}(z-\mu_j))_{j=1}^N\).

Choose \(b\in\mathbb R\) such that \((b,0)\neq\mu_j\) for every \(j\). Since the indices \(NN(\xi(I_n(b)))\) take values in the finite set \(\{1,\ldots,N\}\), some index \(i\) occurs along an infinite subsequence \((n_k)\). Along this subsequence, \(F(\alpha_{n_k}(b))\) converges to the vector whose \(i\)-th block is \((b,0)-\mu_i\) and whose remaining blocks are zero. This vector is nonzero. By Lemma~\ref{lem:codebook-diagonal-obstruction}, the result follows.
\end{proof}

The stability results for stable persistence VLAD and the persistence Fisher vector in \cite[Sections~3.6--3.7]{zielinski2021persistence} additionally assume that the persistence diagrams lie in a fixed compact subset of \(\mathbb R^2\). For each sequence used below, all sufficiently large diagrams lie in any fixed compact neighborhood of the chosen limiting point. The following conclusions therefore arise from deletion at the diagonal in \(D(X,A)\) rather than escape from the compact domain.

\begin{proposition}\label{prop:spvlad-no-lipschitz-subclass}
No stable persistence VLAD vectorization from \cite[Equation~(6)]{zielinski2021persistence} is locally Lipschitz with respect to \(W_1\) at the empty persistence diagram under the quotient convention used here. In particular, no such vectorization is Lipschitz with respect to \(W_1\) or has a Lipschitz extension to \(K(X,A)\).
\end{proposition}

\begin{proof}
Fix \(\lambda\) as above and define \(\kappa_\lambda(z):=(\gamma_i(z)(z-\mu_i))_{i=1}^N\). By \cite[Equation~(6)]{zielinski2021persistence}, the value of the stable persistence VLAD vectorization \(F_\lambda\) on a singleton persistence diagram with birth-persistence coordinate \(z\) is \(\kappa_\lambda(z)\), while \(F_\lambda(0)=0\).

Choose \(b\in\mathbb R\) such that \((b,0)\neq\mu_i\) for every \(i\). Since every \(\gamma_i(b,0)\) is strictly positive, \(\kappa_\lambda(b,0)\neq0\). By continuity,
\[
F_\lambda(\alpha_n(b))
=
\kappa_\lambda(\xi(I_n(b)))
\longrightarrow
\kappa_\lambda(b,0).
\]
By Lemma~\ref{lem:codebook-diagonal-obstruction}, the result follows.
\end{proof}

\begin{proposition}\label{prop:pfv-no-lipschitz-subclass}
Consider the diagonal-covariance persistence Fisher vectors from \cite[Equations~(9)--(11)]{zielinski2021persistence} for which every diagonal Fisher-information entry used in \cite[Equation~(11)]{zielinski2021persistence} is finite and strictly positive. No such persistence Fisher vectorization is locally Lipschitz with respect to \(W_1\) at the empty persistence diagram under the quotient convention used here. In particular, no such vectorization is Lipschitz with respect to \(W_1\) or has a Lipschitz extension to \(K(X,A)\).
\end{proposition}

\begin{proof}
Fix such a persistence Fisher vector and its Gaussian mixture model \(\lambda=\{(w_i,\mu_i,\Sigma_i):1\leq i\leq N\}\). For \(1\leq i\leq N\) and \(d\in\{1,2\}\), let \(\sigma_i^d\) denote the \(d\)-th component of \(\operatorname{diag}(\Sigma_i)\), and let \(f_{\mu_i^d}\) and \(f_{\sigma_i^d}\) denote the corresponding diagonal Fisher-information entries from \cite[Equation~(11)]{zielinski2021persistence}. For \(z\in\mathbb R^2\), define
\begin{align*}
\kappa_{i,d}^{\mu}(z)
&:=
f_{\mu_i^d}^{-1/2}
\gamma_i(z)
\frac{z^d-\mu_i^d}{(\sigma_i^d)^2},\\
\kappa_{i,d}^{\sigma}(z)
&:=
f_{\sigma_i^d}^{-1/2}
\gamma_i(z)
\left(
\frac{(z^d-\mu_i^d)^2}{(\sigma_i^d)^3}
-
\frac{1}{\sigma_i^d}
\right).
\end{align*}
Let \(\kappa(z)\) be the concatenation of these coordinates. By \cite[Equations~(9)--(11)]{zielinski2021persistence}, a singleton persistence diagram with birth-persistence coordinate \(z\) has persistence Fisher vector \(\kappa(z)\), while the empty persistence diagram has vector \(0\).

Fix \(i\in\{1,\ldots,N\}\), \(d\in\{1,2\}\), and \(b\in\mathbb R\), and set \(a:=(b,0)\). If \(a^d\neq\mu_i^d\), then \(\kappa_{i,d}^{\mu}(a)\neq0\). If \(a^d=\mu_i^d\), then \(\kappa_{i,d}^{\sigma}(a)=-f_{\sigma_i^d}^{-1/2}\gamma_i(a)/\sigma_i^d\neq0\). Hence \(\kappa(a)\neq0\). By continuity,
\[
F(\alpha_n(b))
=
\kappa(\xi(I_n(b)))
\longrightarrow
\kappa(a).
\]
By Lemma~\ref{lem:codebook-diagonal-obstruction}, the result follows.
\end{proof}

\subsection{ATOL}
\label{subsec:atol}

Royer et al.~\cite[Section~2.1]{royer2021atol} represent persistence diagrams as finite measures. \(W_1\) permits matching points to \(A\), so we classify the ATOL contrast functions with respect to this metric.

\begin{figure}[t]
    \centering
    \includegraphics[width=0.72\textwidth]{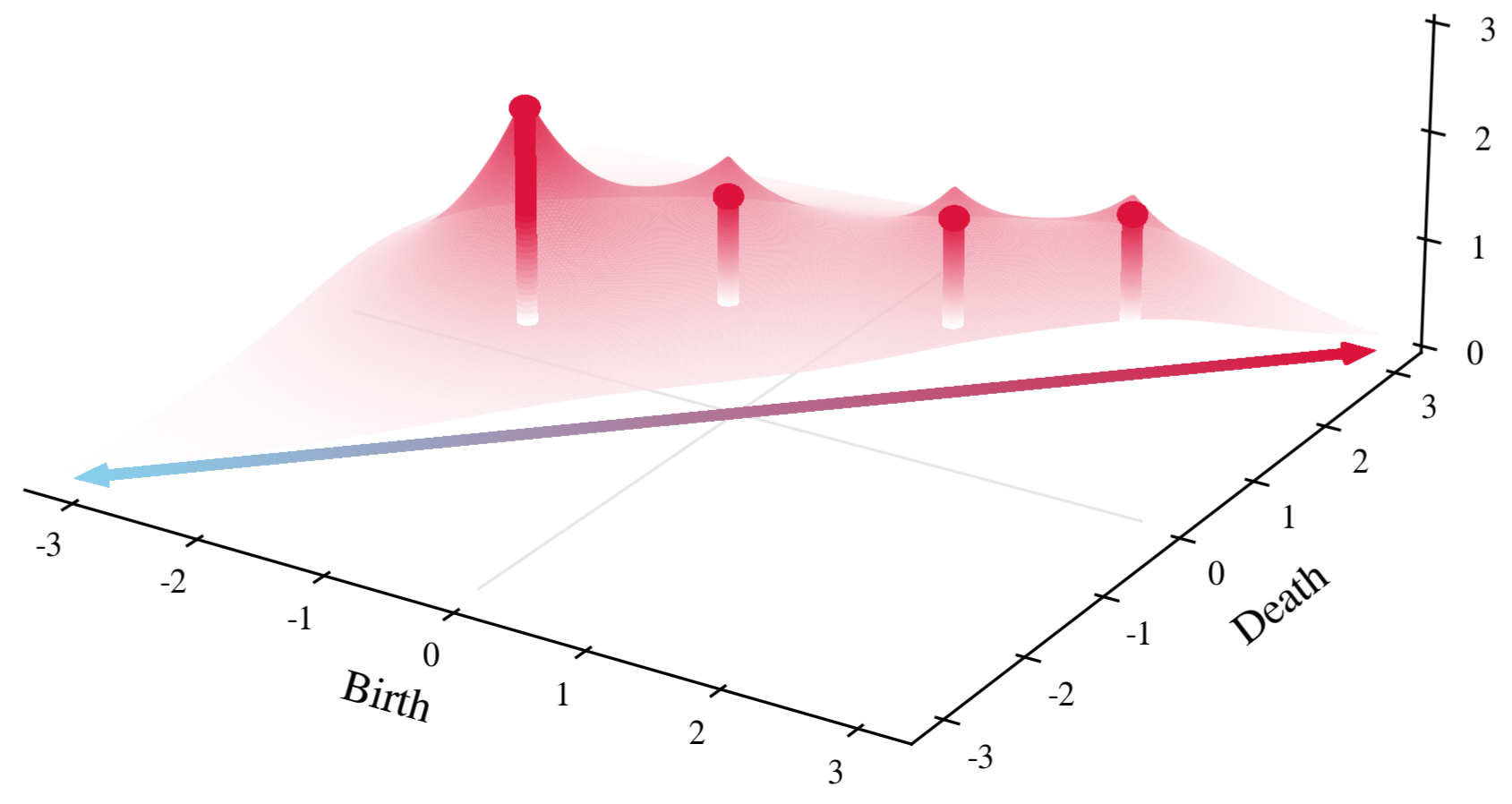}
    \caption{ATOL vectorization of a persistence diagram.}
    \label{fig:atol}
\end{figure}

\begin{proposition}\label{prop:atol-no-lipschitz-subclass}
Every ATOL vectorization from \cite[Section~2.1]{royer2021atol} satisfying \(\sigma_i(\widehat c_n)>0\) for every \(i\in\{1,\ldots,b\}\) is not Lipschitz with respect to \(W_1\) and has no Lipschitz extension to \(K(X,A)\).
\end{proposition}

\begin{proof}
Fix an ATOL vectorization \(v_{\mathrm{Atol}}^{\widehat c_n}\) with learned codebook \(\widehat c_n=(c_1,\ldots,c_b)\) such that \(\sigma_i(\widehat c_n)>0\) for every \(i\in\{1,\ldots,b\}\). Fix \(i\in\{1,\ldots,b\}\). For \(\varepsilon>0\), let \(\alpha_\varepsilon\) denote the singleton persistence diagram \([0,\varepsilon]\). Then \(W_1(\alpha_\varepsilon,0)=d([0,\varepsilon],A)=\varepsilon\).

By \cite[Equations~(1)--(2)]{royer2021atol},
\[
v_{\mathrm{Atol}}^{\widehat c_n}(\alpha_\varepsilon)_i
=
\exp\left(
-\frac{\|(0,\varepsilon)-c_i\|_2}{\sigma_i(\widehat c_n)}
\right).
\]
As \(\varepsilon\downarrow0\), this coordinate converges to
\[
\exp\left(
-\frac{\|c_i\|_2}{\sigma_i(\widehat c_n)}
\right)
>
0.
\]
Hence there exist \(C_0>0\) and \(\varepsilon_0>0\) such that \(\lvert v_{\mathrm{Atol}}^{\widehat c_n}(\alpha_\varepsilon)_i\rvert\geq C_0\) whenever \(0<\varepsilon<\varepsilon_0\). Since \(v_{\mathrm{Atol}}^{\widehat c_n}(0)=0\), we have
\[
\frac{
\|v_{\mathrm{Atol}}^{\widehat c_n}(\alpha_\varepsilon)
-
v_{\mathrm{Atol}}^{\widehat c_n}(0)\|_\infty
}{
W_1(\alpha_\varepsilon,0)
}
\geq
\frac{C_0}{\varepsilon}
\longrightarrow\infty.
\]
By Lemma~\ref{lem:no-lipschitz-extension-obstruction}, \(v_{\mathrm{Atol}}^{\widehat c_n}\) is not Lipschitz with respect to \(W_1\) and has no Lipschitz extension to \(K(X,A)\). Since \(v_{\mathrm{Atol}}^{\widehat c_n}\) was arbitrary, every ATOL vectorization satisfying the stated positive-bandwidth hypothesis has these properties.
\end{proof}

\subsection{PersLay}
\label{subsec:perslay}

For the PersLay layer of \cite[Equation~(4) and Section~3.1]{Carrire2019PersLayAN}, let \(g_j:=w\phi_j\) denote the \(j\)-th weighted point-transformation coordinate. We consider the permutation invariant operations of \cite[Section~3.1]{Carrire2019PersLayAN} applied coordinatewise, before any further neural-network postcomposition. For the order-statistic operations, we adopt the following zero-padding convention in this appendix. For maximum and minimum, we adjoin a zero before taking the corresponding order statistic. For the \(k\)-th-largest operation, we adjoin \(k\) zeros before taking the \(k\)-th-largest value. Thus the maximum operation gives the largest value of \((g_j)_+\), the minimum operation gives the negative of the largest value of \((-g_j)_+\), and the \(k\)-th-largest operation gives the \(k\)-th-largest value of \((g_j)_+\) after zero padding. In particular, these operations take value zero on the empty persistence diagram. In the special case \(w=1\) and \(\phi=\phi_\Lambda\), this convention agrees with the \(k\)-th persistence landscape identified in \cite[Section~3.1]{Carrire2019PersLayAN}.

\begin{figure}[t]
    \centering
    \includegraphics[width=0.72\textwidth]{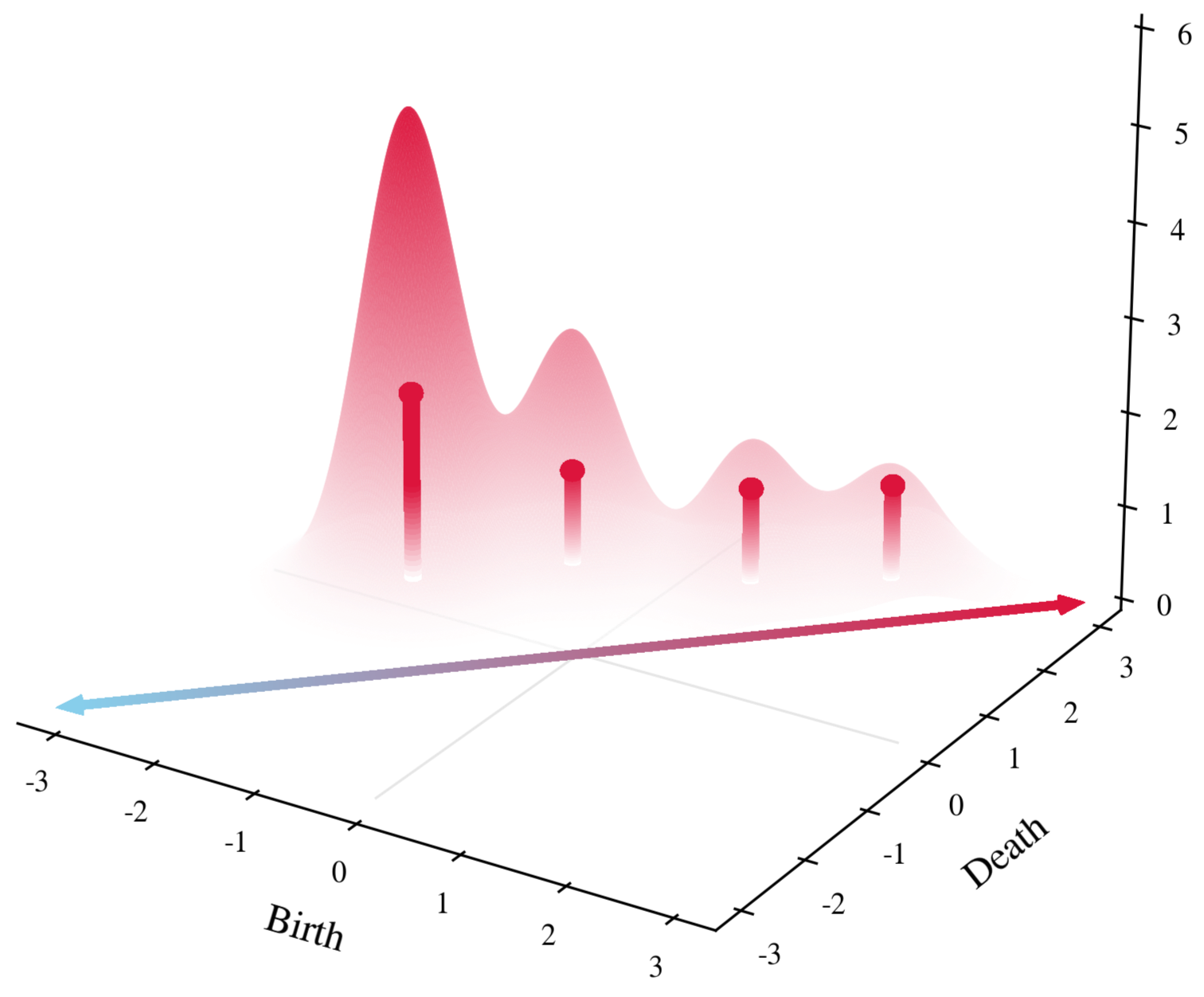}
    \caption{PersLay vectorization with sum aggregation.}
    \label{fig:perslay-sum}
\end{figure}

\begin{figure}[t]
    \centering
    \includegraphics[width=0.72\textwidth]{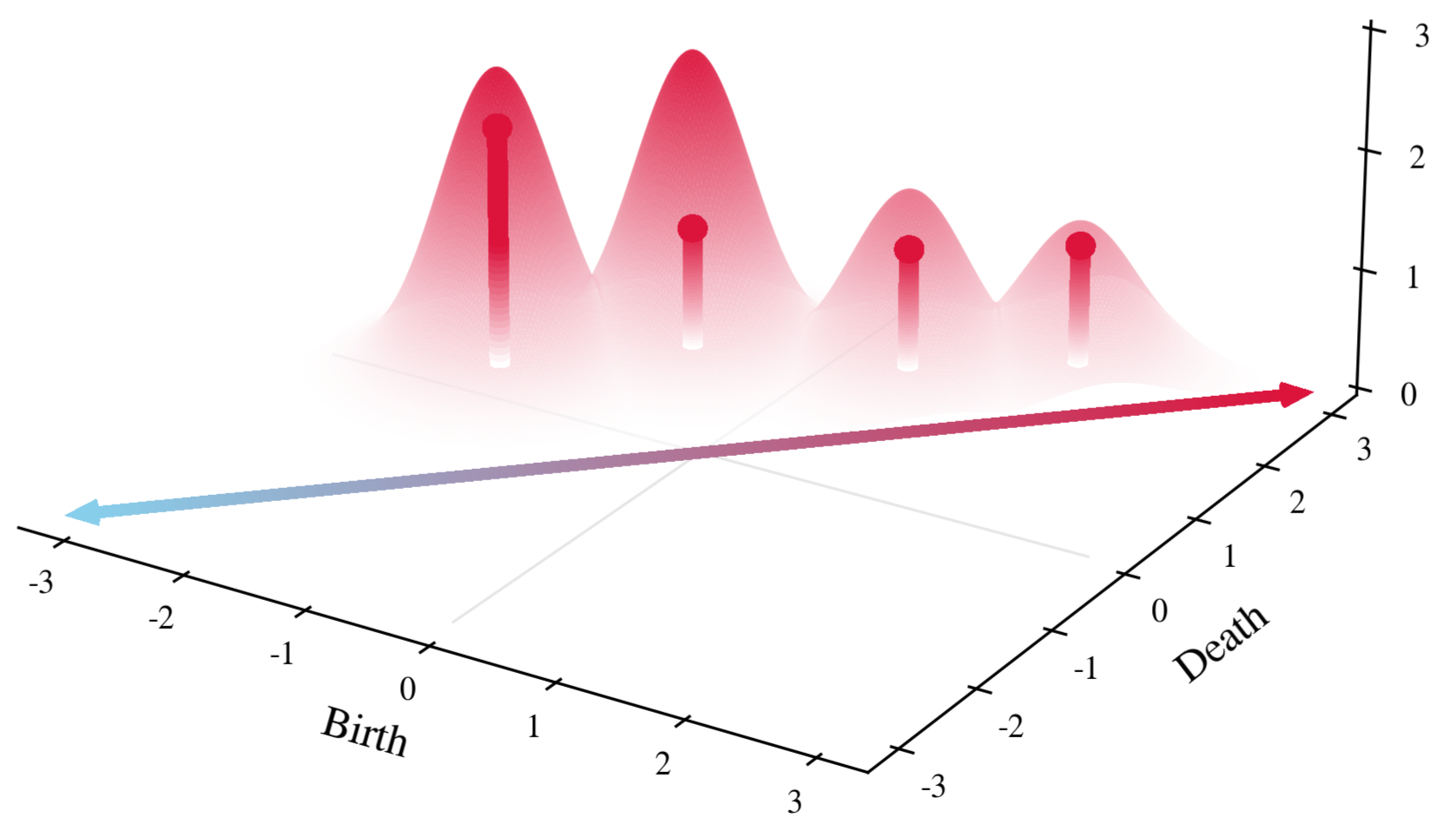}
    \caption{PersLay vectorization with maximum aggregation.}
    \label{fig:perslay-max}
\end{figure}

\begin{figure}[t]
    \centering
    \includegraphics[width=0.72\textwidth]{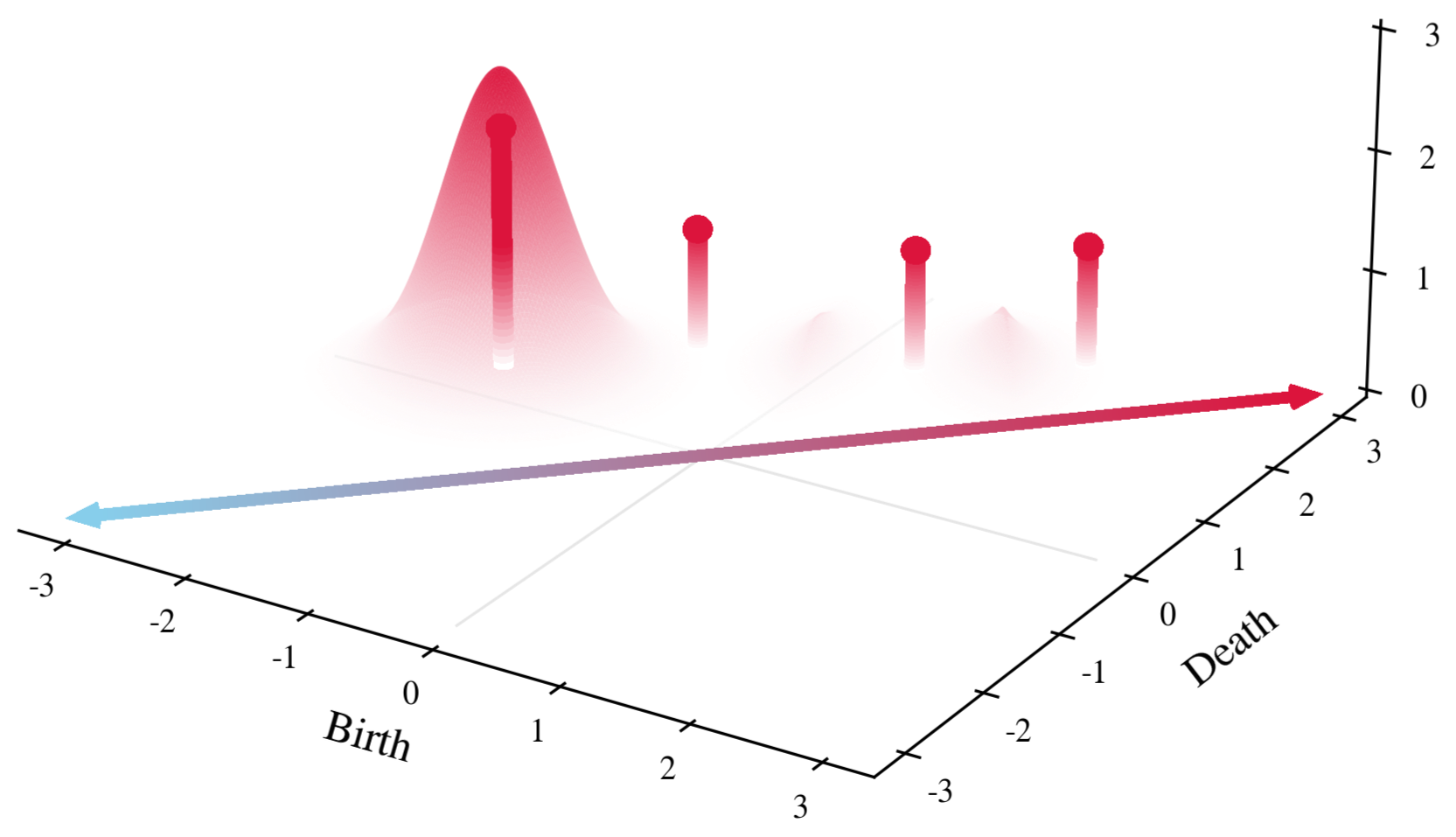}
    \caption{PersLay vectorization with second-largest aggregation.}
    \label{fig:perslay-second-largest}
\end{figure}

\begin{lemma}\label{lem:perslay-scalar-reduction}
Let \(F_j\colon D(X,A)\to\mathbb R\) be a scalar coordinate of a PersLay layer that is Lipschitz with respect to \(W_1\).
\begin{enumerate}
\item For the maximum operation, \((g_j)_+\) extends by zero on \(A\) to a function that is Lipschitz with respect to \(d_1\).
\item For the minimum operation, \((-g_j)_+\) extends by zero on \(A\) to a function that is Lipschitz with respect to \(d_1\).
\item For the \(k\)-th-largest operation, \((g_j)_+\) extends by zero on \(A\) to a function that is Lipschitz with respect to \(d_1\).
\end{enumerate}
\end{lemma}

\begin{proof}
Let \(L_j:=\operatorname{Lip}(F_j)\), and identify \(x\in X\setminus A\) with its singleton persistence diagram.

For the maximum operation, \(F_j(x)=(g_j)_+(x)\). Hence \(|(g_j)_+(x)-(g_j)_+(y)|\leq L_jd_1(x,y)\) and \((g_j)_+(x)=|F_j(x)-F_j(0)|\leq L_jd(x,A)\). Since \(d_1(x,a)=d(x,A)\) for \(a\in A\), the zero extension is Lipschitz with respect to \(d_1\).

For the minimum operation, \(F_j(x)=-(-g_j)_+(x)\). Applying the preceding argument to \(-F_j\), we obtain the second conclusion.

For the \(k\)-th-largest operation, let \(kx\) denote the diagram containing \(k\) copies of \(x\). Then \(F_j(kx)=(g_j)_+(x)\). We match corresponding copies and obtain \(W_1(kx,ky)\leq kd_1(x,y)\), so \(|(g_j)_+(x)-(g_j)_+(y)|\leq kL_jd_1(x,y)\) and \((g_j)_+(x)\leq kL_jd(x,A)\). Thus the zero extension is Lipschitz with respect to \(d_1\).
\end{proof}

\begin{lemma}\label{lem:threshold-order-statistic-extension}
Let \(h\colon X\to[0,\infty)\) be Lipschitz with respect to \(d_1\), with \(h|_A=0\), and let \(k\in\mathbb Z_{>0}\). For \(\xi=\sum_{x\in X\setminus A}n_xx\in K(X,A)\), define \(N_t(\xi):=\sum_{h(x)\geq t}n_x\) for \(t>0\). Then
\begin{equation}
\label{eq:threshold-order-statistic-extension}
\mathcal T_{k,h}(\xi)
:=
\int_0^\infty
\mathbf1_{\{N_t(\xi)\geq k\}}\,dt
\end{equation}
defines a Lipschitz extension of the zero-padded \(k\)-th-largest statistic associated with \(h\), and \(\operatorname{Lip}(\mathcal T_{k,h})\leq\operatorname{Lip}(h)\).
\end{lemma}

\begin{proof}
Each \(N_t\colon K(X,A)\to\mathbb Z\) is additive. For fixed \(\xi\), the function \(t\mapsto N_t(\xi)\) has finitely many jumps and vanishes for all sufficiently large \(t\), so Equation~\eqref{eq:threshold-order-statistic-extension} is well-defined.

For \(\alpha\in D(X,A)\), let \(a_{[k]}\) be the \(k\)-th-largest value of \(h\) on the points of \(\alpha\), counted with multiplicity and using the zero-padding convention above. Then \(N_t(\alpha)\geq k\) exactly when \(a_{[k]}\geq t\), so \(\mathcal T_{k,h}(\alpha)=a_{[k]}\).

Let \(\xi=\alpha-\beta\) and \(\eta=\gamma-\delta\). Set \(\alpha':=\alpha+\delta\) and \(\beta':=\gamma+\beta\), so \(\rho(\xi,\eta)=W_1(\alpha',\beta')\). For a matching \(\sigma=\sum_r(x_r,y_r)\) between \(\alpha'\) and \(\beta'\), by additivity and the marginal identities, we have
\[
N_t(\xi)-N_t(\eta)
=
\sum_r
\left(
\mathbf1_{\{h(x_r)\geq t\}}
-
\mathbf1_{\{h(y_r)\geq t\}}
\right).
\]
Since \(|\mathbf1_{\{a\geq k\}}-\mathbf1_{\{b\geq k\}}|\leq|a-b|\) for integers \(a,b\),
\begin{align*}
|\mathcal T_{k,h}(\xi)-\mathcal T_{k,h}(\eta)|
&\leq
\sum_r
\int_0^\infty
\left|
\mathbf1_{\{h(x_r)\geq t\}}
-
\mathbf1_{\{h(y_r)\geq t\}}
\right|dt\\
&=
\sum_r|h(x_r)-h(y_r)|\\
&\leq
\operatorname{Lip}(h)
\sum_r d_1(x_r,y_r).
\end{align*}
Taking the infimum over all matchings, we obtain the stated Lipschitz bound.
\end{proof}

\begin{lemma}\label{lem:threshold-order-statistic-synthesis}
Under the hypotheses of Lemma~\ref{lem:threshold-order-statistic-extension}, \(\mathcal T_{k,h}\) has uniformly discrete spectral synthesis.
\end{lemma}

\begin{proof}
Fix a uniformly discrete subpair \(A\subseteq Z\subseteq X\), set \(G:=K(Z,A)\), and write \(u:=\mathcal T_{k,h}\circ i_Z\) and \(V:=\operatorname{Var}_G\{u\}\). By Theorem~\ref{thm:vpd-local-compactness}, \(G\) is discrete. If \(h|_Z=0\), then \(V=\{0\}\), which is synthesizable. Assume \(h|_Z\neq0\), and identify points of \(Z\setminus A\) with their singleton persistence diagrams in \(G\).

Choose \(z\in Z\setminus A\) with \(h(z)>0\), and set \(r:=h(z)\). Set \(s:=\sup(\{0\}\cup(h(Z\setminus A)\cap[0,r)))\), and choose a nondecreasing sequence \(s_n\to s\) with \(s_n\in h(Z\setminus A)\cup\{0\}\) and \(s_n<r\). Set \(y_n:=0\) when \(s_n=0\). When \(s_n>0\), choose \(z_n\in Z\setminus A\) with \(h(z_n)=s_n\) and set \(y_n:=z_n\). Thus \(N_t(y_n)=\mathbf1_{\{t\leq s_n\}}\) for \(t>0\).

For \(s_n<t\leq r\), translation by \(z\) increases \(N_t\) by one, whereas translation by \(y_n\) leaves \(N_t\) unchanged. Therefore
\[
\frac{\tau_zu-\tau_{y_n}u}{r-s_n}(\xi)
=
\frac{1}{r-s_n}
\int_{s_n}^r
\mathbf1_{\{N_t(\xi)=k-1\}}\,dt.
\]
If \(s<r\), then \(N_t(\xi)=N_r(\xi)\) for \(s<t\leq r\). If \(s=r\), then, since \(\xi\) has finite support, the same equality holds for \(s_n<t\leq r\) once \(n\) is sufficiently large. Hence the displayed functions converge pointwise to \(\mathbf1_{\{N_r=k-1\}}\), which therefore belongs to \(V\).

Since \(N_r(\xi+mz)=N_r(\xi)+m\), by translation invariance, \(\mathbf1_{\{N_r=n\}}\in V\) for every \(n\in\mathbb Z\). Since \(N_r\) is additive, \(\chi_{r,\theta}(\xi):=e^{i\theta N_r(\xi)}\) is an exponential on \(G\). Moreover,
\[
\sum_{|n|\leq M}
e^{i\theta n}\mathbf1_{\{N_r=n\}}
\longrightarrow
\chi_{r,\theta}
\]
pointwise as \(M\to\infty\). By pointwise closedness, \(\chi_{r,\theta}\in V\) for every \(\theta\in\mathbb R\). Since the preceding argument applies to every \(z\in Z\setminus A\) with \(h(z)>0\), we have \(\chi_{r,\theta}\in V\) for every \(r\in h(Z\setminus A)\setminus\{0\}\) and every \(\theta\in\mathbb R\).

Let \(\Omega\subseteq G\) be finite, and let \(0=r_0<r_1<\cdots<r_m\) be the distinct positive values of \(h\) that occur in the supports of elements of \(\Omega\). For every \(\xi\in\Omega\),
\begin{equation}
\label{eq:finite-threshold-decomposition}
u(\xi)
=
\sum_{j=1}^m
(r_j-r_{j-1})
\mathbf1_{\{N_{r_j}(\xi)\geq k\}}.
\end{equation}
For each \(j\), the set \(\{N_{r_j}(\xi):\xi\in\Omega\}\) is finite. Choose \(\theta_j\in\mathbb R\) so that \(n\mapsto e^{i\theta_jn}\) is injective on this set. By complex polynomial interpolation, choose a polynomial \(P_j\) such that
\[
P_j\!\left(e^{i\theta_jN_{r_j}(\xi)}\right)
=
\mathbf1_{\{N_{r_j}(\xi)\geq k\}}
\qquad
(\xi\in\Omega).
\]
Each power of \(\chi_{r_j,\theta_j}\) equals \(\chi_{r_j,\ell\theta_j}\) for some \(\ell\in\mathbb Z_{\geq0}\), and the preceding argument shows that these exponentials belong to \(V\). Hence \(P_j(\chi_{r_j,\theta_j})\in V\). By Equation~\eqref{eq:finite-threshold-decomposition}, we obtain an exponential polynomial in \(V\) that agrees with \(u\) on \(\Omega\).

For \(y\in G\), choose an exponential polynomial \(p\in V\) that agrees with \(u\) on \(\Omega+y\). Then \(\tau_yp\) is an exponential polynomial in \(V\) that agrees with \(\tau_yu\) on \(\Omega\). Hence every finite linear combination of translates of \(u\) agrees on every prescribed finite subset of \(G\) with an exponential polynomial in \(V\).

Finite linear combinations of translates of \(u\) are pointwise dense in \(V\). Therefore, for every \(f\in V\), every finite \(\Omega\subseteq G\), and every \(\varepsilon>0\), an exponential polynomial in \(V\) approximates \(f\) within \(\varepsilon\) on \(\Omega\). Thus the exponential monomials contained in \(V\) span a pointwise dense subspace of \(V\). By Theorem~\ref{thm:synthesis-equivalent-characterizations}, \(V\) is synthesizable. Since \(Z\) was arbitrary, \(\mathcal T_{k,h}\) has uniformly discrete spectral synthesis.
\end{proof}

\begin{theorem}\label{thm:perslay-synthesis}
Under the zero-padding convention above, let \(F\colon D(X,A)\to\mathbb R^q\) be a PersLay layer as in \cite[Equation~(4)]{Carrire2019PersLayAN}. Suppose that \(F\) is Lipschitz with respect to \(W_1\) and that its permutation invariant operation acts coordinatewise as sum, minimum, maximum, or \(k\)-th largest. Then \(F\) has a Lipschitz extension to \(K(X,A)\) with uniformly discrete spectral synthesis.
\end{theorem}

\begin{proof}
For the sum operation, \(F\) is additive. By Lemma~\ref{lem:additive-extension-synthesis}, \(F\) has a Lipschitz extension to \(K(X,A)\) with uniformly discrete spectral synthesis. Assume that the permutation invariant operation is maximum, minimum, or \(k\)-th largest.

Since each coordinate projection on \(\mathbb R^q\) is a bounded linear functional, every scalar coordinate \(F_j\) is Lipschitz with respect to \(W_1\). For the maximum operation, set \(m:=1\), \(\varepsilon:=1\), and \(h_j:=(g_j)_+\). For the minimum operation, set \(m:=1\), \(\varepsilon:=-1\), and \(h_j:=(-g_j)_+\). For the \(k\)-th-largest operation, set \(m:=k\), \(\varepsilon:=1\), and \(h_j:=(g_j)_+\). By Lemma~\ref{lem:perslay-scalar-reduction}, each \(h_j\) extends by zero on \(A\) to a function that is Lipschitz with respect to \(d_1\).

Define \(\overline F:=\varepsilon(\mathcal T_{m,h_1},\ldots,\mathcal T_{m,h_q})\). By Lemma~\ref{lem:threshold-order-statistic-extension}, \(\overline F\) extends \(F\). Since all norms on \(\mathbb R^q\) are equivalent, the coordinatewise Lipschitz bounds imply that \(\overline F\) is Lipschitz.

Fix a uniformly discrete subpair \(A\subseteq Z\subseteq X\). For each \(j\), set \(V_j:=\operatorname{Var}_{K(Z,A)}\{\mathcal T_{m,h_j}\circ i_Z\}\). By Lemma~\ref{lem:threshold-order-statistic-synthesis}, every \(V_j\) is synthesizable. Since the \(j\)-th coordinate of \(\overline F\circ i_Z\), up to the fixed nonzero scalar \(\varepsilon\), is a scalarization of \(\overline F\circ i_Z\), we have \(V_j\subseteq V_Z(\overline F)\) for every \(j\). Since scalarizations of \(\overline F\circ i_Z\) are precisely the finite complex linear combinations of its coordinate functions, we have \(V_Z(\overline F)=\overline{V_1+\cdots+V_q}\), where the closure uses pointwise convergence.

For each \(j\), let \(\mathcal E_j\) be the span of the exponential monomials contained in \(V_j\). Since \(V_j\) is synthesizable, \(\mathcal E_j\) is pointwise dense in \(V_j\). Hence, on every finite subset of \(K(Z,A)\), each element of \(V_1+\cdots+V_q\) can be approximated by an element of \(\mathcal E_1+\cdots+\mathcal E_q\). Since every \(V_j\) lies in \(V_Z(\overline F)\), every exponential monomial contained in any \(V_j\) also belongs to \(V_Z(\overline F)\). Therefore the exponential monomials contained in \(V_Z(\overline F)\) span a pointwise dense subspace of \(V_Z(\overline F)\). By Theorem~\ref{thm:synthesis-equivalent-characterizations}, \(V_Z(\overline F)\) is synthesizable.

Since \(Z\) was arbitrary, \(\overline F\) has uniformly discrete spectral synthesis.
\end{proof}

\section{Derivations for the Spectral Synthesis of Kernel Methods}
\label{sec:kernel-synthesis-derivations}

We use the conventions and general results from Section~\ref{sec:vectorization-synthesis-derivations}. Throughout this section, let
\[
\Omega:=\{(x_1,x_2)\in\mathbb R^2:x_2\geq x_1\},\qquad
X:=\Omega,\qquad
A:=\partial\Omega,
\]
and equip \(X\) with \(d_\infty(x,y):=\|x-y\|_\infty\). We write \(W_1\) for the metric on persistence diagrams induced by the metric pair \((X,d_\infty,A)\). This metric agrees with the standard \(1\)-Wasserstein distance on persistence diagrams, and a point of persistence \(\ell\) has distance \(\ell/2\) from \(A\).

\subsection{Sliced Wasserstein Kernel}

\begin{figure}[t]
    \centering
    \includegraphics[width=0.72\textwidth]{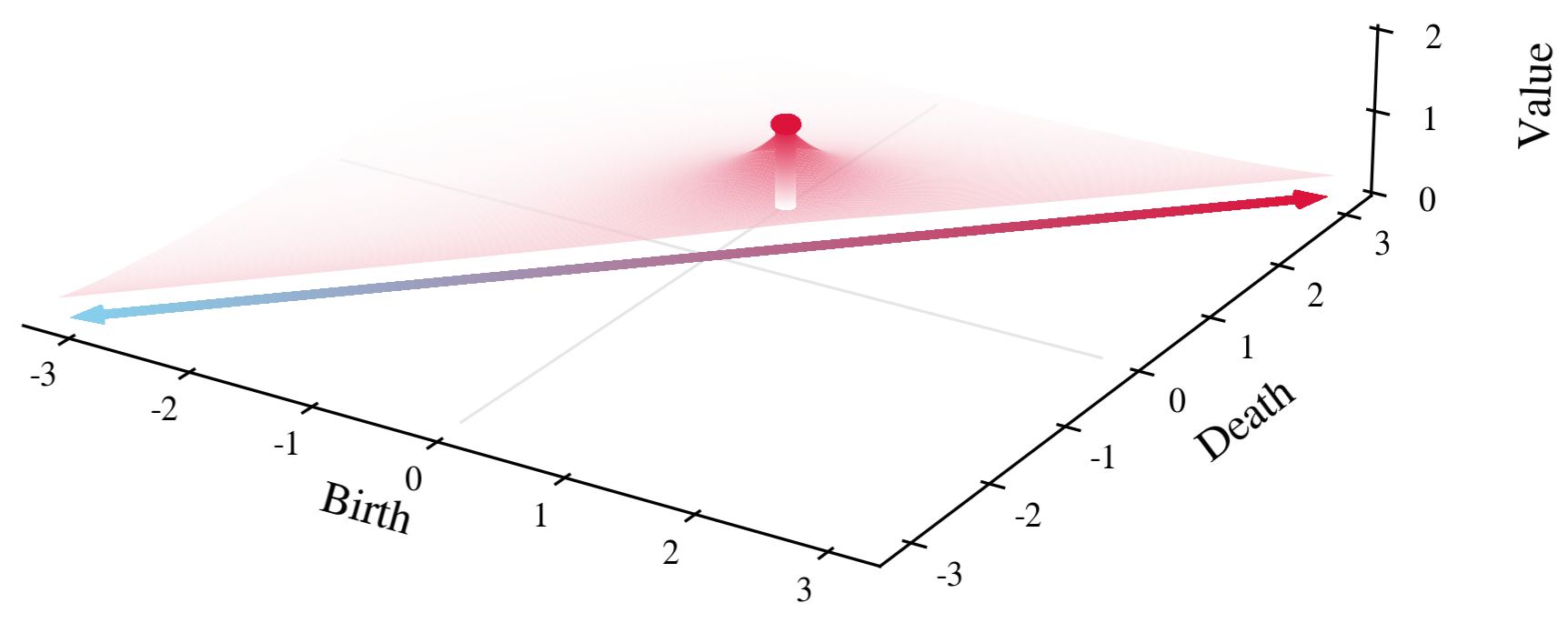}
    \caption{Sliced Wasserstein kernel section at the singleton persistence diagram $\{(0,1)\}$.}
    \label{fig:sliced-wasserstein-kernel}
\end{figure}

\begin{proposition}\label{prop:sliced-wasserstein-no-lipschitz-subclass}
For every \(\sigma>0\), the feature map into the reproducing kernel Hilbert space associated with the sliced Wasserstein kernel of \cite[Definition~3.1 and Equation~(5)]{carriere2017sliced} is not Lipschitz with respect to \(W_1\) on the persistence diagrams with finitely many points and bounded coordinates and has no Lipschitz extension to \(K(X,A)\).
\end{proposition}

\begin{proof}
Fix \(\sigma>0\), and let \(\Phi_{SW}^{\sigma}\) denote the corresponding feature map into the reproducing kernel Hilbert space. Choose \(\ell_n\downarrow0\), set \(p_n:=(0,\ell_n)\), and let \(B_n:=\{p_n\}\). Since \(\pi_{\Delta}(p_n)=(\ell_n/2,\ell_n/2)\) and \(\|p_n-\pi_{\Delta}(p_n)\|_\infty=\ell_n/2\), we have \(W_1(B_n,0)=\ell_n/2\).

Specializing \cite[Definition~3.1]{carriere2017sliced} to \(B_n\) and \(0\) gives
\[
SW(B_n,0)
=
\frac{1}{2\pi}
\int_{S^1}
W\!\left(
\delta_{\pi_{\theta}(p_n)},
\delta_{\pi_{\theta}(\pi_{\Delta}(p_n))}
\right)
\,d\theta.
\]
Under the identification \(z\mapsto\langle z,\theta\rangle\) of \(L(\theta)\) with \(\mathbb R\), the integrand equals \(\lvert\langle p_n-\pi_{\Delta}(p_n),\theta\rangle\rvert\). Rotational invariance and \(\int_0^{2\pi}\lvert\cos t\rvert\,dt=4\) therefore give
\[
SW(B_n,0)
=
\frac{2}{\pi}
\|p_n-\pi_{\Delta}(p_n)\|_2
=
\frac{\sqrt{2}}{\pi}\ell_n.
\]

By \cite[Equation~(5)]{carriere2017sliced}, \(k_{SW}^{\sigma}(D_1,D_2)=\exp(-SW(D_1,D_2)/(2\sigma^2))\). Since \(k_{SW}^{\sigma}(D,D)=1\), Theorem~\ref{thm:moore-aronszajn} gives
\[
\left\|
\Phi_{SW}^{\sigma}(B_n)
-
\Phi_{SW}^{\sigma}(0)
\right\|_{\mathcal H_{k_{SW}^{\sigma}}}^2
=
2\left(
1-
\exp\!\left(
-\frac{SW(B_n,0)}{2\sigma^2}
\right)
\right).
\]
For all sufficiently large \(n\), the inequality \(1-e^{-u}\geq u/2\) for \(0\leq u\leq1\) gives
\[
\left\|
\Phi_{SW}^{\sigma}(B_n)
-
\Phi_{SW}^{\sigma}(0)
\right\|_{\mathcal H_{k_{SW}^{\sigma}}}^2
\geq
\frac{SW(B_n,0)}{2\sigma^2}
=
\frac{\sqrt{2}}{2\pi\sigma^2}\ell_n.
\]
Hence
\[
\frac{
\left\|
\Phi_{SW}^{\sigma}(B_n)
-
\Phi_{SW}^{\sigma}(0)
\right\|_{\mathcal H_{k_{SW}^{\sigma}}}
}{
W_1(B_n,0)
}
\geq
2
\left(
\frac{\sqrt{2}}{2\pi\sigma^2}
\right)^{1/2}
\ell_n^{-1/2}
\longrightarrow
\infty.
\]
Thus \(\Phi_{SW}^{\sigma}\) is not Lipschitz with respect to \(W_1\). By Lemma~\ref{lem:no-lipschitz-extension-obstruction}, it has no Lipschitz extension to \(K(X,A)\).
\end{proof}

\subsection{Persistence Fisher Kernel}

For \(\sigma>0\), let \(d_{FIM}^{\sigma}\) denote the Fisher information metric of \cite[Definition~1]{le2018persistence}, formed from the smoothed and normalized measures of \cite[Equation~(1)]{le2018persistence} with \(\Theta=\mathbb R^2\). For \(t>0\), let
\[
k_{PF}^{t,\sigma}(D,E)
:=
\exp\left(
-t\,d_{FIM}^{\sigma}(D,E)
\right)
\]
be the Persistence Fisher kernel of \cite[Equation~(4)]{le2018persistence}. By \cite[Lemma~3.1 and Equation~(4)]{le2018persistence}, this kernel is positive definite on the nonempty persistence diagrams, so let \(\Phi_{PF}^{t,\sigma}\) denote its feature map into the reproducing kernel Hilbert space on that domain. Equation~(1) of \cite{le2018persistence} does not define the smoothed and normalized measure of the empty persistence diagram because its normalizing constant vanishes.

\begin{figure}[t]
    \centering
    \includegraphics[width=0.72\textwidth]{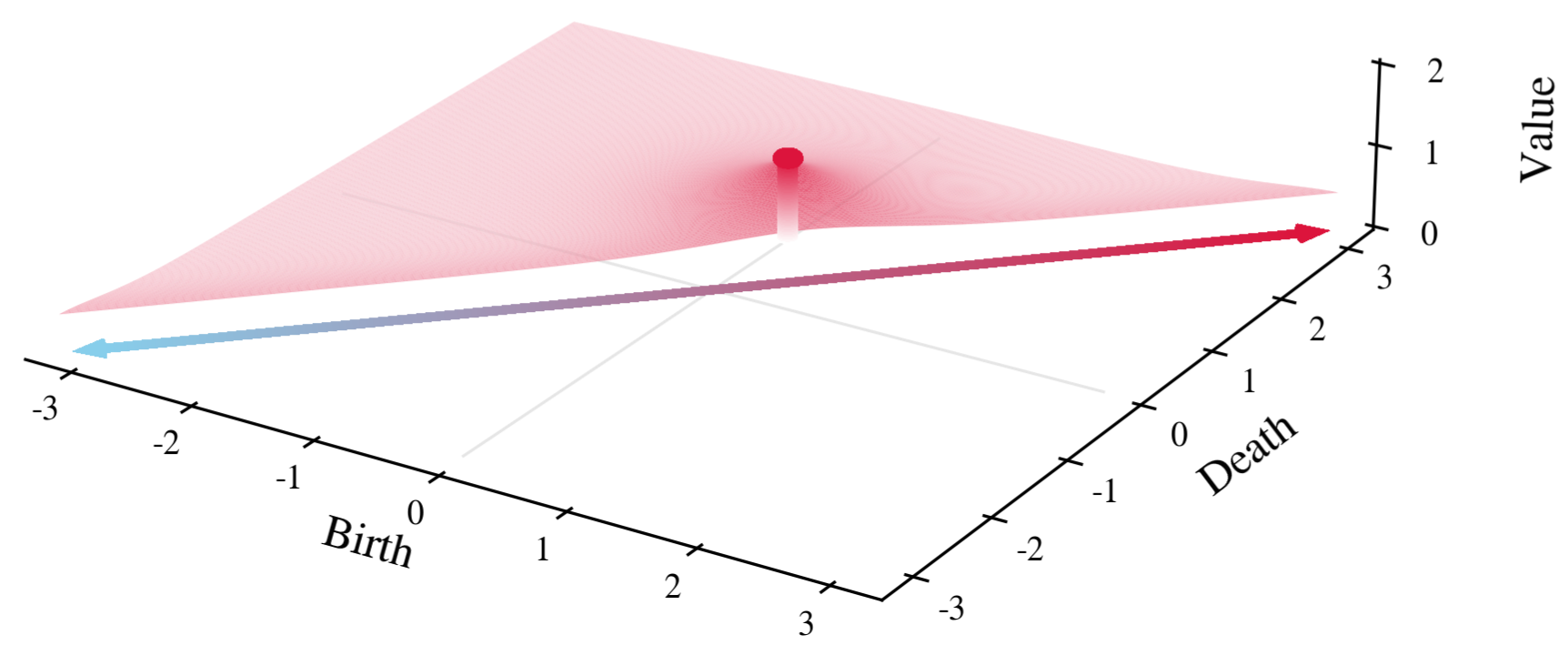}
    \caption{Persistence Fisher kernel section at the singleton persistence diagram $\{(0,1)\}$.}
    \label{fig:persistence-fisher-kernel}
\end{figure}

\begin{proposition}
\label{prop:persistence-fisher-not-locally-lipschitz}
For every \(t,\sigma>0\), the feature map into the reproducing kernel Hilbert space \(\Phi_{PF}^{t,\sigma}\) is not Lipschitz with respect to \(W_1\) on the nonempty persistence diagrams. Moreover, every extension of \(\Phi_{PF}^{t,\sigma}\) to \(D(X,A)\) is not locally Lipschitz with respect to \(W_1\) at the empty persistence diagram, and no such extension has a Lipschitz extension to \(K(X,A)\).
\end{proposition}

\begin{proof}
Fix \(t,\sigma>0\). For \(x,u\in\mathbb R^2\), set
\[
N_\sigma(x;u)
:=
\frac{1}{2\pi\sigma}
\exp\left(
-\frac{\|x-u\|_2^2}{2\sigma}
\right).
\]
For \(\ell>0\), set \(p_\ell:=(0,\ell)\), \(q_\ell:=\pi_\Delta(p_\ell)=(\ell/2,\ell/2)\), and \(B_\ell:=\{p_\ell\}\). Then \(W_1(B_\ell,B_{2\ell})=\ell\).

For \(B_\ell\) and \(B_{2\ell}\), Definition~1 and Algorithm~1 of \cite{le2018persistence} give
\[
B_\ell\cup B_{2\ell}^\Delta
=
\{p_\ell,q_{2\ell}\},
\qquad
B_{2\ell}\cup B_\ell^\Delta
=
\{p_{2\ell},q_\ell\}.
\]
Since \(\Theta=\mathbb R^2\), the corresponding smoothed and normalized measures have densities
\[
\rho_\ell
:=
\frac12
\left(
N_\sigma(\,\cdot\,;p_\ell)
+
N_\sigma(\,\cdot\,;q_{2\ell})
\right),
\qquad
\widetilde\rho_\ell
:=
\frac12
\left(
N_\sigma(\,\cdot\,;p_{2\ell})
+
N_\sigma(\,\cdot\,;q_\ell)
\right).
\]

Set \(f(x):=x_2-x_1\). For a Gaussian with covariance \(\sigma I\) and mean \(u\), we have
\[
\int_{\mathbb R^2} f(x)N_\sigma(x;u)\,dx
=
u_2-u_1
\]
and
\[
\int_{\mathbb R^2} f(x)^2N_\sigma(x;u)\,dx
=
2\sigma+(u_2-u_1)^2.
\]
Hence
\[
\int_{\mathbb R^2} f(x)\rho_\ell(x)\,dx
=
\frac{\ell}{2},
\qquad
\int_{\mathbb R^2} f(x)\widetilde\rho_\ell(x)\,dx
=
\ell,
\]
while
\[
\int_{\mathbb R^2} f(x)^2\rho_\ell(x)\,dx
=
2\sigma+\frac{\ell^2}{2},
\qquad
\int_{\mathbb R^2} f(x)^2\widetilde\rho_\ell(x)\,dx
=
2\sigma+2\ell^2.
\]

For \(0<\ell\leq1\), set
\[
H_\ell
:=
\left\|
\sqrt{\rho_\ell}
-
\sqrt{\widetilde\rho_\ell}
\right\|_{L^2(\mathbb R^2)}.
\]
Since
\[
\rho_\ell-\widetilde\rho_\ell
=
\left(
\sqrt{\rho_\ell}
-
\sqrt{\widetilde\rho_\ell}
\right)
\left(
\sqrt{\rho_\ell}
+
\sqrt{\widetilde\rho_\ell}
\right),
\]
the Cauchy--Schwarz inequality and \((\sqrt a+\sqrt b)^2\leq2(a+b)\) give
\[
\frac{\ell}{2}
\leq
H_\ell
\left(
2
\int_{\mathbb R^2}
f(x)^2
\left(
\rho_\ell(x)+\widetilde\rho_\ell(x)
\right)
dx
\right)^{1/2}
\leq
H_\ell\sqrt{8\sigma+5}.
\]
Therefore
\[
H_\ell
\geq
\frac{\ell}{2\sqrt{8\sigma+5}}.
\]

Let
\[
d_\ell
:=
d_{FIM}^{\sigma}(B_\ell,B_{2\ell}).
\]
By \cite[Definition~1]{le2018persistence},
\[
d_\ell
=
\arccos\left(
\int_{\mathbb R^2}
\sqrt{
\rho_\ell(x)\widetilde\rho_\ell(x)
}
\,dx
\right).
\]
Since
\[
H_\ell^2
=
2\left(
1-
\int_{\mathbb R^2}
\sqrt{
\rho_\ell(x)\widetilde\rho_\ell(x)
}
\,dx
\right)
=
4\sin^2\left(
\frac{d_\ell}{2}
\right),
\]
we have \(H_\ell=2\sin(d_\ell/2)\leq d_\ell\). Hence
\begin{equation}
\label{eq:persistence-fisher-two-bar-lower-bound}
d_{FIM}^{\sigma}(B_\ell,B_{2\ell})
\geq
\frac{\ell}{2\sqrt{8\sigma+5}}.
\end{equation}

Since \(k_{PF}^{t,\sigma}(D,D)=1\), the Moore--Aronszajn identity gives
\[
\left\|
\Phi_{PF}^{t,\sigma}(B_\ell)
-
\Phi_{PF}^{t,\sigma}(B_{2\ell})
\right\|_{\mathcal H_{k_{PF}^{t,\sigma}}}^2
=
2\left(
1-e^{-t d_\ell}
\right).
\]
Because \(0\leq d_\ell\leq\pi/2\), the mean value theorem gives
\[
1-e^{-t d_\ell}
\geq
t e^{-t\pi/2}d_\ell.
\]
Using Equation~\eqref{eq:persistence-fisher-two-bar-lower-bound}, we obtain
\[
\frac{
\left\|
\Phi_{PF}^{t,\sigma}(B_\ell)
-
\Phi_{PF}^{t,\sigma}(B_{2\ell})
\right\|_{\mathcal H_{k_{PF}^{t,\sigma}}}
}{
W_1(B_\ell,B_{2\ell})
}
\geq
\left(
\frac{
t e^{-t\pi/2}
}{
\sqrt{8\sigma+5}
}
\right)^{1/2}
\ell^{-1/2}
\longrightarrow
\infty
\]
as \(\ell\downarrow0\). Thus \(\Phi_{PF}^{t,\sigma}\) is not Lipschitz with respect to \(W_1\) on the nonempty persistence diagrams.

Let \(\widetilde\Phi\colon D(X,A)\to\mathcal H_{k_{PF}^{t,\sigma}}\) be any extension of \(\Phi_{PF}^{t,\sigma}\). Since \(B_\ell\to0\) and \(B_{2\ell}\to0\) in \(W_1\), the same quotient shows that \(\widetilde\Phi\) is not locally Lipschitz with respect to \(W_1\) at the empty persistence diagram. By Lemma~\ref{lem:no-lipschitz-extension-obstruction}, \(\widetilde\Phi\) has no Lipschitz extension to \(K(X,A)\).
\end{proof}

\subsection{Persistence Weighted Gaussian Kernel}

For \(C,\sigma,p>0\), let \(k_G(x,y):=\exp(-\|x-y\|_2^2/(2\sigma^2))\) and \(w_{C,p}(x):=\arctan(C\operatorname{pers}(x)^p)\). We use the RKHS embedding from \cite[Equation~(2)]{kusano2016persistence},
\[
\Phi_{PWG}^{C,\sigma,p}(\alpha)
:=
\sum_{x\in\operatorname{supp}(\alpha)}
\alpha(x)w_{C,p}(x)k_G(\cdot,x)
\in\mathcal H_{k_G}.
\]

\begin{figure}[t]
    \centering
    \includegraphics[width=0.72\textwidth]{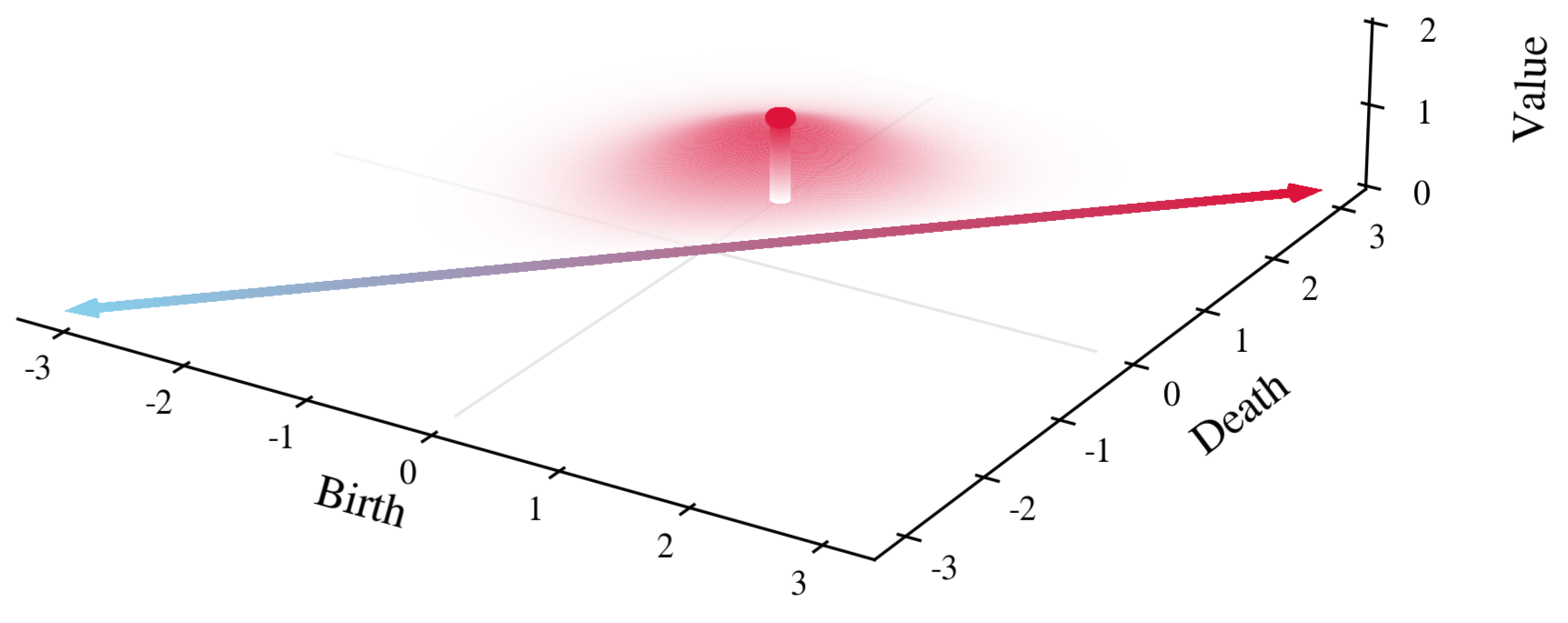}
    \caption{Persistence Weighted Gaussian kernel section at the singleton persistence diagram $\{(0,1)\}$.}
    \label{fig:persistence-weighted-gaussian-kernel}
\end{figure}

\begin{proposition}
\label{prop:pwgk-lipschitz-classification}
The map \(\Phi_{PWG}^{C,\sigma,p}\colon D(X,A)\to\mathcal H_{k_G}\) is Lipschitz with respect to \(W_1\) if and only if \(p\geq1\). If \(0<p<1\), then \(\Phi_{PWG}^{C,\sigma,p}\) is not locally Lipschitz with respect to \(W_1\) at the empty persistence diagram and has no Lipschitz extension to \(K(X,A)\).
\end{proposition}

\begin{proof}
Suppose first that \(0<p<1\). For \(\ell>0\), set \(x_\ell:=(0,\ell)\) and \(B_\ell:=\{x_\ell\}\). Then \(\operatorname{pers}(x_\ell)=\ell\) and \(W_1(B_\ell,0)=\ell/2\). Since \(k_G(x_\ell,x_\ell)=1\), we have \(\|\Phi_{PWG}^{C,\sigma,p}(B_\ell)-\Phi_{PWG}^{C,\sigma,p}(0)\|_{\mathcal H_{k_G}}=\arctan(C\ell^p)\). For all sufficiently small \(\ell>0\), \(C\ell^p\leq1\), so \(\arctan u\geq u/2\) for \(0\leq u\leq1\) gives
\[
\frac{
\left\|
\Phi_{PWG}^{C,\sigma,p}(B_\ell)
-
\Phi_{PWG}^{C,\sigma,p}(0)
\right\|_{\mathcal H_{k_G}}
}{
W_1(B_\ell,0)
}
\geq
C\ell^{p-1}
\longrightarrow
\infty
\]
as \(\ell\downarrow0\). Since \(B_\ell\to0\) in \(W_1\), the map is not locally Lipschitz with respect to \(W_1\) at the empty persistence diagram and therefore is not Lipschitz with respect to \(W_1\). By Lemma~\ref{lem:no-lipschitz-extension-obstruction}, it has no Lipschitz extension to \(K(X,A)\).

Suppose now that \(p\geq1\). Set \(g(r):=\arctan(Cr^p)\) for \(r\geq0\). For \(r>0\), \(g'(r)=Cp\,r^{p-1}/(1+C^2r^{2p})\). Since \(p\geq1\), \(g'\) is bounded near \(0\) and tends to \(0\) as \(r\to\infty\). Hence \(L_w:=\sup_{r>0}|g'(r)|<\infty\). The mean value theorem gives \(|g(r)-g(s)|\leq L_w|r-s|\) for \(r,s\geq0\), and \(g(0)=0\) gives \(g(r)\leq L_wr\).

Define
\[
\psi(x)
:=
\begin{cases}
w_{C,p}(x)k_G(\cdot,x),&x\in X\setminus A,\\
0,&x\in A.
\end{cases}
\]
For \(x\in X\setminus A\), the identity \(\operatorname{pers}(x)=2d_\infty(x,A)\) gives \(\|\psi(x)\|_{\mathcal H_{k_G}}\leq2L_wd_\infty(x,A)\).

For \(x,y\in X\setminus A\), the RKHS identity \(\|k_G(\cdot,x)-k_G(\cdot,y)\|_{\mathcal H_{k_G}}^2=2(1-k_G(x,y))\), together with \(1-e^{-u}\leq u\) and \(\|x-y\|_2\leq\sqrt2\|x-y\|_\infty\), gives \(\|k_G(\cdot,x)-k_G(\cdot,y)\|_{\mathcal H_{k_G}}\leq(\sqrt2/\sigma)\|x-y\|_\infty\). Moreover, \(|\operatorname{pers}(x)-\operatorname{pers}(y)|\leq2\|x-y\|_\infty\), so \(|w_{C,p}(x)-w_{C,p}(y)|\leq2L_w\|x-y\|_\infty\). Since \(\|k_G(\cdot,x)\|_{\mathcal H_{k_G}}=1\) and \(0\leq w_{C,p}\leq\pi/2\), we obtain
\[
\|\psi(x)-\psi(y)\|_{\mathcal H_{k_G}}
\leq
\left(
2L_w+\frac{\pi}{\sqrt2\,\sigma}
\right)
\|x-y\|_\infty.
\]
Set \(K:=2L_w+\pi/(\sqrt2\,\sigma)\). We also have \(\|\psi(x)-\psi(y)\|_{\mathcal H_{k_G}}\leq2L_w(d_\infty(x,A)+d_\infty(y,A))\). Since \(K\geq2L_w\) and \(d_1\) is the minimum of the direct \(d_\infty\) cost and the cost through \(A\), these bounds imply \(\|\psi(x)-\psi(y)\|_{\mathcal H_{k_G}}\leq Kd_1(x,y)\) for \(x,y\in X\setminus A\). If \(x\in X\setminus A\) and \(a\in A\), then \(d_1(x,a)=d_\infty(x,A)\) and \(\|\psi(x)-\psi(a)\|_{\mathcal H_{k_G}}\leq2L_wd_\infty(x,A)\leq Kd_1(x,a)\). If \(a,a'\in A\), then \(\psi(a)=\psi(a')=0\). Hence \(\psi\colon(X,d_1)\to\mathcal H_{k_G}\) is \(K\)-Lipschitz.

Let \(\alpha,\beta\in D(X,A)\), and let \(\sum_i(x_i,y_i)\) be any matching between them. Since the defining sum for \(\Phi_{PWG}^{C,\sigma,p}\) is additive over diagram points,
\[
\left\|
\Phi_{PWG}^{C,\sigma,p}(\alpha)
-
\Phi_{PWG}^{C,\sigma,p}(\beta)
\right\|_{\mathcal H_{k_G}}
\leq
\sum_i
\|\psi(x_i)-\psi(y_i)\|_{\mathcal H_{k_G}}
\leq
K\sum_i d_1(x_i,y_i).
\]
Taking the infimum over matchings gives \(\|\Phi_{PWG}^{C,\sigma,p}(\alpha)-\Phi_{PWG}^{C,\sigma,p}(\beta)\|_{\mathcal H_{k_G}}\leq KW_1(\alpha,\beta)\). Thus \(\Phi_{PWG}^{C,\sigma,p}\) is Lipschitz with respect to \(W_1\).
\end{proof}

\begin{proposition}
\label{prop:pwgk-synthesis}
For every \(C,\sigma>0\) and \(p\geq1\), the map \(\Phi_{PWG}^{C,\sigma,p}\) has a Lipschitz extension to \(K(X,A)\) with uniformly discrete spectral synthesis. Its unique group-homomorphic extension is given by \(\overline\Phi_{PWG}^{C,\sigma,p}(\alpha-\beta)=\Phi_{PWG}^{C,\sigma,p}(\alpha)-\Phi_{PWG}^{C,\sigma,p}(\beta)\).
\end{proposition}

\begin{proof}
For \(\alpha,\beta\in D(X,A)\), the defining sum satisfies \(\Phi_{PWG}^{C,\sigma,p}(\alpha+\beta)=\Phi_{PWG}^{C,\sigma,p}(\alpha)+\Phi_{PWG}^{C,\sigma,p}(\beta)\), so \(\Phi_{PWG}^{C,\sigma,p}\) is additive. By Proposition~\ref{prop:pwgk-lipschitz-classification}, \(\Phi_{PWG}^{C,\sigma,p}\) is Lipschitz with respect to \(W_1\) for \(p\geq1\). By Lemma~\ref{lem:additive-extension-synthesis}, the displayed formula defines the unique group-homomorphic extension to \(K(X,A)\), which is Lipschitz and has uniformly discrete spectral synthesis.
\end{proof}

\subsection{Persistence Scale-Space Kernel}

For the metric pair \((\Omega,d_\infty,\partial\Omega)\) fixed above, \(W_1\) agrees with \(d_{W,1}\) of \cite[Equation~(2)]{reininghaus2015stable}. A match to \(\partial\Omega\) corresponds to a match to a diagonal copy in the formulation of \cite{reininghaus2015stable}. Definition~1 of \cite{reininghaus2015stable} defines the persistence scale-space feature map. For \(p=(p_1,p_2)\in\Omega\), set \(p^\Delta:=(p_2,p_1)\), and for \(\sigma>0\), set \(N_p^\sigma(x):=(4\pi\sigma)^{-1}\exp(-\|x-p\|_2^2/(4\sigma))\). By \cite[Equation~(9)]{reininghaus2015stable}, the feature map \(\Phi_\sigma\colon D(X,A)\to L^2(\Omega)\) has the form
\[
\Phi_\sigma(\alpha)
=
\sum_{p\in\operatorname{supp}\alpha}
\alpha(p)
\left(
N_p^\sigma-N_{p^\Delta}^\sigma
\right)\big|_\Omega.
\]
If \(p\in A=\partial\Omega\), then \(p^\Delta=p\), so the corresponding summand is zero. Hence the formula descends to \(D(X,A)\).

\begin{figure}[t]
    \centering
    \includegraphics[width=0.72\textwidth]{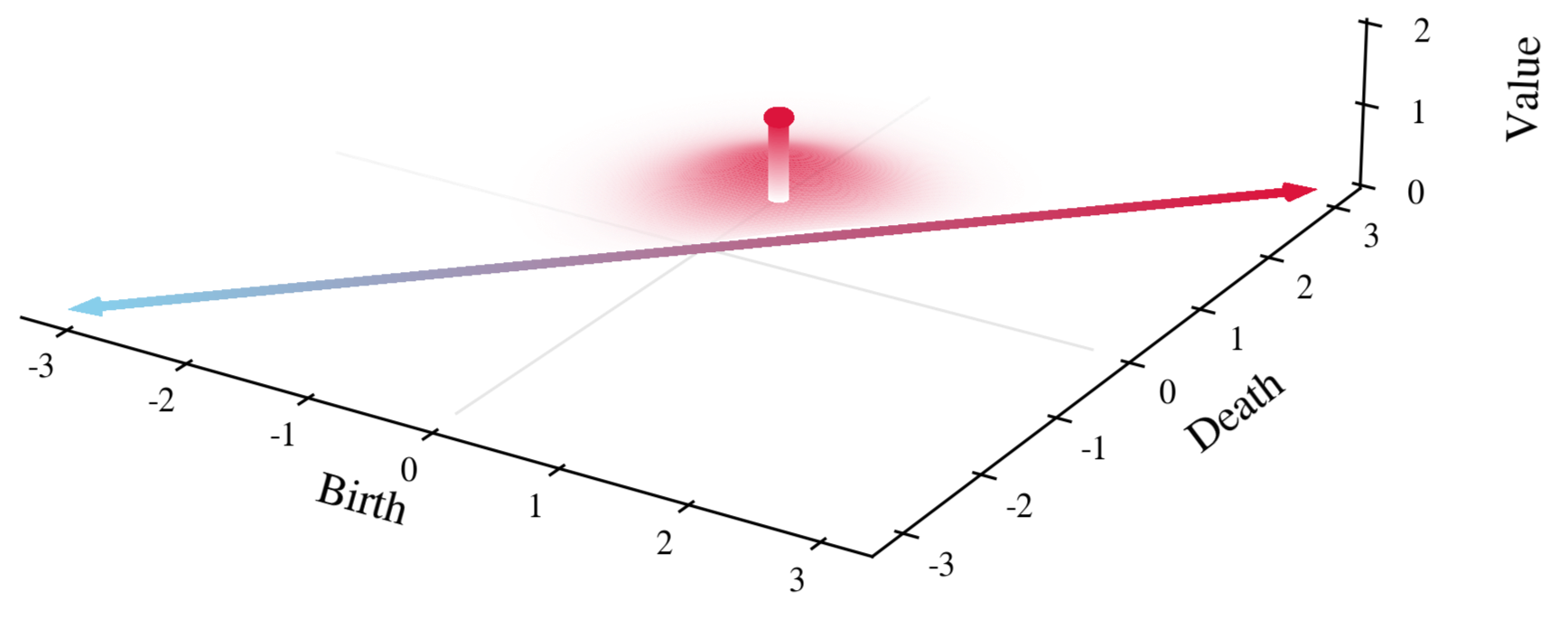}
    \caption{Persistence scale-space kernel section at the singleton persistence diagram $\{(0,1)\}$.}
    \label{fig:persistence-scale-space-kernel}
\end{figure}

\begin{proposition}\label{prop:scale-space-synthesis}
For every \(\sigma>0\), the persistence scale-space feature map \(\Phi_\sigma\) has a Lipschitz extension \(\overline\Phi_\sigma\colon K(X,A)\to L^2(\Omega)\) with uniformly discrete spectral synthesis. Its unique group-homomorphic extension satisfies \(\overline\Phi_\sigma(\alpha-\beta)=\Phi_\sigma(\alpha)-\Phi_\sigma(\beta)\).
\end{proposition}

\begin{proof}
For \(\alpha,\beta\in D(X,A)\), the defining sum satisfies \(\Phi_\sigma(\alpha+\beta)=\Phi_\sigma(\alpha)+\Phi_\sigma(\beta)\), so \(\Phi_\sigma\) is additive. By the metric identification above and \cite[Theorem~2 and Equation~(11)]{reininghaus2015stable}, we have
\[
\|\Phi_\sigma(\alpha)-\Phi_\sigma(\beta)\|_{L^2(\Omega)}
\leq
\frac{1}{2\sqrt{\pi\sigma}}
W_1(\alpha,\beta)
\]
for all \(\alpha,\beta\in D(X,A)\). By Lemma~\ref{lem:additive-extension-synthesis}, the formula \(\overline\Phi_\sigma(\alpha-\beta)=\Phi_\sigma(\alpha)-\Phi_\sigma(\beta)\) defines the unique group-homomorphic extension to \(K(X,A)\); this extension is Lipschitz and has uniformly discrete spectral synthesis.
\end{proof}

\end{appendices}

\end{document}